\documentclass[12pt]{amsart}
\usepackage{amsmath,amssymb,amsthm,bm,bbm}
\usepackage{graphicx,enumerate}
\usepackage{layout}
\usepackage[all,cmtip]{xy}

\theoremstyle{plain}
\newtheorem{theorem}{Theorem}[section]
\newtheorem{lemma}[theorem]{Lemma}
\newtheorem{proposition}[theorem]{Proposition}
\newtheorem{corollary}[theorem]{Corollary}

\newtheorem{question}[theorem]{Question}

\theoremstyle{definition}
\newtheorem{definition}[theorem]{Definition}
\newtheorem*{notation}{Standing Notations}

\newtheorem{remark}[theorem]{Remark}

\newtheorem*{acknowledgments}{Acknowledgments}

\newcommand{\bZ}{\mathbbm{Z}}\newcommand{\bQ}{\mathbbm{Q}}
\newcommand{\bC}{\mathbbm{C}}
\newcommand{\z}{\zeta}

\title[{Birational classification of fields of invariants for groups of order $128$}]
{Birational classification of fields of invariants\\ for groups of order $128$}
\author{Akinari Hoshi}
\thanks{This work was partially supported by JSPS KAKENHI Grant Number 25400027. 
A part of this work was done during a visit to 
National Taiwan University,  the National Center for Theoretic Sciences 
(Taipei Office), whose support is gratefully acknowledged.}

\begin{document}
\maketitle

\begin{abstract}
Let $G$ be a finite group 
acting on the rational function field $\bC(x_g : g\in G)$ by 
$\bC$-automorphisms $h(x_g)=x_{hg}$ for any $g,h\in G$. 
Noether's problem asks whether the invariant field $\bC(G)=k(x_g : g\in G)^G$ 
is rational (i.e. purely transcendental) over $\bC$. 
By Fischer's theorem, $\bC(G)$ is rational over $\bC$ when $G$ 
is a finite abelian group. 
Saltman and Bogomolov, respectively, showed that for any prime $p$ 
there exist groups $G$ of order $p^9$ and of order $p^6$ such that 
$\bC(G)$ is not rational over $\bC$ 
by showing the non-vanishing of the unramified Brauer group: 
${\rm Br}_{\rm nr}(\bC(G))\neq 0$, 
which is an avatar of the birational invariant $H^3(X,\bZ)_{\rm tors}$ 
given by Artin and Mumford where $X$ is a smooth projective complex variety 
whose function field is $\bC(G)$. 
For $p=2$, 
Chu, Hu, Kang and Prokhorov proved that 
if $G$ is a $2$-group of order $\leq 32$, then $\bC(G)$ is rational over $\bC$. 
Chu, Hu, Kang and Kunyavskii showed that 
if $G$ is of order $64$, then $\bC(G)$ is rational over $\bC$ 
except for the groups $G$ belonging to the two isoclinism families  
$\Phi_{13}$ with ${\rm Br}_{\rm nr}(\bC(G))=0$ and 
$\Phi_{16}$ with ${\rm Br}_{\rm nr}(\bC(G))\simeq C_2$.  
Bogomolov and B\"ohning's theorem claims that if $G_1$ and $G_2$ 
belong to the same isoclinism family, then $\bC(G_1)$ and $\bC(G_2)$ 
are stably $\bC$-isomorphic. 
We investigate the birational classification of $\bC(G)$ for groups $G$ 
of order $128$ with ${\rm Br}_{\rm nr}(\bC(G))\neq 0$. 
Moravec showed that 
there exist exactly $220$ groups $G$ of order $128$ with 
${\rm Br}_{\rm nr}(\bC(G))\neq 0$ forming $11$ isoclinism families $\Phi_j$. 
We show that if $G_1$ and $G_2$ belong to 
$\Phi_{16}$, $\Phi_{31}$, $\Phi_{37}$, $\Phi_{39}$, $\Phi_{43}$, 
$\Phi_{58}$, $\Phi_{60}$ or $\Phi_{80}$ (resp. $\Phi_{106}$ or $\Phi_{114}$), 
then $\bC(G_1)$ and $\bC(G_2)$ are stably $\bC$-isomorphic 
with ${\rm Br}_{\rm nr}(\bC(G_i))\simeq C_2$. 
Explicit structures of non-rational fields $\bC(G)$ are given 
for each cases including also the case $\Phi_{30}$ with 
${\rm Br}_{\rm nr}(\bC(G))\simeq C_2\times C_2$. 
\end{abstract}

\section{Introduction}
Let $k$ be a field and 
$G$ be a finite group acting on the rational function field 
$k(x_g : g\in G)$ by $k$-automorphisms $h(x_g)=x_{hg}$ for any $g,h\in G$. 
We denote the fixed field $k(x_g : g\in G)^G$ by $k(G)$. 
Emmy Noether \cite{Noe13, Noe17} asked 
whether $k(G)$ is rational (= purely transcendental) over $k$. 
This is called Noether's problem for $G$ over $k$, 
and is related to the inverse Galois problem,
to the existence of generic $G$-Galois extensions over $k$, and
to the existence of versal $G$-torsors over $k$-rational field extensions
(see Saltman \cite{Sal82a}, Swan \cite{Swa83}, 
Manin and Tsfasman \cite{MT86}, 
Garibaldi, Merkurjev and Serre \cite[Section 33.1, page 86]{GMS03}), 
Colliot-Th\'el\`ene and Sansuc \cite{CTS07}). 

\begin{theorem}[{Fischer \cite{Fis15}, see also Swan \cite[Theorem 6.1]{Swa83}}]\label{thFis}
Let $G$ be a finite abelian group with exponent $e$. 
Assume that {\rm (i)} either {\rm char} $k=0$ or {\rm char} $k=p$ with 
$p$ ${\not |}$ $e$, and 
{\rm (ii)} $k$ contains a primitive $e$-th root of unity. 
Then $k(G)$ is $k$-rational. In particular, $\bC(G)$ is $\bC$-rational.
\end{theorem}

\begin{theorem}[Kuniyoshi \cite{Kun54}, \cite{Kun55}, \cite{Kun56}, 
see also Gasch\"utz \cite{Gas59}] \label{thKun}
Let $k$ be a field with ${\rm char}$ $k=p>0$ and $G$ be a finite
$p$-group. Then $k(G)$ is $k$-rational.
\end{theorem}

We now recall some relevant definitions of $k$-rationality of fields. 

\begin{definition}\label{defr}
Let $K/k$ and $L/k$ be finite generated extensions of fields.\\ 
{\rm (1)} $K$ is said to be {\it rational} over $k$ 
(for short, $k$-{\it rational}) if $K$ is purely transcendental over $k$, 
i.e. $K\simeq k(x_1,\ldots,x_n)$ for some algebraically independent elements 
$x_1,\ldots,x_n$ over $k$;\\
{\rm (2)} $K$ is said to be {\it stably $k$-rational} if $K(y_1,\ldots,y_m)$ is 
$k$-rational for some algebraically independent elements 
$y_1,\ldots,y_m$ over $K$;\\ 
{\rm (3)} $K$ and $L$ are said to be {\it stably $k$-isomorphic} 
if $K(y_1,\ldots,y_m)\simeq L(z_1,\ldots,z_n)$ 
for some algebraically independent elements 
$y_1,\ldots,y_m$ over $K$ and $z_1,\ldots,z_n$ over $L$;\\
{\rm (4)} (Saltman, \cite[Definition 3.1]{Sal84b}) 
$K$ is said to be 
{\it retract $k$-rational} if there exists a $k$-algebra
$A$ contained in $K$ such that (i) $K$ is the quotient field of
$A$, (ii) there exist a non-zero polynomial $f\in
k[x_1,\ldots,x_n]$ and $k$-algebra homomorphisms $\varphi\colon A\to
k[x_1,\ldots,x_n][1/f]$ and $\psi\colon k[x_1,\ldots,x_n][1/f]\to
A$ satisfying $\psi\circ\varphi =1_A$;\\ 
{\rm (5)} $K$ is said to be {\it $k$-unirational} 
if $k\subset K\subset k(x_1,\ldots,x_n)$ for some integer $n$. 
\end{definition}

In Saltman's original definition of retract $k$-rationality 
(\cite[page 130]{Sal82b}, \cite[Definition 3.1]{Sal84b}), 
a base field $k$ is required to be infinite in order to 
guarantee the existence of sufficiently many $k$-specializations. 
We now assume that $k$ is an infinite field. 
Then if $K$ and $L$ are stably $k$-isomorphic and $K$ is retract $k$-rational, 
then $L$ is also retract $k$-rational (see \cite[Proposition 3.6]{Sal84b}), 
and it is not difficult to verify the following implications:
\begin{center}
$k$-rational\ \ $\Rightarrow$\ \ 
stably $k$-rational\ \ $\Rightarrow$\ \ 
retract $k$-rational\ \ $\Rightarrow$\ \ 
$k$-unirational. 
\end{center}
Note that $k(G)$ is retract $k$-rational 
if and only if there exists a generic $G$-Galois extension over $k$ 
(see \cite[Theorem 5.3]{Sal82a}, \cite[Theorem 3.12]{Sal84b}). 
In particular, if $k$ is a Hilbertian field, e.g. number field, 
and $k(G)$ is retract $k$-rational, then inverse Galois problem 
for $G$ over $k$ has a positive answer, i.e. 
there exists a Galois extension $K/k$ with ${\rm Gal}(K/k)\simeq G$.


Swan \cite{Swa69} gave the first negative solution to Noether's problem. 
He proved that if $p=47$, $113$ or $233$, 
then  $\bQ(C_p)$ is not $\bQ$-rational, 
where $C_p$ is the cyclic group of order prime $p$, 
by using Masuda's idea of Galois descent \cite{Mas55, Mas68}. 

Noether's problem for abelian groups was studied
extensively by Masuda, Kuniyoshi, Swan, Voskresenskii, Endo and Miyata, etc. 
Eventually, Lenstra \cite{Len74} gave a necessary and sufficient condition 
to Noether's problem for finite abelian groups. 
For details, see Swan's survey paper \cite{Swa83}, 
Voskresenskii's book \cite[Section 7]{Vos98} or \cite{Hos}. 
On the other hand, just a handful of results about 
Noether's problem are obtained when the groups are non-abelian.  

\begin{theorem}[{Maeda \cite[Theorem, page 418]{Mae89}}]
Let $k$ be a field and $A_5$ be the alternating group of degree $5$. 
Then $k(A_5)$ is $k$-rational. 
\end{theorem}

\begin{theorem}[{Serre \cite[Chapter IX]{GMS03}, see also Kang \cite{Kan05}}]
Let $G$ be a group with a $2$-Sylow subgroup which is cyclic of 
order $\geq 8$ or the generalized quaternion $Q_{16}$ of order $16$. 
Then $\bQ(G)$ is not stably $\bQ$-rational. 
\end{theorem}

\begin{theorem}[{Plans \cite[Theorem 2]{Pla09}}]
Let $A_n$ be the alternating group of degree $n$. 
If $n\geq 3$ is odd integer, then 
$\bQ(A_n)$ is rational over $\bQ(A_{n-1})$. 
In particular, if $\bQ(A_{n-1})$ is $\bQ$-rational, then so is $\bQ(A_n)$. 
\end{theorem}

However, it is an open problem whether $k(A_n)$ is $k$-rational for $n\geq 6$.\\

{}From now on, 
we restrict ourselves to the case where $G$ is a $p$-group. 
By Theorem \ref{thFis} and Theorem \ref{thKun}, 
we may focus on the case where $G$ is a non-abelian $p$-group 
and $k$ is a field with char $k\neq p$. 
For $p$-groups of small order, the following results are known.

\begin{theorem}[Chu and Kang \cite{CK01}] \label{thCK01}
Let $p$ be any prime and $G$ be a $p$-group of order $\le p^4$
and of exponent $e$. If $k$ is a field containing a primitive $e$-th root
of unity, then $k(G)$ is $k$-rational.
\end{theorem}

\begin{theorem}[Chu, Hu, Kang and Prokhorov \cite{CHKP08}]\label{thCHKP08}
Let $G$ be a group of order $32$ and of exponent $e$.  
If $k$ is a field containing a primitive $e$-th root of unity, 
then $k(G)$ is $k$-rational. 
\end{theorem}

For more recent results, see e.g. 
\cite{HK10}, \cite{Kan11}, \cite{KMZ12}.


Saltman introduced a notion of retract $k$-rationality (see Definition \ref{defr}) 
and the unramified Brauer group. 
Recall that the implications for an infinite field $k$: 
$k$-rational $\Rightarrow$ stably
$k$-rational $\Rightarrow$ retract $k$-rational. 
Hence if $k(G)$ is not retract $k$-rational, then it is not $k$-rational.


\begin{definition}[{Saltman \cite[Definition 3.1]{Sal84a}, \cite[page 56]{Sal85}}] \label{defSal}
Let $K/k$ be an extension of fields. 
The {\it unramified Brauer group} ${\rm Br}_{\rm nr}(K/k)$ 
of $K$ over $k$ is defined to be 
\[
{\rm Br}_{\rm nr}(K/k)=\bigcap_R {\rm Image} \{ {\rm Br}(R)\to{\rm Br}(K)\}
\]
where ${\rm Br}(R)\to {\rm Br}(K)$ is the natural map of
Brauer groups and $R$ runs over all the discrete valuation rings $R$ 
such that $k\subset R\subset K$ and $K$ is the quotient field of $R$. 
We omit $k$ from the notation 
and write just ${\rm Br}_{\rm nr}(K)$ when the base field $k$ is clear 
from the context.
\end{definition}

\begin{proposition}[{Saltman \cite{Sal84a}, \cite[Proposition 1.8]{Sal85}, \cite{Sal87}}] \label{propSal}
If $K$ is retract $k$-rational, then
${\rm Br}(k)\overset{\sim}{\longrightarrow}{\rm Br}_{\rm nr}(K)$. 
In particular, if $k$ is an algebraically closed field and $K$ is
retract $k$-rational, then ${\rm Br}_{\rm nr}(K)=0$.
\end{proposition}

\begin{theorem}[{Bogomolov \cite[Theorem 3.1]{Bog88}, Saltman \cite[Theorem 12]{Sal90}}]\label{thBog}
Let $G$ be a finite group and $k$ be an algebraically closed field with 
{\rm char} $k=0$ or {\rm char} $k=p$ with $p {\not |}$ $|G|$. 
Then ${\rm Br}_{\rm nr}(k(G)/k)$ is isomorphic to the group $B_0(G)$ defined by
\[
B_0(G)=\bigcap_A {\rm Ker}\{{\rm res}: H^2(G,\bQ/\bZ)\to H^2(A,\bQ/\bZ)\}
\]
where $A$ runs over all the bicyclic subgroups of $G$ $($a group $A$
is called bicyclic if $A$ is either a cyclic group or a direct
product of two cyclic groups$)$.
\end{theorem}

\begin{remark}
For a smooth projective variety $X$ over $\bC$ with function 
field $K$, ${\rm Br}_{\rm nr}(K/\bC)$ is isomorphic to the birational invariant 
$H^3(X,\bZ)_{\rm tors}$ which was used by Artin and Mumford \cite{AM72} 
to provide some elementary examples of $k$-unirational varieties 
which are not $k$-rational (see \cite[Theorem 1.1 and Corollary]{Bog88}). 
\end{remark}

Following Kunyavskii \cite{Kun10}, 
we call $B_0(G)$ the Bogomolov multiplier of $G$. 
Note that $B_0(G)$ is a subgroup of $H^2(G,\bQ/\bZ)$ 
which is isomorphic to the Schur multiplier 
$H_2(G,\bZ)$ of $G$ (see Karpilovsky \cite{Kar87}). 
Because of Theorem \ref{thBog}, we will not 
distinguish $B_0(G)$ and ${\rm Br}_{\rm nr}(k(G)/k)$ when $k$ is 
an algebraically closed field, and {\rm char} $k=0$ or 
{\rm char} $k=p$ with $p {\not |}$ $|G|$. 

Using the Bogomolov multiplier $B_0(G)$, Saltman and Bogomolov gave 
counter-examples to Noether's problem for non-abelian
$p$-groups over algebraically closed field.

\begin{theorem}[{Saltman \cite{Sal84a}, Bogomolov \cite{Bog88}}] \label{thSB}
Let $p$ be any prime and $k$ be any algebraically closed field
with ${\rm char}$ $k\ne p$.\\
{\rm (1) (Saltman \cite[Theorem 3.6]{Sal84a})} 
There exists a meta-abelian group $G$ of order $p^9$
such that $B_0(G)\ne 0$. 
In particular, $k(G)$ is not $($retract, stably$)$ $k$-rational;\\
{\rm (2) (Bogomolov \cite[Lemma 5.6]{Bog88})} 
There exists a group $G$ of order $p^6$ such that $B_0(G)\ne 0$. 
In particular, $k(G)$ is not $($retract, stably$)$ $k$-rational.
\end{theorem}


Colliot-Th\'el\`ene and Ojanguren \cite{CTO89} 
generalized the notion of the unramified Brauer group 
${\rm Br}_{\rm nr}(K/k)$ to the unramified cohomology 
$H_{\rm nr}^i(K/k,\mu_n^{\otimes j})$ of degree $i\geq 1$, 
that is $F_n^{i,j}(K/k)$ in \cite[Definition 1.1]{CTO89}.

\begin{definition}[{Colliot-Th\'el\`ene and Ojanguren \cite{CTO89}, see also \cite[Sections 2--4]{CT95}}] 
Let $n$ be a positive integer and $k$ be an algebraically closed field 
with {\rm char} $k=0$ or {\rm char} $k=p$ with $p {\not |}$ $n$. 
Let $K/k$ be a function field, that is finitely generated as a field over $k$. 
The {\it unramified cohomology group} $H^i_{\rm nr}(K/k,\mu_n^{\otimes j})$ 
of $K$ over $k$ of degree $i\geq 1$ is defined to be 
\[
H^i_{\rm nr}(K/k,\mu_n^{\otimes j})=\bigcap_R {\rm Image} 
\{ H^i_{\rm \acute{e}t}(R,\mu_n^{\otimes j})\to H^i_{\rm \acute{e}t}(K,\mu_n^{\otimes j})\}
\]
where $R$ runs over all the discrete valuation rings $R$ of rank one 
such that $k\subset R\subset K$ and $K$ is the quotient field of $R$.
We write just $H^i_{\rm nr}(K,\mu_n^{\otimes j})$ when the base field $k$ is clear.
\end{definition}

Note that the unramified cohomology groups of degree two are isomorphic to 
the $n$-torsion part of the unramified Brauer group: 
${}_n{\rm Br}_{\rm nr}(K/k)\simeq H_{\rm nr}^2(K/k,\mu_n)$.

\begin{proposition} 
Let $k$ be an algebraically closed field 
with {\rm char} $k=0$ or {\rm char} $k=p$ with $p {\not |}$ $n$.\\
{\rm (1) (Colliot-Th\'el\`ene and Ojanguren \cite[Proposition 1.2]{CTO89})} 
If $K$ and $L$ are stably $k$-isomorphic, then 
$H_{\rm nr}^i(K/k,\mu_n^{\otimes j}) \overset{\sim}{\longrightarrow} H_{\rm nr}^i(L/k,\mu_n^{\otimes j})$.
In particular, 
$K$ is stably $k$-rational, then $H_{\rm nr}^i(K/k,\mu_n^{\otimes j})=0$;\\
{\rm (2) (\cite[Proposition 2.15]{Mer08}, 
see also \cite[Remarque 1.2.2]{CTO89}, \cite[Sections 2--4]{CT95}, 
\cite[Example 5.9]{GS10})} 
If $K$ is retract $k$-rational, then $H_{\rm nr}^i(K/k,\mu_n^{\otimes j})=0$.
\end{proposition}

Colliot-Th\'el\`ene and Ojanguren 
\cite[Section 3]{CTO89} 
produced the first example of not stably $\bC$-rational but $\bC$-unirational field $K$ 
with $H_{\rm nr}^3(K,\mu_2^{\otimes 3})\neq 0$, 
where $K$ is the function field of 
a quadric of the type $\langle\langle f_1,f_2\rangle\rangle=\langle g_1g_2\rangle$
over the rational function field $\bC(x,y,z)$ 
with three variables $x,y,z$ 
for a $2$-fold Pfister form $\langle\langle f_1,f_2\rangle\rangle$, 
as a generalization of Artin and Mumford \cite{AM72}. 
Peyre \cite[Corollary 3]{Pey93} gave a sufficient condition for 
$H_{\rm nr}^i(K/k,\mu_p^{\otimes i})\neq 0$ and 
produced an example of the function field $K$ with 
$H_{\rm nr}^3(K/k,\mu_p^{\otimes 3})\neq 0$ and ${\rm Br}_{\rm nr}(K/k)=0$ 
using a result of Susulin \cite{Sus91} 
where $K$ is the function field of 
a product of some norm varieties 
associated to cyclic central simple algebras of degree $p$ 
(see \cite[Proposition 7]{Pey93}). 
Using a result of Jacob and Rost \cite{JR89}, 
Peyre \cite[Proposition 9]{Pey93} 
also gave an example of $H_{\rm nr}^4(K/k,\mu_2^{\otimes 4})\neq 0$ 
and ${\rm Br}_{\rm nr}(K/k)=0$ 
where $K$ is the function field of a product of quadrics associated to 
a $4$-fold Pfister form $\langle\langle a_1,a_2,a_3,a_4\rangle\rangle$ 
(see also \cite[Section 4.2]{CT95}).

Take the direct limit with respect to $n$: 
\[
H^i(K/k,\bQ/\bZ(j))=
\lim_{\overset{\longrightarrow}{n}}H^i(K/k,\mu_n^{\otimes j})
\]
and we also define the unramified cohomology group 
\[
H_{\rm nr}^i(K/k,\bQ/\bZ(j))
=\bigcap_R {\rm Image} 
\{H^i_{\rm \acute{e}t}(R,\bQ/\bZ(j))\to H^i_{\rm \acute{e}t}(K,\bQ/\bZ(j))\}.
\]
Then we have ${\rm Br}_{\rm nr}(K/k)\simeq H_{\rm nr}^2(K/k,\bQ/\bZ(1))$. 

Peyre \cite{Pey08} was able to construct an example of a field $K$, as $K=\bC(G)$, 
whose unramified Brauer group vanishes, but 
unramified cohomology of degree three does not vanish:

\begin{theorem}[{Peyre \cite[Theorem 3]{Pey08}}] \label{thPeyre}
Let $p$ be any odd prime.
There exists a $p$-group $G$ of order $p^{12}$ such that 
$B_0(G)=0$ and 
$H_{\rm nr}^3(\bC(G),\bQ/\bZ)\ne 0$.
In particular, $\bC(G)$ is not $($retract, stably$)$ $\bC$-rational.
\end{theorem}

Asok \cite{Aso13} generalized Peyre's argument \cite{Pey93} and 
established the following theorem 
for a smooth proper model $X$ (resp. a smooth projective model $Y$) 
of the function field of a product of 
quadrics of the type $\langle\langle s_1,\cdots,s_n\rangle\rangle=\langle s_n\rangle$ 
(resp. Rost varieties) 
over some rational function field over $\bC$ with many variables.

\vspace*{4mm}
\begin{theorem}[{Asok \cite{Aso13}, see also \cite[Theorem 3]{AM11} for retract $\bC$-rationality}]\label{thAsok}{}~{}\\
{\rm (1)} $($\cite[Theorem 1]{Aso13}$)$ 
For any $n>0$, 
there exists a smooth projective complex variety $X$ that is 
$\bC$-unirational, for which 
$H_{\rm nr}^i(\bC(X),\mu_2^{\otimes i})=0$ for each $i<n$, yet
$H_{\rm nr}^n(\bC(X),\mu_2^{\otimes n})\neq 0$, and so 
$X$ is not $\mathbb{A}^1$-connected, 
nor $($retract, stably$)$ $\bC$-rational;\\
{\rm (2)} $($\cite[Theorem 3]{Aso13}$)$ 
For any prime $l$ and any $n\geq 2$, 
there exists a smooth projective rationally connected complex 
variety $Y$ such that 
$H_{\rm nr}^n(\bC(Y),\mu_l^{\otimes n})\neq 0$. 
In particular, $Y$ is not $\mathbb{A}^1$-connected, 
nor $($retract, stably$)$ $\bC$-rational. 
\end{theorem}

Namely, the triviality of the unramified Brauer group or the unramified
cohomology of higher degree is just a necessary condition of
$\bC$-rationality of fields. 
It is unknown whether the vanishing of
all the unramified cohomologies is a sufficient condition for
$\bC$-rationality. 
It is interesting to consider an analog of Theorem \ref{thAsok} 
for quotient varieties $V/G$, e.g. $\bC(V_{\rm reg}/G)=\bC(G)$. \\

{\it The case where $G$ is a group of order $p^5$ $(p\geq 3)$.} 

{}From Theorem \ref{thSB} (2), 
Bogomolov \cite[Remark 1]{Bog88} raised a question 
to classify the groups of order $p^6$ with $B_0(G)\neq 0$. 
He also claimed that if $G$ is a $p$-group of order $\leq p^5$, 
then $B_0(G)=0$ (\cite[Lemma 5.6]{Bog88}). 
However, this claim was disproved by Moravec:

\begin{theorem}[{Moravec \cite[Section 8]{Mor12}}] \label{thMo}
Let $G$ be a group of order $243$. 
Then $B_0(G)\ne 0$ if and only if 
$G=G(3^5,i)$ with $28\le i\le 30$, where $G(3^5,i)$ is the $i$-th
group of order $243$ in the GAP database {\rm \cite{GAP}}. 
Moreover, if $B_0(G)\neq 0$, then $B_0(G)\simeq C_3$.
\end{theorem}


Moravec \cite{Mor12} gave a formula for $B_0(G)$ 
by using a nonabelian exterior square $G \wedge G$ of $G$ 
and an implemented algorithm {\bf b0g.g} in computer algebra system 
GAP \cite{GAP}, which is available from his website
\verb"www.fmf.uni-1j.si/~moravec/b0g.g". 
The number of all solvable groups $G$ of order $\leq 729$ apart 
from the orders $512$, $576$ and $640$ with $B_0(G)\neq 0$ 
was given as in \cite[Table 1]{Mor12}. 


Hoshi, Kang and Kunyavskii \cite{HKK13} determined $p$-groups $G$ 
of order $p^5$ with $B_0(G)\ne 0$ for any $p\geq 3$. 
It turns out that they belong to the same isoclinism family.

\begin{definition}[{Hall \cite[page 133]{Hal40}}] \label{defHall}
Let $G$ be a finite group. 
Let $Z(G)$ be the center of $G$ and 
$[G,G]$ be the commutator subgroup of $G$. 
Two $p$-groups $G_1$ and $G_2$ are called {\it isoclinic} if there exist
group isomorphisms $\theta\colon G_1/Z(G_1) \to G_2/Z(G_2)$ and
$\phi\colon [G_1,G_1]\to [G_2,G_2]$ such that $\phi([g,h])$
$=[g',h']$ for any $g,h\in G_1$ with $g'\in \theta(gZ(G_1))$, $h'\in
\theta(hZ(G_1))$: 
\[\xymatrix{
G_1/Z_1\times G_1/Z_1 \ar[d]_{[\cdot,\cdot]} \ar[r]^{(\theta,\theta)} 
\ar@{}[dr]| \circlearrowleft 
& G_2/Z_2\times G_2/Z_2 \ar[d]_{[\cdot,\cdot]} \\
[G_1, G_1] \ar[r]^\phi & [G_2, G_2].\\
}\]
For a prime $p$ and an integer $n$, 
we denote by $G_n(p)$ the set of all non-isomorphic groups of order $p^n$. 
In $G_n(p)$, consider an equivalence relation: two groups $G_1$ and $G_2$ are
equivalent if and only if they are isoclinic. 
Each equivalence class of $G_n(p)$ is called an {\it isoclinism family}, 
and the $j$-th isoclinism family is denoted by $\Phi_j$.
\end{definition}

For $p\geq 5$ (resp. $p=3$), there exist $2p+61+\gcd\{4,p-1\}+2\gcd\{3,p-1\}$ 
(resp. $67$) groups $G$ of order $p^5$ which are classified into ten 
isoclinism families $\Phi_1,\ldots,\Phi_{10}$ (see \cite[Section 4]{Jam80}). 
The main theorem of \cite{HKK13} can be stated as follows:

\begin{theorem}[{Hoshi, Kang and Kunyavskii \cite[Theorem 1.12]{HKK13}, \cite[page 424]{Kan14}}] \label{thKKu13}
Let $p$ be any odd prime and $G$ be a group of order $p^5$. Then
$B_0(G)\ne 0$ if and only if $G$ belongs to
the isoclinism family $\Phi_{10}$. 
Moreover, if $B_0(G)\neq 0$, then $B_0(G)\simeq C_p$.
\end{theorem}

For the last statement, see \cite[Remark, page 424]{Kan14}.
The proof of Theorem \ref{thKKu13} was given by purely algebraic way. 
There exist exactly $3$ groups which belong to $\Phi_{10}$ if $p=3$, i.e. 
$G=G(243,i)$ with $28\leq i\leq 30$. 
This agrees with Moravec's computational result (Theorem \ref{thMo}).
For $p\ge5$, the exist exactly 
$1+\gcd\{4,p-1\}+\gcd \{3,p-1\}$ groups which belong to $\Phi_{10}$ 
(\cite[page 621]{Jam80}). 

The following result for the $k$-rationality of $k(G)$ 
supplements Theorem \ref{thMo} although it is unknown whether $k(G)$ 
is $k$-rational for groups $G$ which belong to $\Phi_7$: 

\begin{theorem}[{Chu, Hoshi, Hu and Kang \cite[Theorem 1.13]{CHHK}}]
Let $G$ be a group of order $243$ with exponent $e$. 
If $B_0(G)=0$ and $k$ be a field containing a primitive $e$-th root of unity, 
then $k(G)$ is $k$-rational 
except possibly for the five groups $G$ which belong to $\Phi_7$, 
i.e. $G=G(243,i)$ with $56 \le i \le 60$.
\end{theorem}


In \cite{HKK13} and \cite{CHHK}, 
not only the evaluation of the Bogomolov multiplier 
$B_0(G)$ and the $k$-rationality of $k(G)$ 
but also the $k$-isomorphisms between 
$k(G_1)$ and $k(G_2)$ for some groups $G_1$ and $G_2$ 
belonging to the same isoclinism family were given. 

Bogomolov and B\"ohning \cite{BB13} 
gave an answer to the question raised as 
\cite[Question 1.11]{HKK13} in the affirmative as follows.  

\begin{theorem}[{Bogomolov and B\"ohning \cite[Theorem 6]{BB13}}]\label{thBB}
If $G_1$ and $G_2$ are isoclinic, 
then $\bC(G_1)$ and $\bC(G_2)$ are stably $\bC$-isomorphic. 
In particular, $H_{\rm nr}^i(\bC(G_1),\mu_n^{\otimes j})$ 
$\overset{\sim}{\longrightarrow}$ $H_{\rm nr}^i(\bC(G_2),\mu_n^{\otimes j})$. 
\end{theorem}

A partial result of Theorem \ref{thBB} was already given by Moravec. 
Indeed, 
Moravec \cite[Theorem 1.2]{Mor14} proved that 
if $G_1$ and $G_2$ are isoclinic, then $B_0(G_1)\simeq B_0(G_2)$.\\

{\it The case where $G$ is a group of order $64$.} 

The classification of the groups $G$ of order $p^6$ with $B_0(G)\neq 0$ 
for $p=2$ was obtained by Chu, Hu, Kang and Kunyavskii 
\cite{CHKK10}.
Moreover, they investigated Noether's problem for groups $G$ 
with $B_0(G)=0$. 
There exist $267$ groups $G$ of order $64$ which are classified into $27$ 
isoclinism families $\Phi_1,\ldots,\Phi_{27}$ 
by Hall and Senior \cite{HS64} (see also \cite[Table I]{JNO90}). 
The main result of \cite{CHKK10} 
can be stated in terms of the isoclinism families as follows.

\begin{theorem}[{Chu, Hu, Kang and Kunyavskii \cite{CHKK10}}]\label{thCHKK10}
Let $G=G(2^6,i)$, $1\leq i\leq 267$, be the $i$-th group of order $64$ 
in the GAP database {\rm \cite{GAP}}.\\
{\rm (1) (\cite[Theorem 1.8]{CHKK10})} 
$B_0(G)\ne 0$ if and only if $G$ belongs to the isoclinism family 
$\Phi_{16}$, i.e. $G=G(2^6,i)$ with $149\le i\le 151$,
$170\le i\le 172$, $177\le i\le 178$ or $i=182$. 
Moreover, if $B_0(G)\neq 0$, then $B_0(G)\simeq C_2$ 
$($see \cite[Remark, page 424]{Kan14} for this statement$)$;\\
{\rm (2) (\cite[Theorem 1.10]{CHKK10})} 
If $B_0(G)= 0$ and $k$ is an quadratically closed field, 
then $k(G)$ is $k$-rational except
possibly for five groups which belong to $\Phi_{13}$, i.e. 
$G=G(2^6,i)$ with $241\le i\le 245$.
\end{theorem}

For groups $G$ which belong to $\Phi_{13}$, 
$k$-rationality of $k(G)$ is unknown. 
The following two propositions supplement the cases $\Phi_{13}$ 
and $\Phi_{16}$ of Theorem \ref{thCHKK10}. 

\begin{definition}
\label{defL01}
Let $k$ be a field with {\rm char} $k\neq 2$ and 
$k(X_1,X_2,X_3,X_4,X_5,X_6)$ be the rational function field over $k$ 
with variables $X_1,X_2,X_3,X_4,X_5,X_6$.\\
{\rm (i)} {\it The field $L_k^{(0)}$} is defined to be $k(X_1,X_2,X_3,X_4,X_5,X_6)^H$ 
where $H=\langle \sigma_1, \sigma_2\rangle\simeq C_2\times C_2$ act on 
$k(X_1,X_2,X_3,X_4,X_5,X_6)$ by $k$-automorphisms  
\begin{align*}
\sigma_1 &: X_1\mapsto X_3,\ X_2\mapsto \frac{1}{X_1X_2X_3},\ X_3\mapsto X_1,\ 
X_4\mapsto X_6,\ X_5\mapsto \frac{1}{X_4X_5X_6},\ X_6\mapsto X_4,\\
\sigma_2 &: X_1\mapsto X_2,\ X_2\mapsto X_1,\ X_3\mapsto \frac{1}{X_1X_2X_3},\ 
X_4\mapsto X_5,\ X_5\mapsto X_4,\ X_6\mapsto \frac{1}{X_4X_5X_6}. 
\end{align*}
{\rm (ii)} {\it The field $L_k^{(1)}$} 
is defined to be $k(X_1,X_2,X_3,X_4)^{\langle\tau\rangle}$ 
where $\langle\tau\rangle \simeq C_2$ acts on 
$k(X_1,X_2,X_3,X_4)$ by $k$-automorphisms 
\begin{align*}
\tau: X_1\mapsto -X_1,\ 
X_2\mapsto \frac{X_4}{X_2},\ 
X_3\mapsto \frac{(X_4-1)(X_4-X_1^2)}{X_3},\ 
X_4\mapsto X_4.
\end{align*}
\end{definition}

\begin{proposition}[{\cite[Proposition 6.3]{CHKK10}, see also \cite[Proposition 12.5]{HY}}]\label{prop1}
Let $G$ be a group of order $64$ which belongs to $\Phi_{13}$, i.e. 
$G=G(2^6,i)$ with $241\leq i\leq 245$. 
There exists a $\bC$-injective homomorphism $\varphi : L_\bC^{(0)}\rightarrow \bC(G)$ 
such that $\bC(G)$ is rational over $\rho(L_\bC^{(0)})$. 
In particular, $\bC(G)$ and $L_\bC^{(0)}$ are stably $\bC$-isomorphic 
and $B_0(G)\simeq {\rm Br}_{\rm nr}(L_\bC^{(0)})=0$. 
\end{proposition}

\begin{proposition}[{\cite[Example 5.11, page 2355]{CHKK10}, \cite[Proof of Theorem 6.3]{HKK14}}]\label{prop2}
Let $G$ be a group of order $64$ which belongs to $\Phi_{16}$, 
i.e. $G=G(2^6,i)$ with $149\le i\le 151$,
$170\le i\le 172$, $177\le i\le 178$ or $i=182$. 
There exists a $\bC$-injective homomorphism 
$\varphi : L_\bC^{(1)}\rightarrow \bC(G)$ such that $\bC(G)$ is rational 
over $\varphi(L_\bC^{(1)})$. 
In particular, $\bC(G)$ and $L_\bC^{(1)}$ are stably $\bC$-isomorphic, 
$B_0(G)\simeq {\rm Br}_{\rm nr}(L_\bC^{(1)})\simeq C_2$ and hence $\bC(G)$ and $L_\bC^{(1)}$ 
are not $($retract, stably$)$ $\bC$-rational. 
\end{proposition}
\begin{proof}
The case of $G=G(2^6,149)$ is given in \cite[Proof of Theorem 6.3]{HKK14}, 
see also \cite[Example 5.11, page 2355]{CHKK10}.
The proof for other cases is similar. 
\end{proof}

\begin{question}[{\cite[Section 6]{CHKK10}, see also \cite[Section 12]{HY}}]
Is $L_k^{(0)}$ $k$-rational?
\end{question}

\smallskip
{\it The case where $G$ is a group of order $128$.} 

There exist $2328$ groups 
of order $128$ which are classified into $115$ 
isoclinism families $\Phi_1,\ldots,\Phi_{115}$ (\cite[Tables I, II, III]{JNO90}).
By using Moravec's algorithm {\bf b0g.g} \cite{Mor12} of GAP \cite{GAP}, 
e.g. ``\verb+for i in [1..2328] do Print([i,B0G(SmallGroup(128,i))],"\n");od;+'', 
we obtain the following theorem. 

\begin{theorem}[{Moravec \cite[Section 8, Table 1]{Mor12}}]\label{thMo128}
Let $G$ be a group of order $128$. Then 
$B_0(G)\neq 0$ if and only if $G$ is one of the following $220$ groups:\\
{\rm (1)} $G(2^7,i)$ with 
$i=227$,$228$,$229$,$301$,$324$,$325$,$326$,$541$,$543$,$568$,$570$,$579$,$581$,$626$,$627$,$629$,$667$,$668$,\\ 
$670$,$675$,$676$,$678$,$691$,$692$,$693$,$695$,$703$,$704$,$705$,$707$,$724$,$725$,$727$,$1783$,$1784$,$1785$,$1786$,$1864$,$1865$,\\ 
$1866$,$1867$,$1880$,$1881$,$1882$,$1893$,$1894$,$1903$,$1904$;\\
{\rm (2)} $G(2^7,i)$ with 
$1345\leq i\leq 1399$;\\
{\rm (3)} $G(2^7,i)$ with 
$242\leq i\leq 247$, $265\leq i\leq 269$, $287\leq i\leq 293$;\\
{\rm (4)} $G(2^7,i)$ with 
$36\leq i\leq 41$;\\
{\rm (5)} $G(2^7,i)$ with 
$1924\leq i\leq 1929$, $1945\leq i\leq 1951$, $1966\leq i\leq 1972$, 
$1983\leq i \leq 1988$;\\
{\rm (6)} $G(2^7,i)$ with 
$417\leq i\leq 436$;\\
{\rm (7)} $G(2^7,i)$ with 
$446\leq i\leq 455$;\\
{\rm (8)} $G(2^7,i)$ with 
$i=950$, $951$, $952$, $975$, $976$, $977$, $982$, $983$, $987$;\\
{\rm (9)} $G(2^7,i)$ with 
$i=144$, $145$;\\
{\rm (10)} $G(2^7,i)$ with 
$i=138$, $139$;\\
{\rm (11)} $G(2^7,i)$ with 
$1544\leq i\leq 1577$.

Moreover, if $G$ is a group in $(1)$--$(10)$ $($resp. $(11))$, 
then $B_0(G)\simeq C_2$ $($resp. $C_2\times C_2)$. 
\end{theorem}

By \cite[Tables I, II, III]{JNO90}, we can get the classification of 
$115$ isoclinism families for groups $G$ of order $128$ 
in terms of the GAP database \cite{GAP}. 
We will present the complete classification as Table 2 
in Section \ref{seTable}. 
Using this, we see that the groups 
as in (1)--(11) of Theorem \ref{thMo128} correspond to the isoclinism families 
$\Phi_{16}$, $\Phi_{31}$, $\Phi_{37}$, $\Phi_{39}$, $\Phi_{43}$, 
$\Phi_{58}$, $\Phi_{60}$, $\Phi_{80}$, $\Phi_{106}$, 
$\Phi_{114}$, $\Phi_{30}$ respectively:\\

\begin{table}[h]
\begin{tabular}{|c|cccccccccc|c|c|}\hline
& (1) & (2) & (3) & (4) & (5) & (6) & (7) & (8) & (9) & (10) & (11) & Total\\\hline
Family & $\Phi_{16}$ & $\Phi_{31}$ & $\Phi_{37}$ & $\Phi_{39}$ & $\Phi_{43}$ & 
$\Phi_{58}$ & $\Phi_{60}$ & $\Phi_{80}$ & $\Phi_{106}$ & 
$\Phi_{114}$ & $\Phi_{30}$ & \\\hline
$\exp(G)$ & $8$ & $4$ & $8$ & $4$ or $8$ & $8$ & $8$ & $8$ & $16$ & $8$ & $8$ & $4$ &\\\hline
$B_0(G)$ & & &  &  & $C_2$ & & & & & & $C_2\times C_2$ & \\\hline 
\# $G$'s & 48 & 55 & 18 & 6 & 26 & 20 & 10 & 9 & 2 & 2 & 34 & 220\\\hline
\end{tabular}\\
\vspace*{5mm}
Table $1$: Isoclinism families $\Phi_j$ for groups $G$ of order $128$ with $B_0(G)\neq 0$
\end{table}

\begin{corollary}[{Moravec \cite[Section 8, Table 1]{Mor12}}]
Let $G$ be a group of order $128$. Then 
$B_0(G)\neq 0$ if and only if 
$G$ belongs to the isoclinism family 
$\Phi_{16}$, $\Phi_{30}$, 
$\Phi_{31}$, $\Phi_{37}$, $\Phi_{39}$, $\Phi_{43}$, 
$\Phi_{58}$, $\Phi_{60}$, $\Phi_{80}$, $\Phi_{106}$ or $\Phi_{114}$. 
Moreover, if $B_0(G)\neq 0$, then 
\begin{align*}
B_0(G)\simeq\begin{cases}
C_2&{\rm if}\ G\ {\rm belongs\ to\ } \Phi_{16}, \Phi_{31}, \Phi_{37}, 
\Phi_{39}, \Phi_{43}, \Phi_{58}, \Phi_{60}, \Phi_{80}, \Phi_{106}\ {\rm or}\ \Phi_{114},\\
C_2\times C_2&{\rm if}\ G\ {\rm belongs\ to\ } \Phi_{30}.
\end{cases}
\end{align*}
In particular, $\bC(G)$ is not $($retract, stably$)$ $\bC$-rational.
\end{corollary}


The aim of this paper is to investigate the birational classification 
of $\bC(G)$ for groups $G$ of order $128$. 
In particular, we consider what happens to $\bC(G)$ with $B_0(G)\neq 0$?
The main theorem (Theorem \ref{thm}) of this paper gives a partial 
answer to this question. 

\begin{definition}
\label{defL23}
Let $k$ be a field with {\rm char} $k\neq 2$ and 
$k(X_1,X_2,X_3,X_4,X_5,X_6,X_7)$ be the rational function field 
over $k$ with variables $X_1,X_2,X_3,X_4,X_5,X_6,X_7$.\\
{\rm (i)} {\it The field $L_k^{(2)}$} is defined to be 
$k(X_1,X_2,X_3,X_4,X_5,X_6)^{\langle\rho\rangle}$ 
where $\langle\rho\rangle\simeq C_4$ acts on\\ 
$k(X_1,X_2,X_3,X_4,X_5,X_6)$ by $k$-automorphisms 
\begin{align*}
\rho : &\ X_1 \mapsto X_2, X_2 \mapsto -X_1, X_3 \mapsto X_4, 
X_4 \mapsto X_3,\\
&\ X_5 \mapsto X_6, 
 X_6 \mapsto \frac{(X_1^2X_2^2-1)(X_1^2X_3^2+X_2^2-X_3^2-1)}{X_5}.
\end{align*}
{\rm (ii)} {\it The field $L_k^{(3)}$} 
is defined to be $k(X_1,X_2,X_3,X_4,X_5,X_6,X_7)^{\langle\lambda_1,\lambda_2\rangle}$ 
where $\langle\lambda_1,\lambda_2\rangle \simeq C_2\times C_2$ acts on 
$k(X_1,X_2,X_3,X_4,X_5,X_6,X_7)$ by $k$-automorphisms 
\begin{align*}
\lambda_1 : &\ X_1 \mapsto X_1, X_2 \mapsto \frac{X_1}{X_2}, 
X_3 \mapsto \frac{1}{X_1 X_3}, X_4 \mapsto \frac{X_2 X_4}{X_1 X_3},\\ 
&\ X_5 \mapsto -\frac{X_1 X_6^2-1}{X_5}, X_6 \mapsto -X_6, X_7 \mapsto X_7,\\
\lambda_2 : &\ X_1 \mapsto \frac{1}{X_1}, X_2 \mapsto X_3, X_3 \mapsto X_2, 
 X_4 \mapsto \frac{(X_1 X_6^2-1) (X_1 X_7^2-1)}{X_4},\\
&\ X_5 \mapsto -X_5, X_6 \mapsto -X_1 X_6, X_7 \mapsto -X_1 X_7.
\end{align*}
\end{definition}

\begin{theorem}[see Theorem \ref{thmain2}]\label{thm}
Let $G$ be a group of order $128$. 
Assume that $B_0(G)\neq 0$. 
Then $\bC(G)$ and $L_\bC^{(m)}$ are stably $\bC$-isomorphic where 
\begin{align*}
m=
\begin{cases}
1&{\rm if}\ G\ {\rm belongs\ to\ } \Phi_{16}, \Phi_{31}, \Phi_{37}, 
\Phi_{39}, \Phi_{43}, \Phi_{58}, \Phi_{60}\ {\rm or}\ \Phi_{80},\\
2&{\rm if}\ G\ {\rm belongs\ to\ } \Phi_{106}\ {\rm or}\ \Phi_{114},\\
3&{\rm if}\ G\ {\rm belongs\ to\ } \Phi_{30}.
\end{cases}
\end{align*}
In particular, ${\rm Br}_{\rm nr}(L_\bC^{(1)})\simeq {\rm Br}_{\rm nr}(L_\bC^{(2)})\simeq C_2$ 
and ${\rm Br}_{\rm nr}(L_\bC^{(3)})\simeq C_2\times C_2$ and hence the fields 
$L_\bC^{(1)}$, $L_\bC^{(2)}$ and $L_\bC^{(3)}$ are not $($retract, stably$)$ $\bC$-rational.
\end{theorem}

For $m=1,2$, 
the fields $L_\bC^{(m)}$ and $L_\bC^{(3)}$ are not stably $\bC$-isomorphic 
because their unramified Brauer groups are not isomorphic. 
However, we do not know whether the fields $L_\bC^{(1)}$ and $L_\bC^{(2)}$ are stably 
$\bC$-isomorphic. 
If not, it is interesting to evaluate the higher unramified cohomologies. 
Unfortunately, a useful formula like Bogomolov's formula (Theorem \ref{thBog}) 
or Moravec's formula \cite[Section 3]{Mor12} for $B_0(G)$ 
is unknown for higher unramified cohomologies. 

Theorem \ref{thm} gives another proof of $B_0(G)\simeq C_2$ to 
Theorem \ref{thMo128} when $G$ belongs to 
$\Phi_{16}$, $\Phi_{31}$, $\Phi_{37}$, $\Phi_{39}$, $\Phi_{43}$, 
$\Phi_{58}$, $\Phi_{60}$ or $\Phi_{80}$. 
Especially, 
this proof is based on the result of order $64$ for $\Phi_{16}$ 
(Theorem \ref{thCHKK10}) and it does not 
depend on the computer calculations of GAP. 

Although Theorem \ref{thm} gives only the first step,  
the author hopes that it will stimulate further work 
towards a more complete understanding of 
the (stably) birational classification of $\bC(G)$ for non-abelian groups $G$.

\begin{notation}
Throughout this paper, 
$G$ is a finite group and $k$ is a base field. 
$Z(G)$ denotes the center of $G$. 
For $g, h \in G$, define the commutator 
$[g,h]=g^{-1}h^{-1}gh$. 
The exponent of a group $G$ is defined to be 
${\rm lcm}\{{\rm ord}(g):g\in G\}$ where ${\rm ord}(g)$ is the order of
the element $g$. 

We denote by $\zeta_n$ a primitive $n$-th root of unity 
in a fixed algebraic closure of $k$. 
Whenever we write $\zeta_n\in k$, it is understood that either 
{\rm char} $k=0$ or {\rm char} $k=p$ with $p {\not |}$ $n$. 
We will write 
$\zeta$ for $\zeta_4$ for simplicity, 
$\eta$ for a primitive $8$th root of unity $\zeta_8$ satisfying $\eta^2=\zeta$ 
and $\omega$ for a primitive $16$th root of unity $\zeta_{16}$ 
satisfying $\omega^2=\eta$.

The group
$G(2^7,i)$ of order $128$, or $G(i)$ for short, is the $i$-th group of 
order $128$ in the GAP database \cite{GAP}. 
The version of GAP used in this paper is GAP4, 
Version: 4.4.12 \cite{GAP}.
\end{notation}

\begin{acknowledgments}
Before getting a Ph.D., 
the author stayed Regensburg University for one year: 2004--2005. 
He thanks his teacher Manfred Knebusch who led him to the extensive 
world of quadratic forms. 
He also thanks Ming-chang Kang for many helpful suggestions.
\end{acknowledgments}

\section{Proof of Theorem \ref{thm}}

We will prove Theorem \ref{thm}. 
By Bogomolov and B\"ohning's theorem (Theorem \ref{thBB}), we may choose 
one group $G$ within each isoclinism family $\Phi_j$. 
We will choose the first one, i.e. $G=G(2^7,i)=G(i)$ with the minimal $i$ of the GAP database 
within each isoclinism family (see also Table 2 in Section 3). 
More precisely, we will show the following theorem.
\begin{theorem}\label{thmain2}
Let $G=G(i)$ be the $i$-th group of order $128$ with exponent $e$ 
in the GAP database {\rm \cite{GAP}}.
Let $k$ be a field with {\rm char} $k\neq 2$ and $\zeta_e\in k$.\\
{\rm (i)} If $G$ is one of the groups 
$G(227)$, $G(1345)$, $G(242)$, $G(36)$, $G(1924)$, $G(417)$, 
$G(446)$ and $G(950)$ 
which belong to the isoclinism families $\Phi_{16}$, $\Phi_{31}$, $\Phi_{37}$, 
$\Phi_{39}$, $\Phi_{43}$, $\Phi_{58}$, $\Phi_{60}$ 
and $\Phi_{80}$ respectively, 
then there exists a $k$-injective homomorphism $\varphi : L_k^{(1)}\rightarrow k(G)$ 
such that $k(G)$ is rational over $\varphi(L_k^{(1)})$. 
In particular, $\bC(G)$ and $L_\bC^{(1)}$ are stably $\bC$-isomorphic 
and $B_0(G)\simeq {\rm Br}_{\rm nr}(L_\bC^{(1)})\simeq C_2$;\\
{\rm (ii)} If $G$ is one of the groups 
$G(144)$ and $G(138)$ 
which belong to $\Phi_{106}$ and $\Phi_{114}$ respectively, 
then there exists a $k$-injective homomorphism $\varphi : L_k^{(2)}\rightarrow k(G)$ 
such that $k(G)$ is rational over $\varphi(L_k^{(2)})$. 
In particular, $\bC(G)$ and $L_\bC^{(2)}$ are stably $\bC$-isomorphic 
and $B_0(G)\simeq {\rm Br}_{\rm nr}(L_\bC^{(2)})\simeq C_2$;\\
{\rm (iii)} If $G=G(1544)$ which belongs to $\Phi_{30}$, 
then there exists a $k$-injective homomorphism $\varphi : L_k^{(3)}\rightarrow k(G)$ 
such that $k(G)$ is rational over $\varphi(L_k^{(3)})$. 
In particular, $\bC(G)$ and $L_\bC^{(3)}$ are stably $\bC$-isomorphic and 
$B_0(G)\simeq {\rm Br}_{\rm nr}(L_\bC^{(3)})\simeq C_2\times C_2$. 
\end{theorem}

\bigskip

{\it Proof of Theorem \ref{thmain2}.}

We first prepare the following lemmas which will be used. 

\begin{theorem}[Hajja and Kang {\cite[Theorem 1]{HK95}}] \label{thHK}
Let $L$ be any field and 
$G$ be a finite group acting on $L(x_1,\ldots,x_n)$,
the rational function field of $n$ variables over $L$.
Suppose that\\
{\rm (i)} for any $\sigma \in G$, $\sigma(L)\subset L$;\\
{\rm (ii)} the restriction of the action of $G$ to $L$ is faithful; and\\
{\rm (iii)} for any $\sigma \in G$,
\[
\begin{pmatrix} \sigma (x_1) \\ \sigma (x_2) \\ \vdots \\ \sigma (x_n) \end{pmatrix}
= A(\sigma)\cdot \begin{pmatrix} x_1 \\ x_2 \\ \vdots \\ x_n \end{pmatrix} +B(\sigma)
\]
where $A(\sigma)\in GL_n(L)$ and $B(\sigma)$ is an $n\times 1$ matrix over $L$.\\
Then there exist $z_1,\ldots,z_n\in L(x_1,\ldots,x_n)$ such
that $L(x_1,\ldots,x_n)=L(z_1,\ldots,z_n)$
and $\sigma(z_i)=z_i$ for any $\sigma\in G$, any $1\le i\le n$.
\end{theorem}
\begin{theorem}[Ahmad, Hajja and Kang {\cite[Theorem 3.1]{AHK00}}]\label{thAHK}
Let $L$ be any field, $L(x)$ be the rational function field 
over $L$ with variable $x$ and 
$G$ be a finite group acting on $L(x)$. 
Suppose that, for any $\sigma\in G$, 
$\sigma (L)\subset L$ and $\sigma(x)=a_\sigma x+b_\sigma$ where 
$a_\sigma, b_\sigma\in L$ and $a_\sigma\neq 0$. 
Then $L(x)^G=L^G(f)$ for some polynomial $f\in L[x]$. 
In fact, if 
$m={\rm min} \{${\rm deg} $g(x) : g(x)\in L[x]^G\setminus L\}$, 
any polynomial $f\in L[x]^G$ with {\rm deg} $f=m$ satisfies 
the property $L(x)^G=L^G(f)$. 
\end{theorem}
\begin{lemma}[Hoshi, Kitayama and Yamasaki {\cite[Lemma 3.9]{HKY11}}]\label{lemHKY}
Let $k$ be a field with {\rm char} $k\neq 2$ and 
$\langle\tau_3\rangle\simeq C_2$ act on the rational function 
field $k(x,y,z)$ over $k$ with variables $x,y,z$ 
by $k$-automorphisms 
\begin{align*}
\tau_3:x \mapsto y\mapsto x,\ 
z\mapsto \frac{c}{xyz}\mapsto z,\ c\in k^\times. 
\end{align*}
Then $k(x,y,z)^{\langle \tau_3\rangle}=k(t_1,t_2,t_3)$ where
\begin{align*} 
t_1=\frac{xy}{x+y},\quad t_2=\frac{xyz+\frac{c}{z}}{x+y},\quad 
t_3=\frac{xyz-\frac{c}{z}}{x-y}.
\end{align*}
\end{lemma}

\bigskip

We will separate the proof of Theorem \ref{thmain2} {\rm (i), (ii), (iii)} 
into Case $1$ to Case $11$:\\

\begin{tabular}{ll} 
{\rm (i)} & 
Case 1: $G=G(2^7,227)$ which belongs to $\Phi_{16}$;\\
& Case 2: $G=G(2^7,1345)$ which belongs to $\Phi_{31}$;\\ 
& Case 3: $G=G(2^7,242)$ which belongs to $\Phi_{37}$;\\ 
& Case 4: $G=G(2^7,36)$ which belongs to $\Phi_{39}$;\\
& Case 5: $G=G(2^7,1924)$ which belongs to $\Phi_{43}$;\\ 
& Case 6: $G=G(2^7,417)$ which belongs to $\Phi_{58}$;\\
& Case 7: $G=G(2^7,446)$ which belongs to $\Phi_{60}$;\\
& Case 8; $G=G(2^7,950)$ which belongs to $\Phi_{80}$;\\
{\rm (ii)} & 
Case 9: $G=G(2^7,144)$ which belongs to $\Phi_{106}$;\\
& Case 10: $G=G(2^7,138)$ which belongs to $\Phi_{114}$;\\
{\rm (iii)} & 
Case 11: $G=G(2^7,1544)$ which belongs to $\Phi_{30}$.
\end{tabular}

\bigskip

The generators and the relations of the groups $G=G(2^7,i)$ 
can be found in the GAP database, e.g. 
\verb+PrintPcpPresentation(PcGroupToPcpGroup(SmallGroup(2^7,i)))+.\\
Recall that $\zeta=\zeta_4$ is a primitive $4$th roof of unity, 
$\eta$ is a primitive $8$th root of unity satisfying $\eta^2=\zeta$ 
and $\omega$ is a primitive $16$th root of unity satisfying $\omega^2=\eta$.\\

{Case $1$: $G=G(2^7,227)$ which belongs to $\Phi_{16}$.}

$G=\langle g_1,g_2,g_3,g_4,g_5,g_6,g_7\rangle$ with relations 
$g_1^2=g_5$, $g_2^2=1$, $g_3^2=1$, $g_4^2=g_6$, 
$g_5^2=g_7$, $g_6^2=1$, $g_7^2=1$, 
$Z(G)=\langle g_5,g_6,g_7\rangle$, 
$[g_2,g_1]=g_4$, $[g_3,g_1]=g_7$, $[g_3,g_2]=g_6g_7$, 
$[g_4,g_1]=g_6$, $[g_4,g_2]=g_6$.

There exists a faithful representation 
$\rho : G\rightarrow GL(V_{227})\simeq GL_6(k)$ 
of dimension $6$ 
which is decomposable into two irreducible components 
$V_{227}\simeq U_4\oplus U_2$ of dimension $4$ and $2$ respectively. 
By Theorem \ref{thHK}, $k(G)$ is rational over $k(V_{227})^G$. 
Hence it is enough to show that $k(V_{227})^G$ 
is rational over $\varphi(L_k^{(1)})$. 

The action of $G$ on $k(V_{227})=k(y_1,y_2,y_3,y_4,y_5,y_6)$ 
is given by 
\begin{align*}
g_1&:y_1\mapsto \z y_4,y_2\mapsto -\z y_3,y_3\mapsto -y_2,y_4\mapsto y_1,y_5\mapsto \eta y_6,y_6\mapsto \eta y_5,\\
g_2&:y_1\mapsto y_3,y_2\mapsto y_4,y_3\mapsto y_1,y_4\mapsto y_2,y_5\mapsto y_6,y_6\mapsto y_5,\\
g_3&:y_1\mapsto -y_1,y_2\mapsto y_2,y_3\mapsto -y_3,y_4\mapsto y_4,y_5\mapsto -y_5,y_6\mapsto y_6,\\
g_4&:y_1\mapsto -\z y_1,y_2\mapsto -\z y_2,y_3\mapsto \z y_3,y_4\mapsto \z y_4,y_5\mapsto y_5,y_6\mapsto y_6,\\
g_5&:y_1\mapsto \z y_1,y_2\mapsto \z y_2,y_3\mapsto \z y_3,y_4\mapsto \z y_4,y_5\mapsto \z y_5,y_6\mapsto \z y_6,\\
g_6&:y_1\mapsto -y_1,y_2\mapsto -y_2,y_3\mapsto -y_3,y_4\mapsto -y_4,y_5\mapsto y_5,y_6\mapsto y_6,\\
g_7&:y_1\mapsto -y_1,y_2\mapsto -y_2,y_3\mapsto -y_3,y_4\mapsto -y_4,y_5\mapsto -y_5,y_6\mapsto -y_6.
\end{align*}

Define $z_1=\frac{y_1}{y_4}$, $z_2=\frac{y_2}{y_4}$, 
$z_3=\frac{y_3}{y_4}$, $z_4=\frac{y_5}{y_6}$, $z_5=y_4$, $z_6=y_6$. 
Then $k(y_1,y_2,y_3,y_4,y_5,y_6)=k(z_1,z_2,z_3,z_4,z_5,z_6)$ and 
\begin{align*}
g_1&:z_1\mapsto \tfrac{\z}{z_1},z_2\mapsto -\tfrac{\z z_3}{z_1}, z_3\mapsto -\tfrac{z_2}{z_1}, 
z_4\mapsto \tfrac{1}{z_4},z_5\mapsto z_1z_5,z_6\mapsto \eta z_4z_6,\\
g_2&:z_1\mapsto \tfrac{z_3}{z_2},z_2\mapsto \tfrac{1}{z_2},z_3\mapsto \tfrac{z_1}{z_2},z_4\mapsto \tfrac{1}{z_4},
z_5\mapsto z_2z_5,z_6\mapsto z_4z_6,\\
g_3&:z_1\mapsto -z_1,z_2\mapsto z_2,z_3\mapsto -z_3,z_4\mapsto -z_4,z_5\mapsto z_5,z_6\mapsto z_6,\\
g_4&:z_1\mapsto -z_1,z_2\mapsto -z_2,z_3\mapsto z_3,z_4\mapsto z_4,z_5\mapsto \z z_5,z_6\mapsto z_6,\\
g_5&:z_1\mapsto z_1,z_2\mapsto z_2,z_3\mapsto z_3,z_4\mapsto z_4,z_5\mapsto \z z_5,z_6\mapsto \z z_6,\\
g_6&:z_1\mapsto z_1,z_2\mapsto z_2,z_3\mapsto z_3,z_4\mapsto z_4,z_5\mapsto -z_5,z_6\mapsto z_6,\\
g_7&:z_1\mapsto z_1,z_2\mapsto z_2,z_3\mapsto z_3,z_4\mapsto z_4,z_5\mapsto -z_5,z_6\mapsto -z_6.
\end{align*}

Apply Theorem \ref{thAHK} twice to $k(z_1,z_2,z_3,z_4)(z_5,z_6)$, 
$k(V_{227})^G$ $=$ 
$k(z_1,z_2,z_3,z_4,z_5,z_6)^G$ is rational over 
$k(z_1,z_2,z_3,z_4)^G$. 
We find that 
$k(z_1,z_2,z_3,z_4)^G
=k(z_1,z_2,z_3,z_4)^{\langle g_1,g_2,g_3,g_4\rangle}$ 
because $Z(G)=\langle g_5, g_6, g_7\rangle$ acts on $k(z_1,z_2,z_3,z_4)$ trivially. 
Thus it suffices to show that 
$k(z_1,z_2,z_3,z_4)^{\langle g_1,g_2,g_3,g_4\rangle}$ 
is rational over $\varphi(L_k^{(1)})$.

Define $u_1=\frac{z_2z_3}{z_1}$, $u_2=\frac{z_1}{z_2z_4}$, 
$u_3=z_3z_4$, $u_4=\frac{z_3}{z_1z_2}$.
Then $k(u_1,u_2,u_3,u_4)\subset 
k(z_1,z_2,z_3,z_4)^{\langle g_3,g_4\rangle}$ and 
the field extension degree 
$[k(z_1,z_2,z_3,z_4):k(u_1,u_2,u_3,u_4)]=4$ because 
the determinant of the matrix of exponents 
of $u_1$, $u_2$, $u_3$, $u_4$ with respect to $z_1$, $z_2$, $z_3$, $z_4$
is $4$: 
\begin{align}
\det\left(
\begin{array}{cccc}
 -1 & 1 & 0 & -1 \\
 1 & -1 & 0 & -1 \\
 1 & 0 & 1 & 1 \\
 0 & -1 & 1 & 0 \\
\end{array}
\right)=4.\label{eqdet}
\end{align}
Hence we have 
$k(z_1,z_2,z_3,z_4)^{\langle g_3,g_4\rangle}=k(u_1,u_2,u_3,u_4)$ 
and 
\begin{align}
g_1&:u_1\mapsto u_1,u_2\mapsto -\tfrac{1}{u_1u_2},u_3\mapsto -\tfrac{u_1}{u_3},
u_4\mapsto -\tfrac{1}{u_4},\label{act227}\\
g_2&:u_1\mapsto \tfrac{1}{u_1},u_2\mapsto u_3,u_3\mapsto u_2,u_4\mapsto \tfrac{1}{u_4}.\nonumber
\end{align}

Define 
$v_1=\bigl(\frac{u_1+1}{u_1-1}\bigr)\bigl(\frac{u_4+1}{u_4-1}\bigr)$, 
$v_2=u_2+u_3$, 
$v_3=(u_2-u_3)\bigl(\frac{u_1+1}{u_1-1}\bigr)$, 
$v_4=\bigl(\frac{u_4+1}{u_4-1}\bigr)\big/\bigl(\frac{u_1+1}{u_1-1}\bigr)$.
Then $k(u_1,u_2,u_3,u_4)^{\langle g_2\rangle}=k(v_1,v_2,v_3,v_4)$ and 
\begin{align*}
g_1:v_1\mapsto -\tfrac{1}{v_4},v_2\mapsto -\tfrac{4v_1(v_1v_2+v_2v_4+2v_3v_4)}{(v_1-v_4)(v_1v_2^2-v_3^2v_4)}, 
v_3\mapsto \tfrac{4v_1(2v_1v_2+v_1v_3+v_3v_4)}{(v_1-v_4)(v_1v_2^2-v_3^2v_4)},
v_4\mapsto -\tfrac{1}{v_1}.
\end{align*}

Define $w_1=v_1$, $w_2=v_1v_2+v_2v_4+2v_3v_4$, 
$w_3=2v_1v_2+v_1v_3+v_3v_4$, $w_4=\tfrac{v_1}{v_4}$. 
Then $k(v_1,v_2,v_3,v_4)=k(w_1,w_2,w_3,w_4)$ and 
\begin{align*}
g_1: w_1\mapsto -\tfrac{w_4}{w_1}, 
w_2\mapsto -\tfrac{4(w_4-1)(w_2-2w_3+w_2w_4)}{w_2^2w_4-w_3^2}, 
w_3\mapsto \tfrac{4(w_4-1)(w_3-2w_2w_4+w_3w_4)}{w_2^2w_4-w_3^2},
w_4\mapsto w_4.
\end{align*}

We also define 
$X_1=\frac{w_2w_4-w_3}{w_2-w_3}$, 
$X_2=\z w_1$, 
$X_3=\frac{(w_2-2w_3+w_2w_4)(w_2^2w_4-w_3^2)}{2w_2(w_2-w_3)(w_4-1)}$, 
$X_4=w_4$. 
Then $k(w_1,w_2,w_3,w_4)=k(X_1,X_2,X_3,X_4)$ and 
\begin{align*}
g_1:X_1\mapsto -X_1,\ 
X_2\mapsto \tfrac{X_4}{X_2},\ 
X_3\mapsto \tfrac{(X_4-1)(X_4-X_1^2)}{X_3},\ 
X_4\mapsto X_4.
\end{align*}
This action of $g_1$ on $k(X_1,X_2,X_3,X_4)$ 
and that of $\tau$ on $k(X_1,X_2,X_3,X_4)$ 
in Definition \ref{defL01} (ii) 
are exactly the same. 
Hence $k(V_{227})^G$ is rational over $\varphi(L_k^{(1)})$.\\

{Case $2$: $G=G(2^7,1345)$ which belongs to $\Phi_{31}$.}

$G=\langle g_1,g_2,g_3,g_4,g_5,g_6,g_7\rangle$ with relations 
$g_1^2=g_2^2=g_3^2=g_4^2=g_5^2=g_6^2=g_7^2=1$, 
$Z(G)=\langle g_5,g_6,g_7\rangle$, 
$[g_2,g_1]=g_5$, $[g_3,g_1]=g_6$, $[g_3,g_2]=g_7$, 
$[g_4,g_3]=g_5$.

There exists a faithful representation 
$\rho : G\rightarrow GL(V_{1345})\simeq GL_8(k)$ 
of dimension $8$ 
which is decomposable into three irreducible components 
$V_{1345}\simeq U_4\oplus U_2\oplus U_2^\prime$ of dimension $4$, $2$ 
and $2$ respectively. 
By Theorem \ref{thHK}, $k(G)$ is rational over $k(V_{1345})^G$. 
We will show that $k(V_{1345})^G$ is rational over $\varphi(L_k^{(1)})$. 

The action of $G$ on $k(V_{1345})=k(y_1,y_2,y_3,y_4,y_5,y_6,y_7,y_8)$ 
is given by 
\begin{align*}
g_1&:y_1\mapsto y_3,y_2\mapsto y_4,y_3\mapsto y_1,y_4\mapsto y_2,y_5\mapsto y_6,y_6\mapsto y_5,y_7\mapsto y_7,
y_8\mapsto y_8,\\
g_2&:y_1\mapsto -y_1,y_2\mapsto -y_2,y_3\mapsto y_3,y_4\mapsto y_4,y_5\mapsto y_5,y_6\mapsto y_6,
y_7\mapsto y_8,y_8\mapsto y_7,\\
g_3&:y_1\mapsto y_2,y_2\mapsto y_1,y_3\mapsto y_4,y_4\mapsto y_3,y_5\mapsto -y_5,y_6\mapsto y_6,
y_7\mapsto -y_7,y_8\mapsto y_8,\\
g_4&:y_1\mapsto -y_1,y_2\mapsto y_2,y_3\mapsto -y_3,y_4\mapsto y_4,y_5\mapsto y_5,y_6\mapsto y_6,
y_7\mapsto y_7,y_8\mapsto y_8,\\
g_5&:y_1\mapsto -y_1,y_2\mapsto -y_2,y_3\mapsto -y_3,y_4\mapsto -y_4,y_5\mapsto y_5,y_6\mapsto y_6,
y_7\mapsto y_7,y_8\mapsto y_8,\\
g_6&:y_1\mapsto y_1,y_2\mapsto y_2,y_3\mapsto y_3,y_4\mapsto y_4,y_5\mapsto -y_5,y_6\mapsto -y_6,
y_7\mapsto y_7,y_8\mapsto y_8,\\
g_7&:y_1\mapsto y_1,y_2\mapsto y_2,y_3\mapsto y_3,y_4\mapsto y_4,y_5\mapsto y_5,y_6\mapsto y_6,
y_7\mapsto -y_7,y_8\mapsto -y_8.
\end{align*}

Define $z_1=\frac{y_1}{y_4}$, $z_2=\frac{y_2}{y_4}$, 
$z_3=\frac{y_3}{y_4}$, $z_4=\frac{y_5}{y_6}$, $z_5=\frac{y_7}{y_8}$, 
$z_6=y_4$, $z_7=y_6$, $z_8=y_8$. 
Then $k(y_1,y_2,y_3,y_4,y_5,y_6,$ $y_7,y_8)
=k(z_1,z_2,z_3,z_4,z_5,z_6,z_7,z_8)$ and 
\begin{align*}
g_1&:z_1\mapsto \tfrac{z_3}{z_2},z_2\mapsto \tfrac{1}{z_2},z_3\mapsto \tfrac{z_1}{z_2},z_4\mapsto \tfrac{1}{z_4},z_5\mapsto z_5,z_6\mapsto z_2z_6,z_7\mapsto z_4z_7,z_8\mapsto z_8,\\
g_2&:z_1\mapsto -z_1,z_2\mapsto -z_2,z_3\mapsto z_3,z_4\mapsto z_4,z_5\mapsto \tfrac{1}{z_5},z_6\mapsto z_6,z_7\mapsto z_7,z_8\mapsto z_5z_8,\\
g_3&:z_1\mapsto \tfrac{z_2}{z_3},z_2\mapsto \tfrac{z_1}{z_3},z_3\mapsto \tfrac{1}{z_3},z_4\mapsto -z_4,z_5\mapsto -z_5,z_6\mapsto z_3z_6,z_7\mapsto z_7,z_8\mapsto z_8,\\
g_4&:z_1\mapsto -z_1,z_2\mapsto z_2,z_3\mapsto -z_3,z_4\mapsto z_4,z_5\mapsto z_5,z_6\mapsto z_6,z_7\mapsto z_7,z_8\mapsto z_8,\\
g_5&:z_1\mapsto z_1,z_2\mapsto z_2,z_3\mapsto z_3,z_4\mapsto z_4,z_5\mapsto z_5,z_6\mapsto -z_6,z_7\mapsto z_7,z_8\mapsto z_8,\\
g_6&:z_1\mapsto z_1,z_2\mapsto z_2,z_3\mapsto z_3,z_4\mapsto z_4,z_5\mapsto z_5,z_6\mapsto z_6,z_7\mapsto -z_7,z_8\mapsto z_8,\\
g_7&:z_1\mapsto z_1,z_2\mapsto z_2,z_3\mapsto z_3,z_4\mapsto z_4,z_5\mapsto z_5,z_6\mapsto z_6,z_7\mapsto z_7,z_8\mapsto -z_8.
\end{align*}

Apply Theorem \ref{thAHK} three times to 
$k(z_1,z_2,z_3,z_4,z_5)(z_6,z_7,z_8)$, 
$k(z_1,z_2,z_3,z_4,z_5,z_6,z_7,z_8)^G$ is rational over 
$k(z_1,z_2,z_3,z_4,z_5)^G$. 
We find that 
$k(z_1,z_2,z_3,z_4,z_5)^G
=k(z_1,z_2,z_3,z_4,z_5)^{\langle g_1,g_2,g_3,g_4\rangle}$ 
because $Z(G)=\langle g_5, g_6, g_7\rangle$ acts on $k(z_1,z_2,z_3,z_4,z_5)$ trivially. 
It suffices to show that the invariant field 
$k(z_1,z_2,z_3,z_4,z_5)^{\langle g_1,g_2,g_3,g_4\rangle}$ is rational 
over $\varphi(L_k^{(1)})$. 

Define $u_1=\frac{z_1z_2}{z_3}$, $u_2=\frac{z_2z_3}{z_1}$, 
$u_3=\frac{z_1z_3}{z_2}$, $u_4=z_4$, $u_5=\frac{z_5+1}{z_2(z_5-1)}$.
Note that $\frac{u_5^\prime+1}{u_5^\prime-1}=z_5$ where $u_5^\prime=u_5 z_2$.
Hence $k(u_1,u_2,u_3,u_4,u_5)=k(u_1,u_2,u_3,u_4,\frac{z_5}{z_2})$. 
By evaluating the determinant of the matrix $M$ 
of exponents as in Case 1 
(see the equation (\ref{eqdet})), we have $\det M=4$, 
$k(z_1,z_2,z_3,z_4,z_5)^{\langle g_2,g_4\rangle}
=k(u_1,u_2,u_3,u_4,u_5)$ 
and 
\begin{align*}
g_1&:u_1\mapsto \tfrac{1}{u_1}, u_2\mapsto \tfrac{1}{u_2}, 
u_3\mapsto u_3, u_4\mapsto \tfrac{1}{u_4}, u_5\mapsto u_1u_2u_5,\\
g_3&:u_1\mapsto u_1, u_2\mapsto \tfrac{1}{u_2}, 
u_3\mapsto \tfrac{1}{u_3}, u_4\mapsto -u_4, u_5\mapsto \tfrac{1}{u_1u_5}.
\end{align*}

Define 
$v_1=\bigl(\frac{u_1+1}{u_1-1}\bigr)^2$, 
$v_2=\bigl(\frac{u_1+1}{u_1-1}\bigr)\bigl(\frac{u_2+1}{u_2-1}\bigr)$, 
$v_3=\bigl(\frac{u_1+1}{u_1-1}\bigr)\bigl(\frac{u_2+1}{u_2-1}\bigr)
\bigl(\frac{u_3+1}{u_3-1}\bigr)$,
$v_4=\bigl(\frac{u_1+1}{u_1-1}\bigr)\bigl(\frac{u_4+1}{u_4-1}\bigr)$, 
$v_5=(u_1u_2+1)u_5$.
Then $k(u_1,u_2,u_3,u_4,u_5)^{\langle g_1\rangle}=k(v_1,v_2,v_3,v_4,v_5)$ 
and 
\begin{align*}
g_3:v_1\mapsto v_1, 
v_2\mapsto -v_2, v_3\mapsto v_3, v_4\mapsto \tfrac{v_1}{v_4}, 
v_5\mapsto \tfrac{4v_1(v_2+1)(v_2-1)}{(v_1-1)(v_2^2-v_1)v_5}.
\end{align*}

Define $X_1=\frac{v_1}{v_2}$, $X_2=v_4$, 
$X_3=\frac{(v_1-1)(v_1-v_2^2)v_5}{2v_2(v_2+1)}$, $X_4=v_1$, $X_5=v_3$. 
Then $k(v_1,v_2,v_3,v_4,v_5)=k(X_1,X_2,X_3,X_4,X_5)$ and 
\begin{align*}
g_1:X_1\mapsto -X_1,\ 
X_2\mapsto \tfrac{X_4}{X_2},\ 
X_3\mapsto \tfrac{(X_4-1)(X_4-X_1^2)}{X_3},\ 
X_4\mapsto X_4,\ X_5\mapsto X_5.
\end{align*}
The action of $g_1$ on $k(X_1,X_2,X_3,X_4)$ 
and that of $\tau$ on $k(X_1,X_2,X_3,X_4)$ 
in Definition \ref{defL01} (ii)
are exactly the same. 
Because $g_1$ acts on $k(X_5)$ trivially, we obtain that 
$k(X_1,X_2,X_3,X_4,X_5)^{\langle g_1\rangle}$ $=$ 
$k(X_1,X_2,X_3,X_4)^{\langle g_1\rangle}(X_5)$. 
Hence $k(V_{1345})^G$ is rational over $\varphi(L_k^{(1)})$.\\ 

{Case $3$: $G=G(2^7,242)$ which belongs to $\Phi_{37}$.}

$G=\langle g_1,g_2,g_3,g_4,g_5,g_6,g_7\rangle$ with relations 
$g_1^2=g_5$, $g_2^2=1$, $g_3^2=1$, $g_4^2=g_7$, $g_5^2=1$, $g_6^2=1$, 
$g_7^2=1$, $Z(G)=\langle g_6,g_7\rangle$, 
$[g_2,g_1]=g_4$, $[g_3,g_1]=g_7$, $[g_4,g_1]=g_6$, 
$[g_4,g_2]=g_7$, $[g_5,g_2]=g_6g_7$.

There exists a faithful representation 
$\rho : G\rightarrow GL(V_{242})\simeq GL_8(k)$ 
of dimension $8$ 
which is decomposable into two irreducible components 
$V_{242}\simeq U_4\oplus U_4^\prime$ of dimension $4$. 
By Theorem \ref{thHK}, $k(G)$ is rational over $k(V_{242})^G$. 
We will show that $k(V_{242})^G$ is rational over $\varphi(L_k^{(1)})$. 

The action of $G$ on $k(V_{242})=k(y_1,y_2,y_3,y_4,y_5,y_6,y_7,y_8)$ 
is given by 
\begin{align*}
g_1&:y_1\mapsto -y_2,y_2\mapsto y_1,y_3\mapsto -\z y_4,y_4\mapsto \z y_3,y_5\mapsto -y_7,y_6\mapsto y_8,y_7\mapsto y_5,y_8\mapsto y_6,\\
g_2&:y_1\mapsto y_3,y_2\mapsto y_4,y_3\mapsto y_1,y_4\mapsto y_2,y_5\mapsto y_6,y_6\mapsto y_5,y_7\mapsto y_8,y_8\mapsto y_7,\\
g_3&:y_1\mapsto -y_1,y_2\mapsto y_2,y_3\mapsto -y_3,y_4\mapsto y_4,y_5\mapsto y_5,y_6\mapsto y_6,y_7\mapsto y_7,y_8\mapsto y_8,\\
g_4&:y_1\mapsto -\z y_1,y_2\mapsto -\z y_2,y_3\mapsto \z y_3,y_4\mapsto \z y_4,y_5\mapsto -y_5,y_6\mapsto -y_6,y_7\mapsto y_7,y_8\mapsto y_8,\\
g_5&:y_1\mapsto -y_1,y_2\mapsto -y_2,y_3\mapsto y_3,y_4\mapsto y_4,y_5\mapsto -y_5,y_6\mapsto y_6,y_7\mapsto -y_7,y_8\mapsto y_8,\\
g_6&:y_1\mapsto y_1,y_2\mapsto y_2,y_3\mapsto y_3,y_4\mapsto y_4,y_5\mapsto -y_5,y_6\mapsto -y_6,y_7\mapsto -y_7,y_8\mapsto -y_8,\\
g_7&:y_1\mapsto -y_1,y_2\mapsto -y_2,y_3\mapsto -y_3,y_4\mapsto -y_4,y_5\mapsto y_5,y_6\mapsto y_6,y_7\mapsto y_7,y_8\mapsto y_8.
\end{align*}

Define $z_1=\frac{y_1}{y_4}$, $z_2=\frac{y_2}{y_4}$, 
$z_3=\frac{y_3}{y_4}$, $z_4=\frac{y_5}{y_8}$, $z_5=\frac{y_6}{y_8}$, 
$z_6=\frac{y_7}{y_8}$, $z_7=y_4$, $z_8=y_8$. 
Then $k(y_1,y_2,y_3,y_4,y_5,y_6,$ $y_7,y_8)
=k(z_1,z_2,z_3,z_4,z_5,z_6,z_7,z_8)$ and 

\begin{align*}
g_1&:z_1\mapsto -\tfrac{z_2}{\z z_3},z_2\mapsto \tfrac{z_1}{\z z_3},z_3\mapsto -\tfrac{1}{z_3},z_4\mapsto -\tfrac{z_6}{z_5},z_5\mapsto \tfrac{1}{z_5},z_6\mapsto \tfrac{z_4}{z_5},z_7\mapsto \z z_3z_7,z_8\mapsto z_5z_8,\\
g_2&:z_1\mapsto \tfrac{z_3}{z_2},z_2\mapsto \tfrac{1}{z_2},z_3\mapsto \tfrac{z_1}{z_2},z_4\mapsto \tfrac{z_5}{z_6},z_5\mapsto \tfrac{z_4}{z_6},z_6\mapsto \tfrac{1}{z_6},z_7\mapsto z_2z_7,z_8\mapsto z_6z_8,\\
g_3&:z_1\mapsto -z_1,z_2\mapsto z_2,z_3\mapsto -z_3,z_4\mapsto z_4,z_5\mapsto z_5,z_6\mapsto z_6,z_7\mapsto z_7,z_8\mapsto z_8,\\
g_4&:z_1\mapsto -z_1,z_2\mapsto -z_2,z_3\mapsto z_3,z_4\mapsto -z_4,z_5\mapsto -z_5,z_6\mapsto z_6,z_7\mapsto \z z_7,z_8\mapsto z_8,\\
g_5&:z_1\mapsto -z_1,z_2\mapsto -z_2,z_3\mapsto z_3,z_4\mapsto -z_4,z_5\mapsto z_5,z_6\mapsto -z_6,z_7\mapsto z_7,z_8\mapsto z_8,\\
g_6&:z_1\mapsto z_1,z_2\mapsto z_2,z_3\mapsto z_3,z_4\mapsto z_4,z_5\mapsto z_5,z_6\mapsto z_6,z_7\mapsto z_7,z_8\mapsto -z_8,\\
g_7&:z_1\mapsto z_1,z_2\mapsto z_2,z_3\mapsto z_3,z_4\mapsto z_4,z_5\mapsto z_5,z_6\mapsto z_6,z_7\mapsto -z_7,z_8\mapsto z_8.
\end{align*}

Apply Theorem \ref{thAHK} twice to $k(z_1,z_2,z_3,z_4,z_5,z_6)(z_7,z_8)$, 
$k(z_1,z_2,z_3,z_4,z_5,z_6,z_7,z_8)^G$ is rational over 
$k(z_1,z_2,z_3,z_4,z_5,z_6)^G$. 
We find that 
$k(z_1,z_2,z_3,z_4,z_5,z_6)^G
=k(z_1,z_2,z_3,z_4,z_5,z_6)^{\langle g_1,g_2,g_3,g_4,g_5\rangle}$ 
because $Z(G)=\langle g_6, g_7\rangle$ acts on $k(z_1,z_2,z_3,z_4,z_5,z_6)$ trivially. 
Hence it suffices to show that 
$k(z_1,z_2,z_3,z_4,z_5,z_6)^{\langle g_1,g_2,g_3,g_4,g_5\rangle}$ 
is rational over $\varphi(L_k^{(1)})$. 

Define $u_1=\frac{z_1z_2^2z_6}{z_3^3z_5}, 
u_2=\frac{z_3}{z_1^3z_4}, 
u_3=-\frac{z_3^2z_5}{\z z_2^3z_6}, 
u_4=\frac{z_1z_3}{z_2}, 
u_5=\frac{z_1z_2+z_3}{z_1z_2-z_3}, 
u_6=\frac{z_4+z_5z_6}{z_4-z_5z_6}$.
Note that $\frac{u_5+1}{u_5-1}=\frac{z_1z_2}{z_3}$ and 
$\frac{u_6+1}{u_6-1}=\frac{z_4}{z_5z_6}$.
By evaluating the determinant of the matrix $M$ of exponents as in Case 1 
(see the equation (\ref{eqdet})), we have $\det M=-8$,
$k(z_1,z_2,z_3,z_4,z_5,z_6)^{\langle g_3,g_4,g_5\rangle}
=k(u_1,u_2,u_3,u_4,u_5,u_6)$ 
and 
\begin{align}
g_1&:u_1\mapsto \tfrac{-1}{u_1u_2u_3},u_2\mapsto u_3,u_3\mapsto u_2,u_4\mapsto \tfrac{1}{u_4},u_5\mapsto \tfrac{1}{u_5},
u_6\mapsto \tfrac{-1}{u_6},\label{act242}\\
g_2&:u_1\mapsto u_2,u_2\mapsto u_1,u_3\mapsto \tfrac{-1}{u_1u_2u_3},u_4\mapsto u_4,u_5\mapsto -u_5,u_6\mapsto -u_6.\nonumber
\end{align}

By applying Lemma \ref{lemHKY}, we obtain 
$k(u_1,u_2,u_3,u_4,u_5,u_6)^{\langle g_2\rangle}
=k(v_1,v_2,v_3,v_4,v_5,v_6)$ where 
\begin{align*}
v_1=\tfrac{u_1u_2}{u_1+u_2},\ 
v_2=\tfrac{u_1u_2u_3+\tfrac{-1}{u_3}}{u_1+u_2},\ 
v_3=\tfrac{u_1u_2u_3-\tfrac{-1}{u_3}}{u_1-u_2},\ 
v_4=u_4,\ 
v_5=\tfrac{u_5}{u_1-u_2},\ 
v_6=\tfrac{u_6}{u_1-u_2}.
\end{align*}
The action of $g_1$ on $k(v_1,v_2,v_3,v_4,v_5,v_6)$ are given by 
\begin{align*}
g_1:v_1\mapsto \tfrac{v_2^2-v_3^2}{4v_1v_2(v_3^2+1)},
v_2\mapsto -\tfrac{1}{v_2},
v_3\mapsto -\tfrac{1}{v_3},
v_4\mapsto \tfrac{1}{v_4},
v_5\mapsto \tfrac{v_2^2-v_3^2}{4(v_2^2+1)v_3v_5},
v_6\mapsto -\tfrac{v_2^2-v_3^2}{4(v_2^2+1)v_3v_6}.
\end{align*}

Define 
$X_1=\frac{v_3+\z}{v_3-\z}$,
$X_2=\frac{v_3v_6}{v_1v_2}\bigl(\frac{v_2+\z}{v_3-\z}\bigr)$,
$X_3=\frac{4v_1v_2}{\eta(v_2-\z)}\bigl(\frac{v_3+\z}{v_3-\z}\bigr)$,
$X_4=\bigl(\frac{v_2+\z}{v_2-\z}\bigr)\bigl(\frac{v_3+\z}{v_3-\z}\bigr)$,
$X_5=\bigl(\frac{v_2+\z}{v_2-\z}\bigr)\bigl(\frac{v_4+1}{v_4-1}\bigr)$,
$X_6=\bigl(\frac{v_2+\z}{v_2-\z}\bigr)\bigl(\frac{v_5+\z v_6}{v_5-\z v_6}\bigr)$.
Then $k(v_1,v_2,v_3,v_4,v_5,v_6)=k(X_1,X_2,X_3,X_4,X_5,X_6)$ and 
\begin{align*}
g_1:X_1\mapsto -X_1,\ 
X_2\mapsto \tfrac{X_4}{X_2},\ 
X_3\mapsto \tfrac{(X_4-1)(X_4-X_1^2)}{X_3},\ 
X_4\mapsto X_4,\ X_5\mapsto X_5,\ X_6\mapsto X_6.
\end{align*}
The action of $g_1$ on $k(X_1,X_2,X_3,X_4)$ 
and that of $\tau$ on $k(x_1,x_2,x_3,x_4)$ 
in \ref{defL01} (ii)
are exactly the same. 
We also have  
$k(X_1,X_2,X_3,X_4,X_5,X_6)^{\langle g_1\rangle}$ $=$ 
$k(X_1,X_2,X_3,X_4)^{\langle g_1\rangle}(X_5,X_6)$ 
because $g_1$ acts on $k(X_5,X_6)$ trivially. 
Hence $k(V_{242})^{G}$ is rational over $\varphi(L_k^{(1)})$. \\

{Case $4$: $G=G(2^7,36)$ which belongs to $\Phi_{39}$.}

$G=\langle g_1,g_2,g_3,g_4,g_5,g_6,g_7\rangle$ with relations 
$g_1^2=g_4$, $g_2^2=g_5$, $g_3^2=1$, $g_4^2=1$, $g_5^2=1$, $g_6^2=1$, 
$g_7^2=1$, $Z(G)=\langle g_6,g_7\rangle$, 
$[g_2,g_1]=g_3$, $[g_3,g_1]=g_6$, $[g_3,g_2]=g_7$, 
$[g_4,g_2]=g_6$, $[g_5,g_1]=g_7$.

There exists a faithful representation 
$\rho : G\rightarrow GL(V_{36})\simeq GL_8(k)$ 
of dimension $8$ 
which is decomposable into two irreducible components 
$V_{36}\simeq U_4\oplus U_4^\prime$ of dimension $4$.
By Theorem \ref{thHK}, $k(G)$ is rational over $k(V_{36})^G$. 
We will show that $k(V_{36})^G$ is rational over $\varphi(L_k^{(1)})$. 

The action of $G$ on $k(V_{36})=k(y_1,y_2,y_3,y_4,y_5,y_6,y_7,y_8)$ 
is given by 
\begin{align*}
g_1&:y_1\mapsto -y_3,y_2\mapsto y_4,y_3\mapsto y_1,y_4\mapsto y_2,y_5\mapsto y_6,y_6\mapsto y_5,y_7\mapsto y_8,y_8\mapsto y_7,\\
g_2&:y_1\mapsto y_2,y_2\mapsto y_1,y_3\mapsto y_4,y_4\mapsto y_3,y_5\mapsto -y_7,y_6\mapsto y_8,y_7\mapsto y_5,y_8\mapsto y_6,\\
g_3&:y_1\mapsto -y_1,y_2\mapsto -y_2,y_3\mapsto y_3,y_4\mapsto y_4,y_5\mapsto -y_5,y_6\mapsto -y_6,y_7\mapsto y_7,y_8\mapsto y_8,\\
g_4&:y_1\mapsto -y_1,y_2\mapsto y_2,y_3\mapsto -y_3,y_4\mapsto y_4,y_5\mapsto y_5,y_6\mapsto y_6,y_7\mapsto y_7,y_8\mapsto y_8,\\
g_5&:y_1\mapsto y_1,y_2\mapsto y_2,y_3\mapsto y_3,y_4\mapsto y_4,y_5\mapsto -y_5,y_6\mapsto y_6,y_7\mapsto -y_7,y_8\mapsto y_8,\\
g_6&:y_1\mapsto -y_1,y_2\mapsto -y_2,y_3\mapsto -y_3,y_4\mapsto -y_4,y_5\mapsto y_5,y_6\mapsto y_6,y_7\mapsto y_7,y_8\mapsto y_8,\\
g_7&:y_1\mapsto y_1,y_2\mapsto y_2,y_3\mapsto y_3,y_4\mapsto y_4,y_5\mapsto -y_5,y_6\mapsto -y_6,y_7\mapsto -y_7,y_8\mapsto -y_8.
\end{align*}

Define $z_1=\frac{y_1}{y_4}$, $z_2=\frac{y_2}{y_4}$, 
$z_3=\frac{y_3}{y_4}$, $z_4=\frac{y_5}{y_8}$, $z_5=\frac{y_6}{y_8}$, 
$z_6=\frac{y_7}{y_8}$, $z_7=y_4$, $z_8=y_8$. 
Then $k(y_1,y_2,y_3,y_4,y_5,y_6,$ $y_7,y_8)
=k(z_1,z_2,z_3,z_4,z_5,z_6,z_7,z_8)$ and 

\begin{align*}
g_1&:z_1\mapsto -\tfrac{z_3}{z_2},z_2\mapsto \tfrac{1}{z_2},z_3\mapsto \tfrac{z_1}{z_2},z_4\mapsto \tfrac{z_5}{z_6},z_5\mapsto \tfrac{z_4}{z_6},z_6\mapsto \tfrac{1}{z_6},z_7\mapsto z_2z_7,z_8\mapsto z_6z_8,\\
g_2&:z_1\mapsto \tfrac{z_2}{z_3},z_2\mapsto \tfrac{z_1}{z_3},z_3\mapsto \tfrac{1}{z_3},z_4\mapsto -\tfrac{z_6}{z_5},z_5\mapsto \tfrac{1}{z_5},z_6\mapsto \tfrac{z_4}{z_5},z_7\mapsto z_3z_7,z_8\mapsto z_5z_8,\\
g_3&:z_1\mapsto -z_1,z_2\mapsto -z_2,z_3\mapsto z_3,z_4\mapsto -z_4,z_5\mapsto -z_5,z_6\mapsto z_6,z_7\mapsto z_7,z_8\mapsto z_8,\\
g_4&:z_1\mapsto -z_1,z_2\mapsto z_2,z_3\mapsto -z_3,z_4\mapsto z_4,z_5\mapsto z_5,z_6\mapsto z_6,z_7\mapsto z_7,z_8\mapsto z_8,\\
g_5&:z_1\mapsto z_1,z_2\mapsto z_2,z_3\mapsto z_3,z_4\mapsto -z_4,z_5\mapsto z_5,z_6\mapsto -z_6,z_7\mapsto z_7,z_8\mapsto z_8,\\
g_6&:z_1\mapsto z_1,z_2\mapsto z_2,z_3\mapsto z_3,z_4\mapsto z_4,z_5\mapsto z_5,z_6\mapsto z_6,z_7\mapsto -z_7,z_8\mapsto z_8,\\
g_7&:z_1\mapsto z_1,z_2\mapsto z_2,z_3\mapsto z_3,z_4\mapsto z_4,z_5\mapsto z_5,z_6\mapsto z_6,z_7\mapsto z_7,z_8\mapsto -z_8.
\end{align*}

By applying Theorem \ref{thAHK} twice to $k(z_1,z_2,z_3,z_4,z_5,z_6)(z_7,z_8)$, 
we see that 
the invariant field 
$k(z_1,z_2,z_3,z_4,z_5,z_6,z_7,z_8)^G$ is rational over 
$k(z_1,z_2,z_3,z_4,z_5,z_6)^G$. 
Because $Z(G)=\langle g_6, g_7\rangle$ acts on 
$k(z_1,z_2,z_3,z_4,z_5,z_6)$ trivially, 
$k(z_1,z_2,z_3,z_4,z_5,z_6)^G
=k(z_1,z_2,z_3,z_4,z_5,z_6)^{\langle g_1,g_2,g_3,g_4,g_5\rangle}$. 
Hence it suffices to show that 
$k(z_1,z_2,z_3,z_4,z_5,z_6)^{\langle g_1,g_2,g_3,g_4,g_5\rangle}$ 
is rational over $\varphi(L_k^{(1)})$. 

Define 
$u_1=z_2z_5,
u_2=\frac{z_1}{z_3z_5},
u_3=-\frac{z_3z_6}{z_1z_4},
u_4=\bigl(\frac{z_2+z_1z_3}{z_2-z_1z_3}\bigr)
\frac{z_4z_6}{z_5},
u_5=\frac{z_4z_6}{z_5},
u_6=\frac{z_1+z_2z_3}{z_1-z_2z_3}$.
Note that 
$\frac{u_4^\prime+1}{u_4^\prime-1}=\frac{z_2}{z_1z_3}$ 
where $u_4^\prime=\frac{z_2+z_1z_3}{z_2-z_1z_3}$ 
and 
$\frac{u_6+1}{u_6-1}=\frac{z_1}{z_2z_3}$.
By evaluating the determinant of the matrix $M$ 
of exponents as in Case 1 
(see the equation (\ref{eqdet})), we have $\det M=-8$, 
$k(z_1,z_2,z_3,z_4,z_5,z_6)^{\langle g_3,g_4,g_5\rangle}
=k(u_1,u_2,u_3,u_4,u_5,u_6)$ 
and 
\begin{align*}
g_1&:u_1\mapsto \tfrac{-1}{u_1u_2u_3},u_2\mapsto u_3,u_3\mapsto u_2,u_4\mapsto \tfrac{1}{u_4},u_5\mapsto \tfrac{1}{u_5},
u_6\mapsto \tfrac{-1}{u_6},\\
g_2&:u_1\mapsto u_2,u_2\mapsto u_1,u_3\mapsto \tfrac{-1}{u_1u_2u_3},u_4\mapsto u_4,u_5\mapsto -u_5,u_6\mapsto -u_6.
\end{align*}

This action of $\langle g_1,g_2\rangle$ on $k(u_1,u_2,u_3,u_4,u_5,u_6)$ 
is exactly the same to the equation 
(\ref{act242}) in Case $3$: $G=G(2^7,242)$. 
Hence $k(V_{36})^{G(36)}\simeq k(V_{242})^{G(242)}$ and 
$k(V_{36})^{G(36)}$ is rational over $\varphi(L_k^{(1)})$.\\

{Case $5$: $G=G(2^7,1924)$ which belongs to $\Phi_{43}$.}

$G=\langle g_1,g_2,g_3,g_4,g_5,g_6,g_7\rangle$ with relations 
$g_1^2=1$, $g_2^2=1$, $g_3^2=1$, $g_4^2=1$, $g_5^2=1$, $g_6^2=g_7$, 
$g_7^2=1$, $Z(G)=\langle g_5,g_7\rangle$, 
$[g_2,g_1]=g_5$, $[g_3,g_1]=g_6$, $[g_3,g_2]=g_5g_7$, 
$[g_4,g_1]=g_5$, $[g_6,g_1]=g_7$, $[g_6,g_3]=g_7$.

There exists a faithful representation 
$\rho : G\rightarrow GL(V_{1924})\simeq GL_8(k)$ 
of dimension $8$ 
which is decomposable into two irreducible components 
$V_{1924}\simeq U_4\oplus U_4^\prime$ of dimension $4$.
By Theorem \ref{thHK}, $k(G)$ is rational over $k(V_{1924})^G$. 
We will show that $k(V_{1924})^G$ is rational over $\varphi(L_k^{(1)})$. 

The action of $G$ on $k(V_{1924})=k(y_1,y_2,y_3,y_4,y_5,y_6,y_7,y_8)$ 
is given by 
\begin{align*}
g_1&:y_1\mapsto y_3,y_2\mapsto y_4,y_3\mapsto y_1,y_4\mapsto y_2,y_5\mapsto y_8,y_6\mapsto -\z y_7,y_7\mapsto \z y_6,y_8\mapsto y_5,\\
g_2&:y_1\mapsto y_2,y_2\mapsto y_1,y_3\mapsto -y_4,y_4\mapsto -y_3,y_5\mapsto -y_5,y_6\mapsto y_6,y_7\mapsto y_7,y_8\mapsto -y_8,\\
g_3&:y_1\mapsto -y_1,y_2\mapsto y_2,y_3\mapsto -y_3,y_4\mapsto y_4,y_5\mapsto y_7,y_6\mapsto y_8,y_7\mapsto y_5,y_8\mapsto y_6,\\
g_4&:y_1\mapsto -y_1,y_2\mapsto -y_2,y_3\mapsto y_3,y_4\mapsto y_4,y_5\mapsto y_5,y_6\mapsto y_6,y_7\mapsto y_7,y_8\mapsto y_8,\\
g_5&:y_1\mapsto -y_1,y_2\mapsto -y_2,y_3\mapsto -y_3,y_4\mapsto -y_4,y_5\mapsto y_5,y_6\mapsto y_6,y_7\mapsto y_7,y_8\mapsto y_8,\\
g_6&:y_1\mapsto y_1,y_2\mapsto y_2,y_3\mapsto y_3,y_4\mapsto y_4,y_5\mapsto -\z y_5,y_6\mapsto -\z y_6,y_7\mapsto \z y_7,y_8\mapsto \z y_8,\\
g_7&:y_1\mapsto y_1,y_2\mapsto y_2,y_3\mapsto y_3,y_4\mapsto y_4,y_5\mapsto -y_5,y_6\mapsto -y_6,y_7\mapsto -y_7,y_8\mapsto -y_8.
\end{align*}

Define $z_1=\frac{y_1}{y_4}$, $z_2=\frac{y_2}{y_4}$, 
$z_3=\frac{y_3}{y_4}$, $z_4=\frac{y_5}{y_8}$, $z_5=\frac{y_6}{y_8}$, 
$z_6=\frac{y_7}{y_8}$, $z_7=y_4$, $z_8=y_8$. 
Then $k(y_1,y_2,y_3,y_4,y_5,y_6,$ $y_7,y_8)
=k(z_1,z_2,z_3,z_4,z_5,z_6,z_7,z_8)$ and 

\begin{align*}
g_1&:z_1\mapsto \tfrac{z_3}{z_2},z_2\mapsto \tfrac{1}{z_2},z_3\mapsto \tfrac{z_1}{z_2},z_4\mapsto \tfrac{1}{z_4},z_5\mapsto -\tfrac{\z z_6}{z_4},z_6\mapsto \tfrac{\z z_5}{z_4},z_7\mapsto z_7,z_8\mapsto z_8,\\
g_2&:z_1\mapsto -\tfrac{z_2}{z_3},z_2\mapsto -\tfrac{z_1}{z_3},z_3\mapsto \tfrac{1}{z_3},z_4\mapsto z_4,z_5\mapsto -z_5,z_6\mapsto -z_6,z_7\mapsto z_7,z_8\mapsto z_8,\\
g_3&:z_1\mapsto -z_1,z_2\mapsto z_2,z_3\mapsto -z_3,z_4\mapsto \tfrac{z_6}{z_5},z_5\mapsto \tfrac{1}{z_5},z_6\mapsto \tfrac{z_4}{z_5},z_7\mapsto z_7,z_8\mapsto z_8,\\
g_4&:z_1\mapsto -z_1,z_2\mapsto -z_2,z_3\mapsto z_3,z_4\mapsto z_4,z_5\mapsto z_5,z_6\mapsto z_6,z_7\mapsto z_7,z_8\mapsto z_8,\\
g_5&:z_1\mapsto z_1,z_2\mapsto z_2,z_3\mapsto z_3,z_4\mapsto z_4,z_5\mapsto z_5,z_6\mapsto z_6,z_7\mapsto z_7,z_8\mapsto z_8,\\
g_6&:z_1\mapsto z_1,z_2\mapsto z_2,z_3\mapsto z_3,z_4\mapsto -z_4,z_5\mapsto -z_5,z_6\mapsto z_6,z_7\mapsto z_7,z_8\mapsto z_8,\\
g_7&:z_1\mapsto z_1,z_2\mapsto z_2,z_3\mapsto z_3,z_4\mapsto z_4,z_5\mapsto z_5,z_6\mapsto z_6,z_7\mapsto z_7,z_8\mapsto z_8.
\end{align*}

By applying Theorem \ref{thAHK} twice to $k(z_1,z_2,z_3,z_4,z_5,z_6)(z_7,z_8)$, 
$k(z_1,z_2,z_3,z_4,z_5,z_6,z_7,z_8)^G$ is rational over 
$k(z_1,z_2,z_3,z_4,z_5,z_6)^G$. 
Because $Z(G)=\langle g_5, g_7\rangle$ acts on $k(z_1,z_2,z_3,z_4,z_5,z_6)$ trivially, 
$k(z_1,z_2,z_3,z_4,z_5,z_6)^G
=k(z_1,z_2,z_3,z_4,z_5,z_6)^{\langle g_1,g_2,g_3,g_4,g_6\rangle}$.
Hence it suffices to show that 
$k(z_1,z_2,z_3,z_4,z_5,z_6)^{\langle g_1,g_2,g_3,g_4,g_6\rangle}$ 
is rational over $\varphi(L_k^{(1)})$. 

Define $
u_1=\frac{z_1}{z_2}, 
u_2=\frac{z_1z_2}{z_3}, 
u_3=z_3, 
u_4=\frac{z_4}{z_5}, 
u_5=z_5z_4, 
u_6=\frac{z_6z_4}{z_5}$.
By evaluating the determinant of the matrix $M$ 
of exponents as in Case 1 
(see the equation (\ref{eqdet})), we have $\det M=4$, 
$k(z_1,z_2,z_3,z_4,z_5,z_6)^{\langle g_4,g_6\rangle}
=k(u_1,u_2,u_3,u_4,u_5,u_6)$ 
and 
\begin{align*}
g_1&:u_1\mapsto u_3,u_2\mapsto \tfrac{1}{u_2},u_3\mapsto u_1,u_4\mapsto -\tfrac{u_4}{\z u_6},
u_5\mapsto -\tfrac{\z u_6}{u_4^2u_5},u_6\mapsto -\tfrac{1}{u_6},\\
g_2&:u_1\mapsto \tfrac{1}{u_1},u_2\mapsto u_2,u_3\mapsto \tfrac{1}{u_3},u_4\mapsto -u_4,u_5\mapsto -u_5,u_6\mapsto u_6,\\
g_3&:u_1\mapsto -u_1,u_2\mapsto u_2,u_3\mapsto -u_3,u_4\mapsto \tfrac{u_6}{u_4},u_5\mapsto \tfrac{u_6}{u_5},u_6\mapsto u_6.
\end{align*}

Define 
$v_1=\bigl(\frac{u_1+1}{u_1-1}\bigr)\bigl(\frac{u_3+1}{u_3-1}\bigr)$,
$v_2=u_2$,
$v_3=\bigl(\frac{u_1-1}{u_1+1}\bigr)\bigl(\frac{u_3+1}{u_3-1}\bigr)$,
$v_4=u_4\bigl(\frac{u_1-1}{u_1+1}\bigr)$,
$v_5=-\frac{\z u_6}{u_4}\bigl(\frac{u_3+1}{u_3-1}\bigr)$,
$v_6=-\frac{\z u_6}{u_4u_5}$.
Then 
$k(u_1,u_2,u_3,u_4,u_5,u_6)^{\langle g_2\rangle}
=k(v_1,v_2,v_3,v_4,v_5,v_6)$ and 
\begin{align*}
g_1&:v_1\mapsto v_1,v_2\mapsto \tfrac{1}{v_2},v_3\mapsto \tfrac{1}{v_3},v_4\mapsto \tfrac{1}{v_5},v_5\mapsto \tfrac{1}{v_4},v_6\mapsto \tfrac{1}{v_6},\\
g_3&:v_1\mapsto \tfrac{1}{v_1},v_2\mapsto v_2,v_3\mapsto \tfrac{1}{v_3},v_4\mapsto -\tfrac{v_5}{\z v_3},
v_5\mapsto -\tfrac{\z v_4}{v_3},v_6\mapsto -\tfrac{1}{v_6}.
\end{align*}

Define 
$w_1=\bigl(\frac{v_1+1}{v_1-1}\bigr)\bigl(\frac{v_2+1}{v_2-1}\bigr)\bigl(\frac{v_3+1}{v_3-1}\bigr)$,
$w_2=-\bigl(\frac{v_2+1}{v_2-1}\bigr)\big/\bigl(\frac{v_6+1}{v_6-1}\bigr)$,
$w_3=\bigl(\frac{v_2+1}{v_2-1}\bigr)\bigl(\frac{v_3+1}{v_3-1}\bigr)$,
$w_4=v_4+\frac{1}{v_5}$,
$w_5=\bigl(\frac{v_2+1}{v_2-1}\bigr)\bigl(v_4-\frac{1}{v_5}\bigr)$,
$w_6=\bigl(\frac{v_2+1}{v_2-1}\bigr)\bigl(\frac{v_6+1}{v_6-1}\bigr)$.
Then $k(v_1,v_2,v_3,v_4,v_5,v_6)^{\langle g_1\rangle}
=k(w_1,w_2,w_3,w_4,w_5,w_6)$ and 
\begin{align*}
g_3:\ &w_1\mapsto w_1,w_2\mapsto w_6,w_3\mapsto -w_3,
w_4\mapsto -\tfrac{4w_2w_6(w_3^2w_4-2w_3w_5-w_2w_4w_6)}{\z (w_3^2+w_2w_6)(w_5^2+w_2w_4^2w_6)},\\
&w_5\mapsto -\tfrac{4w_2w_6(w_3^2w_5+2w_2w_3w_4w_6-w_2w_5w_6)}{\z (w_3^2+w_2w_6)(w_5^2+w_2w_4^2w_6)},
w_6\mapsto w_2.
\end{align*}

Define 
\begin{align*}
X_1&=-\tfrac{w_3w_5+w_2w_4w_6}{w_3(w_3w_4-w_5)}, 
X_2=\tfrac{\z (w_3w_5+w_2w_4w_6)}{(w_3w_4-w_5)w_6}, 
X_3=-\tfrac{2\eta (w_3^2w_4-2w_3w_5-w_2w_4w_6)(w_3w_5+w_2w_4w_6)}
{w_3w_4w_6(w_3w_4-w_5)^2},\\
X_4&=-\tfrac{(w_3w_5+w_2w_4w_6)^2}{w_2w_6(w_3w_4-w_5)^2}, 
X_5=w_1, 
X_6=-\tfrac{w_2(w_3w_4-w_5)w_6}{w_3w_5+w_2w_4w_6}.
\end{align*}
It follows from 
$w_1=X_5$, 
$w_2=\z X_2 X_6$, 
$w_3=-\frac{X_4 X_6}{X_1}$, 
$w_4=\frac{2 \z\eta (X_1+1) X_2}{X_3}$, 
$w_5=\frac{2 \z\eta (X_1+1) (X_1 - X_4)X_2X_6}{(X_1-1) X_3}$, 
$w_6=\frac{\z X_4 X_6}{X_2}$
that $k(w_1,w_2,w_3,w_4,w_5,w_6)=k(X_1,X_2,X_3,X_4,X_5,X_6)$ and 
\begin{align*}
g_3:X_1\mapsto -X_1,\ 
X_2\mapsto \tfrac{X_4}{X_2},\ 
X_3\mapsto \tfrac{(X_4-1)(X_4-X_1^2)}{X_3},\ 
X_4\mapsto X_4,\ X_5\mapsto X_5,\ X_6\mapsto X_6.
\end{align*}
The action of $g_3$ on $k(X_1,X_2,X_3,X_4)$ 
and  that of $\tau$ on $k(x_1,x_2,x_3,x_4)$ 
in Definition \ref{defL01} (ii)
are exactly the same. 
Because $g_3$ acts on $k(X_5,X_6)$ trivially, 
we have $k(X_1,X_2,X_3,X_4,X_5,X_6)^{\langle g_3\rangle}
=k(X_1,X_2,X_3,X_4)^{\langle g_3\rangle}(X_5,X_6)$. 
Hence $k(V_{1924})^{G}$ is rational over $\varphi(L_k^{(1)})$. \\

{Case $6$: $G=G(2^7,417)$ which belongs to $\Phi_{58}$.}

$G=\langle g_1,g_2,g_3,g_4,g_5,g_6,g_7\rangle$ with relations 
$g_1^2=1$, $g_2^2=g_4$, $g_3^2=1$ ,$g_4^2=g_6$, $g_5^2=g_7$, $g_6^2=1$, 
$g_7^2=1$,
$Z(G)=\langle g_6,g_7\rangle$, 
$[g_2,g_1]=g_4$, $[g_3,g_1]=g_5$, $[g_3,g_2]=g_6$, 
$[g_4,g_1]=g_6$, $[g_5,g_1]=g_7$, $[g_5,g_3]=g_7$.

There exists a faithful representation 
$\rho : G\rightarrow GL(V_{417})\simeq GL_6(k)$ of dimension $6$
which is decomposable into two irreducible components 
$V_{417}\simeq U_4\oplus U_2$ of dimension $4$ and $2$ respectively.
By Theorem \ref{thHK}, $k(G)$ is rational over $k(V_{417})^G$. 
We will show that $k(V_{417})^G$ is rational over $\varphi(L_k^{(1)})$. 

The action of $G$ on $k(V_{417})=k(y_1,y_2,y_3,y_4,y_5,y_6)$ 
is given by 
\begin{align*}
g_1&:y_1\mapsto y_2,y_2\mapsto y_1,y_3\mapsto y_4,y_4\mapsto y_3,y_5\mapsto \eta^3y_6,y_6\mapsto -\eta y_5,\\
g_2&:y_1\mapsto \eta^3y_1,y_2\mapsto \eta y_2,y_3\mapsto -\eta^3y_3,y_4\mapsto -\eta y_4,y_5\mapsto y_5,y_6\mapsto y_6,\\
g_3&:y_1\mapsto y_3,y_2\mapsto -\z y_4,y_3\mapsto y_1,y_4\mapsto \z y_2,y_5\mapsto y_6,y_6\mapsto y_5,\\
g_4&:y_1\mapsto -\z y_1,y_2\mapsto \z y_2,y_3\mapsto -\z y_3,y_4\mapsto \z y_4,y_5\mapsto y_5,y_6\mapsto y_6,\\
g_5&:y_1\mapsto -\z y_1,y_2\mapsto \z y_2,y_3\mapsto \z y_3,y_4\mapsto -\z y_4,y_5\mapsto -\z y_5,y_6\mapsto \z y_6,\\
g_6&:y_1\mapsto -y_1,y_2\mapsto -y_2,y_3\mapsto -y_3,y_4\mapsto -y_4,y_5\mapsto y_5,y_6\mapsto y_6,\\
g_7&:y_1\mapsto -y_1,y_2\mapsto -y_2,y_3\mapsto -y_3,y_4\mapsto -y_4,y_5\mapsto -y_5,y_6\mapsto -y_6.
\end{align*}

Define $z_1=\frac{y_1}{y_4}$, $z_2=\frac{y_2}{y_4}$, 
$z_3=\frac{y_3}{y_4}$, $z_4=\frac{y_5}{y_6}$, $z_5=y_4$, $z_6=y_6$. 
Then $k(y_1,y_2,y_3,y_4,y_5,y_6)=k(z_1,z_2,z_3,z_4,z_5,z_6)$ and 
\begin{align*}
g_1&:z_1\mapsto \tfrac{z_2}{z_3},z_2\mapsto \tfrac{z_1}{z_3},z_3\mapsto \tfrac{1}{z_3},z_4\mapsto -\tfrac{\z}{z_4},z_5\mapsto z_3z_5,z_6\mapsto -\eta z_4z_6,\\
g_2&:z_1\mapsto -\z z_1,z_2\mapsto -z_2,z_3\mapsto \z z_3,z_4\mapsto z_4,z_5\mapsto -\eta z_5,
z_6\mapsto z_6,\\
g_3&:z_1\mapsto \tfrac{z_3}{\z z_2},z_2\mapsto -\tfrac{1}{z_2},z_3\mapsto \tfrac{z_1}{\z z_2},
z_4\mapsto \tfrac{1}{z_4},z_5\mapsto \z z_2z_5,z_6\mapsto z_4z_6,\\
g_4&:z_1\mapsto -z_1,z_2\mapsto z_2,z_3\mapsto -z_3,z_4\mapsto z_4,z_5\mapsto \z z_5,z_6\mapsto z_6,\\
g_5&:z_1\mapsto z_1,z_2\mapsto -z_2,z_3\mapsto -z_3,z_4\mapsto -z_4,z_5\mapsto -\z z_5,z_6\mapsto \z z_6,\\
g_6&:z_1\mapsto z_1,z_2\mapsto z_2,z_3\mapsto z_3,z_4\mapsto z_4,z_5\mapsto -z_5,z_6\mapsto z_6,\\
g_7&:z_1\mapsto z_1,z_2\mapsto z_2,z_3\mapsto z_3,z_4\mapsto z_4,z_5\mapsto -z_5,z_6\mapsto -z_6.
\end{align*}

Apply Theorem \ref{thAHK} twice to $k(z_1,z_2,z_3,z_4)(z_5,z_6)$, 
the invariant field 
$k(z_1,z_2,z_3,z_4,z_5,z_6)^G$ is rational over 
$k(z_1,z_2,z_3,z_4)^G$. 
We find that 
$k(z_1,z_2,z_3,z_4)^G
=k(z_1,z_2,z_3,z_4)^{\langle g_1,g_2,g_3,g_4,g_5\rangle}$ 
because $Z(G)=\langle g_6, g_7\rangle$ acts on $k(z_1,z_2,z_3,z_4)$ trivially. 
It suffices to show that 
$k(z_1,z_2,z_3,z_4)^{\langle g_1,g_2,g_3,g_4,g_5\rangle}$ 
is rational over  $\varphi(L_k^{(1)})$.

Define $u_1=\frac{z_1^2z_3^2}{z_2^2}$, 
$u_2=\frac{z_2^2z_4}{z_1z_3}$, 
$u_3=\frac{\z z_1^2}{z_2z_4}$, $u_4=\frac{z_1}{z_2z_3}$.
By evaluating the determinant of the matrix $M$ 
of exponents as in Case 1 
(see the equation (\ref{eqdet})), we have $\det M=-8$, 
$k(z_1,z_2,z_3,z_4)^{\langle g_2,g_4,g_5\rangle}=k(u_1,u_2,u_3,u_4)$ 
and 
\begin{align*}
g_1&:u_1\mapsto \tfrac{1}{u_1},u_2\mapsto u_3,u_3\mapsto u_2,u_4\mapsto \tfrac{1}{u_4},\\
g_3&:u_1\mapsto u_1,u_2\mapsto -\tfrac{1}{u_1u_2},u_3\mapsto -\tfrac{u_1}{u_3},
u_4\mapsto -\tfrac{1}{u_4}.
\end{align*}
Hence the action of $\langle g_1,g_3\rangle$ on 
$k(u_1,u_2,u_3,u_4)$ 
and that of 
$\langle g_1,g_2\rangle$ on $k(u_1,u_2,u_3,u_4)$ as in $(\ref{act227})$ of 
Case $1$: $G=G(2^7,227)$ 
are exactly the same. 
Hence $k(V_{417})^{G(417)}\simeq k(V_{227})^{G(227)}$ and 
$k(V_{417})^{G(417)}$ is rational over $\varphi(L_k^{(1)})$. \\

{Case $7$: $G=G(2^7,446)$ which belongs to $\Phi_{60}$.}

$G=\langle g_1,g_2,g_3,g_4,g_5,g_6,g_7\rangle$ with relations 
$g_1^2=1$, $g_2^2=g_4$, $g_3^2=g_5$, $g_4^2=g_6$, 
$g_5^2=g_7$, $g_6^2=1$, $g_7^2=1$,
$Z(G)=\langle g_6,g_7\rangle$, 
$[g_2,g_1]=g_4$, $[g_3,g_1]=g_5$, $[g_3,g_2]=g_6$, 
$[g_4,g_1]=g_6$, $[g_5,g_1]=g_7$.

There exists a faithful representation 
$\rho : G\rightarrow GL(V_{446})\simeq GL_6(k)$ 
of dimension $6$
which is decomposable into two irreducible components 
$V_{446}\simeq U_4\oplus U_2$ of dimension $4$ and $2$ respectively.
By Theorem \ref{thHK}, $k(G)$ is rational over $k(V_{446})^G$. 
We will show that $k(V_{446})^G$ is rational over $\varphi(L_k^{(1)})$. 

The action of $G$ on $k(V_{446})=k(y_1,y_2,y_3,y_4,y_5,y_6)$ 
is given by 
\begin{align*}
g_1&:y_1\mapsto y_3,y_2\mapsto y_4,y_3\mapsto y_1,y_4\mapsto y_2,y_5\mapsto y_6,y_6\mapsto y_5,\\
g_2&:y_1\mapsto -\z y_2,y_2\mapsto y_1,y_3\mapsto -y_4,y_4\mapsto -\z y_3,y_5\mapsto y_5,y_6\mapsto y_6,\\
g_3&:y_1\mapsto -y_1,y_2\mapsto y_2,y_3\mapsto -y_3,y_4\mapsto y_4,y_5\mapsto w^3y_5,y_6\mapsto wy_6,\\
g_4&:y_1\mapsto -\z y_1,y_2\mapsto -\z y_2,y_3\mapsto \z y_3,y_4\mapsto \z y_4,y_5\mapsto y_5,y_6\mapsto y_6,\\
g_5&:y_1\mapsto y_1,y_2\mapsto y_2,y_3\mapsto y_3,y_4\mapsto y_4,y_5\mapsto -\z y_5,y_6\mapsto \z y_6,\\
g_6&:y_1\mapsto -y_1,y_2\mapsto -y_2,y_3\mapsto -y_3,y_4\mapsto -y_4,y_5\mapsto y_5,y_6\mapsto y_6,\\
g_7&:y_1\mapsto y_1,y_2\mapsto y_2,y_3\mapsto y_3,y_4\mapsto y_4,y_5\mapsto -y_5,y_6\mapsto -y_6.
\end{align*}

Define $z_1=\frac{y_1}{y_4}$, $z_2=\frac{y_2}{y_4}$, 
$z_3=\frac{y_3}{y_4}$, $z_4=\frac{y_5}{y_6}$, $z_5=y_4$, $z_6=y_6$. 
Then $k(y_1,y_2,y_3,y_4,y_5,y_6)=k(z_1,z_2,z_3,z_4,z_5,z_6)$ and 
\begin{align*}
g_1&:z_1\mapsto \tfrac{z_3}{z_2},z_2\mapsto \tfrac{1}{z_2},z_3\mapsto \tfrac{z_1}{z_2},
z_4\mapsto \tfrac{1}{z_4},z_5\mapsto z_2z_5,z_6\mapsto z_4z_6,\\
g_2&:z_1\mapsto \tfrac{z_2}{z_3},z_2\mapsto -\tfrac{z_1}{\z z_3},z_3\mapsto \tfrac{1}{\z z_3}
,z_4\mapsto z_4,z_5\mapsto -\z z_3z_5,z_6\mapsto z_6,\\
g_3&:z_1\mapsto -z_1,z_2\mapsto z_2,z_3\mapsto -z_3,z_4\mapsto \z z_4,z_5\mapsto z_5,z_6\mapsto \eta z_6,\\
g_4&:z_1\mapsto -z_1,z_2\mapsto -z_2,z_3\mapsto z_3,z_4\mapsto z_4,z_5\mapsto \z z_5,z_6\mapsto z_6,\\
g_5&:z_1\mapsto z_1,z_2\mapsto z_2,z_3\mapsto z_3,z_4\mapsto -z_4,z_5\mapsto z_5,z_6\mapsto \z z_6,\\
g_6&:z_1\mapsto z_1,z_2\mapsto z_2,z_3\mapsto z_3,z_4\mapsto z_4,z_5\mapsto -z_5,z_6\mapsto z_6,\\
g_7&:z_1\mapsto z_1,z_2\mapsto z_2,z_3\mapsto z_3,z_4\mapsto z_4,z_5\mapsto z_5,z_6\mapsto -z_6.
\end{align*}

Apply Theorem \ref{thAHK} twice to $k(z_1,z_2,z_3,z_4)(z_5,z_6)$, 
the invariant field 
$k(z_1,z_2,z_3,z_4,z_5,z_6)^G$ is rational over 
$k(z_1,z_2,z_3,z_4)^G$. 
We find that 
$k(z_1,z_2,z_3,z_4)^G
=k(z_1,z_2,z_3,z_4)^{\langle g_1,g_2,g_3,g_4,g_5\rangle}$ 
because $Z(G)=\langle g_6, g_7\rangle$ acts on $k(z_1,z_2,z_3,z_4)$ trivially. 
It suffices to show that 
$k(z_1,z_2,z_3,z_4)^{\langle g_1,g_2,g_3,g_4,g_5\rangle}$ 
is rational over  $\varphi(L_k^{(1)})$.

Define $u_1=z_4^4$, $u_2=\frac{\eta z_2}{z_1z_4^2}$, 
$u_3=\frac{\eta z_4^2}{z_3}$, $u_4=\frac{\z z_1z_2+z_3}{z_1z_2+\z z_3}$.
Note that $\frac{u_4+\z}{u_4-\z}=\frac{\z z_1z_2}{z_3}$. 
By evaluating the determinant of the matrix $M$ 
of exponents as in Case 1 
(see the equation (\ref{eqdet})), we have $\det M=8$, 
$k(z_1,z_2,z_3,z_4)^{\langle g_3,g_4,g_5\rangle}=k(u_1,u_2,u_3,u_4)$ 
and 
\begin{align*}
g_1&:u_1\mapsto \tfrac{1}{u_1},u_2\mapsto u_3,u_3\mapsto u_2,u_4\mapsto \tfrac{1}{u_4},\\
g_2&:u_1\mapsto u_1,u_2\mapsto -\tfrac{1}{u_1u_2},u_3\mapsto -\tfrac{u_1}{u_3},
u_4\mapsto -\tfrac{1}{u_4}.
\end{align*}

Hence the action of $\langle g_1,g_2\rangle$ on 
$k(u_1,u_2,u_3,u_4)$ 
and that of 
$\langle g_1,g_2\rangle$ on $k(u_1,u_2,u_3,u_4)$ as in 
$(\ref{act227})$ of Case $1$: $G=G(2^7,227)$ 
are exactly the same. 
Hence $k(V_{446})^{G(446)}\simeq k(V_{227})^{G(227)}$ and 
$k(V_{446})^{G(446)}$ is rational over $\varphi(L_k^{(1)})$. \\

{Case $8$: $G=G(2^7,950)$ which belongs to $\Phi_{80}$.}

$G=\langle g_1,g_2,g_3,g_4,g_5,g_6,g_7\rangle$ with relations 
$g_1^2=1$, $g_2^2=g_4g_6$, 
$g_3^2=1$, $g_4^2=g_6g_7$, 
$g_5^2=1$, $g_6^2=g_7$, $g_7^2=1$, 
$Z(G)=\langle g_5,g_7\rangle$, 
$[g_2,g_1]=g_4$, $[g_3,g_1]=g_5$, $[g_3,g_2]=g_7$, 
$[g_4,g_1]=g_6$, $[g_6,g_1]=g_7$. 

There exists a faithful representation 
$\rho : G\rightarrow GL(V_{950})\simeq GL_6(k)$ 
of dimension $6$ 
which is decomposable into two irreducible components 
$V_{950}\simeq U_4\oplus U_2$ of dimension $4$ and $2$ respectively. 
By Theorem \ref{thHK}, $k(G)$ is rational over $k(V_{950})^G$. 
We will show that $k(V_{950})^G$ is rational over $\varphi(L_k^{(1)})$. 

The action of $G$ on $k(V_{950})=k(y_1,y_2,y_3,y_4,y_5,y_6)$ 
is given by 
\begin{align*}
g_1&:y_1\mapsto y_3,y_2\mapsto y_4,y_3\mapsto y_1,y_4\mapsto y_2,y_5\mapsto y_6,y_6\mapsto y_5,\\
g_2&:y_1\mapsto \eta^3y_2,y_2\mapsto y_1,y_3\mapsto y_4,y_4\mapsto -\eta y_3,y_5\mapsto y_5,y_6\mapsto y_6,\\
g_3&:y_1\mapsto -y_1,y_2\mapsto y_2,y_3\mapsto -y_3,y_4\mapsto y_4,y_5\mapsto -y_5,y_6\mapsto y_6,\\
g_4&:y_1\mapsto -\eta y_1,y_2\mapsto -\eta y_2,y_3\mapsto \eta^3y_3,y_4\mapsto \eta^3y_4,y_5\mapsto y_5,y_6\mapsto y_6,\\
g_5&:y_1\mapsto y_1,y_2\mapsto y_2,y_3\mapsto y_3,y_4\mapsto y_4,y_5\mapsto -y_5,y_6\mapsto -y_6,\\
g_6&:y_1\mapsto -\z y_1,y_2\mapsto -\z y_2,y_3\mapsto \z y_3,y_4\mapsto \z y_4,y_5\mapsto y_5,y_6\mapsto y_6,\\
g_7&:y_1\mapsto -y_1,y_2\mapsto -y_2,y_3\mapsto -y_3,y_4\mapsto -y_4,y_5\mapsto y_5,y_6\mapsto y_6.
\end{align*}

Define $z_1=\frac{y_1}{y_4}$, $z_2=\frac{y_2}{y_4}$, 
$z_3=\frac{y_3}{y_4}$, $z_4=\frac{y_5}{y_6}$, $z_5=y_4$, $z_6=y_6$. 
Then $k(y_1,y_2,y_3,y_4,y_5,y_6)=k(z_1,z_2,z_3,z_4,z_5,z_6)$ and 
\begin{align*}
g_1&:z_1\mapsto \tfrac{z_3}{z_2},z_2\mapsto \tfrac{1}{z_2},z_3\mapsto \tfrac{z_1}{z_2},z_4\mapsto \tfrac{1}{z_4},
z_5\mapsto z_2z_5,z_6\mapsto z_4z_6,\\
g_2&:z_1\mapsto -\tfrac{\z z_2}{z_3},z_2\mapsto -\tfrac{z_1}{\eta z_3},
z_3\mapsto -\tfrac{1}{\eta z_3},
z_4\mapsto z_4,z_5\mapsto -\eta z_3z_5,z_6\mapsto z_6,\\
g_3&:z_1\mapsto -z_1,z_2\mapsto z_2,z_3\mapsto -z_3,z_4\mapsto -z_4,z_5\mapsto z_5,z_6\mapsto z_6,\\
g_4&:z_1\mapsto -\tfrac{z_1}{\z},z_2\mapsto -\tfrac{z_2}{\z},z_3\mapsto z_3,z_4\mapsto z_4,z_5\mapsto \eta^3z_5,z_6\mapsto z_6,\\
g_5&:z_1\mapsto z_1,z_2\mapsto z_2,z_3\mapsto z_3,z_4\mapsto z_4,z_5\mapsto z_5,z_6\mapsto -z_6,\\
g_6&:z_1\mapsto -z_1,z_2\mapsto -z_2,z_3\mapsto z_3,z_4\mapsto z_4,z_5\mapsto \z z_5,z_6\mapsto z_6,\\
g_7&:z_1\mapsto z_1,z_2\mapsto z_2,z_3\mapsto z_3,z_4\mapsto z_4,z_5\mapsto -z_5,z_6\mapsto z_6.
\end{align*}

Apply Theorem \ref{thAHK} twice to $k(z_1,z_2,z_3,z_4)(z_5,z_6)$, 
the invariant field 
$k(z_1,z_2,z_3,z_4,z_5,z_6)^G$ is rational over 
$k(z_1,z_2,z_3,z_4)^G$. 
We find that 
$k(z_1,z_2,z_3,z_4)^G
=k(z_1,z_2,z_3,z_4)^{\langle g_1,g_2,g_3,g_4,g_6\rangle}$  
because $Z(G)=\langle g_5, g_7\rangle$ acts on $k(z_1,z_2,z_3,z_4)$ trivially. 
It suffices to show that 
$k(z_1,z_2,z_3,z_4)^{\langle g_1,g_2,g_3,g_4,g_6\rangle}$ 
is rational over  $\varphi(L_k^{(1)})$.

Define $u_1=z_4^2$, $u_2=\frac{\omega z_1}{z_2z_4}$, 
$u_3=\omega z_3z_4$, 
$u_4=\frac{\z z_1^2z_2^2-z_3^2}{z_1^2z_2^2-\z z_3^2}$. 
Note that $\frac{u_4+\z}{u_4-\z}=-\frac{\z z_1^2z_2^2}{z_3^2}$. 
By evaluating the determinant of the matrix $M$ 
of exponents as in Case 1 
(see the equation (\ref{eqdet})), we have $\det M=8$, 
$k(z_1,z_2,z_3,z_4)^{\langle g_3,g_4,g_6\rangle}=k(u_1,u_2,u_3,u_4)$ 
and 
\begin{align*}
g_1&:u_1\mapsto \tfrac{1}{u_1},u_2\mapsto u_3,u_3\mapsto u_2,u_4\mapsto \tfrac{1}{u_4},\\
g_2&:u_1\mapsto u_1,u_2\mapsto -\tfrac{1}{u_1u_2},u_3\mapsto -\tfrac{u_1}{u_3},
u_4\mapsto -\tfrac{1}{u_4}.
\end{align*}

Hence the action of $\langle g_1,g_2\rangle$ on 
$k(u_1,u_2,u_3,u_4)$ and 
that of 
$\langle g_1,g_2\rangle$ on $k(u_1,u_2,u_3,u_4)$ as in 
$(\ref{act227})$ of Case $1$: $G=G(2^7,227)$
are exactly the same. 
Hence $k(V_{950})^{G(950)}\simeq k(V_{227})^{G(227)}$ and 
$k(V_{950})^{G(950)}$ is rational over $\varphi(L_k^{(1)})$.\\

{Case $9$: $G=G(2^7,144)$ which belongs to $\Phi_{106}$.}

$G=\langle g_1,g_2,g_3,g_4,g_5,g_6,g_7\rangle$ with relations 
$g_1^2=g_4$, $g_2^2=g_6$, 
$g_3^2=g_6g_7$, $g_4^2=1$, 
$g_5^2=g_7$, $g_6^2=1$, $g_7^2=1$, 
$Z(G)=\langle g_7\rangle$,
$[g_2,g_1]=g_3$, $[g_3,g_1]=g_5$, $[g_3,g_2]=g_6$, 
$[g_4,g_2]=g_5g_6$, $[g_4,g_3]=g_6g_7$, $[g_5,g_1]=g_6$, 
$[g_5,g_2]=g_7$, $[g_5,g_4]=g_7$, $[g_6,g_1]=g_7$. 

Because the center $Z(G)$ of $G$ is cyclic group of order two, 
there exists a faithful irreducible representation 
$\rho : G\rightarrow GL(V_{144})\simeq GL_8(k)$ 
of dimension $8$. 
By Theorem \ref{thHK}, $k(G)$ is rational over $k(V_{144})^G$. 
We will show that $k(V_{144})^G$ is rational over $\varphi(L_k^{(2)})$. 

The action of $G$ on $k(V_{144})=k(y_1,y_2,y_3,y_4,y_5,y_6,y_7,y_8)$ 
is given by 
\begin{align*}
g_1&:y_1 \mapsto y_2, y_2 \mapsto y_3, y_3 \mapsto y_6, y_4 \mapsto \z y_5, y_5 \mapsto y_8, y_6 \mapsto y_1, 
 y_7 \mapsto y_4, y_8 \mapsto -\z y_7,\\
g_2&:y_1 \mapsto -y_4, y_2 \mapsto \z y_5, 
y_3 \mapsto -y_8, y_4 \mapsto y_1, y_5 \mapsto -\z y_2, y_6 \mapsto -\z y_7,
  y_7 \mapsto \z y_6, y_8 \mapsto y_3,\\
g_3&:y_1 \mapsto -y_1, y_2 \mapsto \z y_2, y_3 \mapsto -y_3, y_4 \mapsto y_4, 
y_5 \mapsto \z y_5, y_6 \mapsto -\z y_6, 
 y_7 \mapsto -\z y_7, y_8 \mapsto y_8,\\
g_4&:y_1 \mapsto y_3, y_2 \mapsto y_6, y_3 \mapsto y_1, y_4 \mapsto \z y_8, 
y_5 \mapsto -\z y_7, y_6 \mapsto y_2, 
 y_7 \mapsto \z y_5, y_8 \mapsto -\z y_4,\\
g_5&:y_1 \mapsto -\z y_1, y_2 \mapsto \z y_2, y_3 \mapsto \z y_3, y_4 \mapsto \z y_4, y_5 \mapsto -\z y_5, 
 y_6 \mapsto -\z y_6, y_7 \mapsto \z y_7, y_8 \mapsto -\z y_8,\\
g_6&:y_1 \mapsto -y_1, y_2 \mapsto y_2, y_3 \mapsto -y_3, y_4 \mapsto -y_4, y_5 \mapsto y_5, y_6 \mapsto y_6, 
 y_7 \mapsto y_7, y_8 \mapsto -y_8,\\
g_7&:y_1 \mapsto -y_1, y_2 \mapsto -y_2, y_3 \mapsto -y_3, y_4 \mapsto -y_4, y_5 \mapsto -y_5, y_6 \mapsto -y_6, 
 y_7 \mapsto -y_7, y_8 \mapsto -y_8.
\end{align*}

Define $z_1=\frac{y_1}{y_8}$, $z_2=\frac{y_2}{y_8}$, 
$z_3=\frac{y_3}{y_8}$, $z_4=\frac{y_4}{y_8}$, $z_5=\frac{y_5}{y_8}$, 
$z_6=\frac{y_6}{y_8}$, $z_7=\frac{y_7}{y_8}$. 
Then we obtain that $k(y_1,y_2,y_3,y_4,y_5,y_6,y_7,y_8)$
$=$ $k(z_1,z_2,z_3,z_4,z_5,z_6,z_7,z_8)$ and 
\begin{align*}
g_1&:z_1 \mapsto -\tfrac{z_2}{\z z_7}, 
z_2 \mapsto -\tfrac{z_3}{\z z_7}, z_3 \mapsto -\tfrac{z_6}{\z z_7}, 
 z_4 \mapsto -\tfrac{z_5}{z_7}, z_5 \mapsto -\tfrac{1}{\z z_7}, 
z_6 \mapsto -\tfrac{z_1}{\z z_7}, 
 z_7 \mapsto -\tfrac{z_4}{\z z_7}, 
z_8 \mapsto -\z z_7 z_8,\\
g_2&:z_1 \mapsto -\tfrac{z_4}{z_3}, z_2 \mapsto \tfrac{\z z_5}{z_3}, 
z_3 \mapsto -\tfrac{1}{z_3}, z_4 \mapsto \tfrac{z_1}{z_3}, 
 z_5 \mapsto -\tfrac{\z z_2}{z_3}, z_6 \mapsto -\tfrac{\z z_7}{z_3}, 
z_7 \mapsto \tfrac{\z z_6}{z_3}, z_8 \mapsto z_3 z_8,\\
g_3&:z_1 \mapsto -z_1, z_2 \mapsto \z z_2, z_3 \mapsto -z_3, z_4 \mapsto z_4, z_5 \mapsto \z z_5, z_6 \mapsto -\z z_6, 
 z_7 \mapsto -\z z_7, z_8 \mapsto z_8,\\
g_4&:z_1 \mapsto -\tfrac{z_3}{\z z_4}, z_2 \mapsto -\tfrac{z_6}{\z z_4}, 
z_3 \mapsto -\tfrac{z_1}{\z z_4}, 
z_4 \mapsto -\tfrac{1}{z_4}, z_5 \mapsto \tfrac{z_7}{z_4}, 
z_6 \mapsto -\tfrac{z_2}{\z z_4}, z_7 \mapsto -\tfrac{z_5}{z_4}, 
 z_8 \mapsto -\z z_4 z_8,\\
g_5&:z_1 \mapsto z_1, z_2 \mapsto -z_2, z_3 \mapsto -z_3, z_4 \mapsto -z_4, z_5 \mapsto z_5, z_6 \mapsto z_6, 
 z_7 \mapsto -z_7, z_8 \mapsto -\z z_8,\\
g_6&:z_1 \mapsto z_1, z_2 \mapsto -z_2, z_3 \mapsto z_3, z_4 \mapsto z_4, z_5 \mapsto -z_5, z_6 \mapsto -z_6, 
 z_7 \mapsto -z_7, z_8 \mapsto -z_8,\\
g_7&:z_1 \mapsto z_1, z_2 \mapsto z_2, z_3 \mapsto z_3, z_4 \mapsto z_4, z_5 \mapsto z_5, z_6 \mapsto z_6, z_7 \mapsto z_7,
  z_8 \mapsto -z_8.
\end{align*}

Apply Theorem \ref{thAHK} to $k(z_1,z_2,z_3,z_4,z_5,z_6,z_7)(z_8)$, 
$k(z_1,z_2,z_3,z_4,z_5,z_6,z_7,z_8)^G$ is rational over 
$k(z_1,z_2,z_3,z_4,z_5,z_6,z_7)^G$. 
Because $Z(G)=\langle g_7\rangle$ acts on $k(z_1,z_2,z_3,z_4,z_5,z_6,z_7)$ 
trivially, we have 
$k(z_1,z_2,z_3,z_4,z_5,z_6,z_7)^G
=k(z_1,z_2,z_3,z_4,z_5,z_6,z_7)^{\langle g_1,g_2,g_3,g_4,g_5,g_6\rangle}$. 
It suffices to show that 
$k(z_1,z_2,z_3,z_4,z_5,z_6,z_7)^{\langle g_1,g_2,g_3,g_4,g_5,g_6\rangle}$ 
is rational over  $\varphi(L_k^{(2)})$.

Define 
$u_1 = -\frac{z_2 z_7}{z_5 z_6}$, 
$u_2=-\frac{z_3}{z_2 z_5}$, 
$u_3=\frac{z_3 z_4}{z_1}$, 
$u_4=-\frac{z_7}{z_4 z_6}$,
$u_5=\frac{z_5}{z_3 z_7}$, 
$u_6=\frac{z_2 z_3}{z_6}$, 
$u_7=\frac{z_3 z_6}{z_1 z_7}$. 
By evaluating the determinant of the matrix $M$ 
of exponents as in Case 1 
(see the equation (\ref{eqdet})), we have $\det M=-8$, 
$k(z_1,z_2,z_3,z_4,z_5,z_6,z_7)^{\langle g_3,g_5,g_6\rangle}=k(u_1,u_2,u_3,u_4,u_5,u_6,u_7)$ 
and 
\begin{align*}
g_1&:u_1 \mapsto -u_3, 
u_2 \mapsto -\tfrac{\z}{u_2 u_5 u_6}, 
u_3 \mapsto \tfrac{1}{u_1}, 
u_4 \mapsto -\tfrac{u_1 u_3}{u_6}, 
 u_5 \mapsto \z u_4, u_6 \mapsto -\tfrac{u_7}{\z}, u_7 \mapsto -\tfrac{1}{u_1 u_3 u_5},\\
g_2&:u_1 \mapsto \tfrac{1}{u_1}, u_2 \mapsto -u_2, u_3 \mapsto \tfrac{1}{u_3}, 
u_4 \mapsto u_7, u_5 \mapsto u_6, u_6 \mapsto u_5, u_7 \mapsto u_4,\\
g_4&:u_1 \mapsto -\tfrac{1}{u_1}, u_2 \mapsto -\tfrac{u_2 u_5 u_6}{u_4 u_7}, 
u_3 \mapsto -\tfrac{1}{u_3}, u_4 \mapsto -\tfrac{\z u_3}{u_1 u_7}, 
u_5 \mapsto -\tfrac{\z u_1 u_3}{u_6}, u_6 \mapsto \tfrac{1}{\z u_1 u_3 u_5},
  u_7 \mapsto \tfrac{u_1}{\z u_3 u_4}.
\end{align*}

Note that $g_1^2=g_4$. 
Thus we will omit the presentation of the action $g_4$.
Define 
$v_1=-\bigl(\frac{u_1 + 1}{u_1 - 1}\bigr)\big/\bigl(\frac{u_3 + 1}{u_3 - 1}\bigr)$, 
$v_2=u_2\bigl(\frac{u_1 + 1}{u_1 - 1}\bigr)$,
$v_3=\bigl(\frac{u_3 + 1}{u_3 - 1}\bigr)\bigl(\frac{u_1 + 1}{u_1 - 1}\bigr)$,
$v_4=u_4 + u_7$, 
$v_5=u_5 + u_6$, 
$v_6=\bigl(\frac{u_4 - u_7}{u_4 +u_7}\bigr)\bigl(\frac{u_1 + 1}{u_1 - 1}\bigr)$, 
$v_7=\bigl(\frac{u_5 - u_6}{u_5 + u_6}\bigr)\bigl(\frac{u_1 + 1}{u_1 - 1}\bigr)$.
Then $k(u_1,u_2,u_3,u_4,u_5,u_6,u_7)^{\langle g_2\rangle}=k(v_1,v_2,v_3,v_4,v_5,v_6,v_7)$ 
and 
\begin{align*}
g_1:\ &v_1 \mapsto \tfrac{1}{v_3}, 
v_2 \mapsto \tfrac{4 \z v_1^2 v_3}{v_2 v_5^2 (v_1 v_3 + v_7^2)}, 
v_3 \mapsto v_1, 
v_4 \mapsto -\tfrac{4 v_1 v_3 (v_1 - v_3 + 4 v_1 v_3 - v_1^2 v_3 + v_1 v_3^2 - 2 v_7 + 
2 v_1 v_7 - 2 v_3 v_7 + 2 v_1 v_3 v_7)}{(v_1 + v_3) (v_1 v_3+1) v_5 (v_1 v_3 + v_7^2)},\\ 
\ &v_5 \mapsto \z v_4, 
v_6 \mapsto \tfrac{2 v_1 v_3 - 2 v_1^2 v_3 + 2 v_1 v_3^2 - 2 v_1^2 v_3^2 + v_1 v_7 - 
    v_3 v_7 + 4 v_1 v_3 v_7 - v_1^2 v_3 v_7 + v_1 v_3^2 v_7}{v_3 (v_1 - v_3 + 4 v_1 v_3 - v_1^2 v_3 + v_1 v_3^2 - 2 v_7 + 2 v_1 v_7 - 
      2 v_3 v_7 + 2 v_1 v_3 v_7)}, 
v_7 \mapsto \tfrac{v_6}{v_3}.
\end{align*}

Define 
\begin{align*}
w_1&=\tfrac{v_1 + 1}{v_1 - 1}, 
w_2=\tfrac{v_2 v_5}{v_4}, 
w_3=\tfrac{v_3 + 1}{v_3 - 1}, 
w_4=v_4, 
w_5=v_5,\\
w_6&=-\tfrac{v_3 + 2 v_1 v_3 + v_1^2 v_3 + v_6 + v_1 v_6 - v_3 v_6 - v_1 v_3 v_6}{(v_3-1) (-v_1 + v_1 v_3 + v_6 + v_1 v_6)}, 
w_7=\tfrac{v_1 + 2 v_1 v_3 + 
 v_1 v_3^2 - v_7 + v_1 v_7 - v_3 v_7 + 
 v_1 v_3 v_7}{(v_1-1) (-v_3 + v_1 v_3 - v_7 - v_3 v_7)}.
\end{align*}
Then it follows from 
\begin{align*}
v_1 = \tfrac{w_1 + 1}{w_1 - 1}, 
v_2 = \tfrac{w_2 w_4}{w_5}, 
v_3 = \tfrac{w_3+1}{w_3-1}, 
v_4 = w_4, 
v_5 = w_5, 
v_6 = \tfrac{w_6 - w_1^2 (w_3^2 + w_6-1)}{w_1(w_1-1)(w_3-1)(w_6-1)}, 
v_7 = -\tfrac{w_7+w_3^2(w_1^2 - w_7-1)}{w_3(w_1-1)(w_3-1)(w_7+1)}
\end{align*}
that $k(v_1,v_2,v_3,v_4,v_5,v_6,v_7)=k(w_1,w_2,w_3,w_4,w_5,w_6,w_7)$ and 
\begin{align*}
g_1:&\ w_1 \mapsto -w_3, w_2 \mapsto \tfrac{(w_1+1)(w_7+1)}{(w_1-1)w_2(w_7-1)}, 
 w_3 \mapsto w_1,\\ 
&\ w_4 \mapsto -\tfrac{4 (w_1^2-1)(w_3^2-1) w_3^2 (w_7^2-1)}{(w_1^2 w_3^2-1)w_5 (-w_3^2 + w_1^2 w_3^2 - w_7^2 + w_3^2 w_7^2)}, 
w_5 \mapsto \z w_4, w_6 \mapsto w_7, w_7 \mapsto -w_6.
\end{align*}
We also define 
\begin{align*}
p_1&=w_1, p_2=-w_3, p_3=w_6, p_4=w_7, p_5=(w_1-1)w_2(w_7-1),\\ 
p_6&=-\tfrac{(w_1^2 w_3^2-1) w_5 (w_3^2 - w_1^2 w_3^2+ w_7^2-w_3^2 w_7^2)}
{(w_1-1)w_3}, 
p_7=\tfrac{(w_1^2 w_3^2-1) w_4 (w_1^2 - w_1^2 w_3^2 + w_6^2-w_1^2 w_6^2)\z}
{w_1 (w_3+1)}.
\end{align*}
Then it follows from 
\begin{align*}
w_1 &= p_1,
w_2 = \tfrac{p_5}{(p_1-1) (p_4-1)}, 
w_3 = -p_2,
w_4 = \tfrac{p_1 (p_2-1) p_7}{\z (p_1^2 p_2^2-1) (p_1^2 (p_2^2+p_3^2-1)-p_3^2)},\\ 
w_5 &= -\tfrac{(p_1-1) p_2 p_6}{(p_1^2 p_2^2-1) (p_2^2 (p_1^2 + p_4^2-1)-p_4^2)}, 
w_6 = p_3, 
w_7 = p_4
\end{align*}
that $k(w_1,w_2,w_3,w_4,w_5,w_6,w_7)=k(p_1,p_2,p_3,p_4,p_5,p_6,p_7)$ and 
\begin{align*}
g_1:\ &p_1 \mapsto p_2, p_2 \mapsto -p_1, p_3 \mapsto p_4, p_4 \mapsto -p_3, 
 p_5 \mapsto -\tfrac{(p_1+1)(p_2-1)(p_3+1)(p_4+1)}{p_5},\\ 
\ &p_6 \mapsto p_7, 
p_7 \mapsto -\tfrac{4\z (p_2^2-1)(p_1^2p_2^2-1)(p_4^2-1)
(p_1^2 p_2^2-p_2^2 + p_2^2 p_4^2- p_4^2)}{p_6}.
\end{align*}

Define 
\begin{align*}
X_1&=p_2, 
X_2=-p_1, 
X_3=\tfrac{p_4}{p_2}, 
X_4=\tfrac{p_3}{p_1},
X_5=\tfrac{p_6}{2 \eta (p_2 + 1) (p_4 + 1) p_2},\\
X_6&=-\tfrac{p_7}{2 \eta (p_1-1) (p_3-1) p_1}, 
X_7=\tfrac{p_1 p_2(p_2 - p_3 + p_2 p_3 + p_5-1)}{p_2 - p_3 + p_2 p_3 - p_5-1}. 
\end{align*}
Then it follows from 
\begin{align*}
p_1 &= -X_2, 
p_2 = X_1, 
p_3 = -X_2 X_4, 
p_4 = X_1 X_3, 
p_5 = \tfrac{(X_1-1)(X_2 X_4-1)(X_1 X_2 + X_7)}{X_1 X_2 - X_7},\\
p_6 &= 2\eta X_1(X_1+1)(X_1 X_3+1) X_5, 
p_7 = 2\eta X_2(X_2+1)(X_2 X_4+1) X_6
\end{align*}
that $k(p_1,p_2,p_3,p_4,p_5,p_6,p_7)=k(X_1,X_2,X_3,X_4,X_5,X_6,X_7)$. 
We see that 
the action of $\langle g_1\rangle$ on $k(X_1,X_2,X_3,X_4,X_5,X_6)$ 
and 
that of $\langle\rho\rangle$ in Definition \ref{defL23} (i) 
are exactly the same. 
We also have 
$k(X_1,X_2,X_3,X_4,X_5,X_6,X_7)^{\langle g_1\rangle}=
k(X_1,X_2,X_3,X_4,X_5,X_6)^{\langle g_1\rangle}(X_7)$ 
because $X_7$ is an invariant under the action of $\langle g_1\rangle$.
Hence $k(V_{144})^G$ is rational over $\varphi(L_k^{(2)})$.\\

{Case $10$: $G=G(2^7,138)$ which belongs to $\Phi_{114}$.}

$G=\langle g_1,g_2,g_3,g_4,g_5,g_6,g_7\rangle$ with relations 
$g_1^2=g_4$, $g_2^2=1$, 
$g_3^2=g_6$, $g_4^2=1$, 
$g_5^2=g_7$, $g_6^2=1$, $g_7^2=1$, 
$Z(G)=\langle g_7\rangle$,
$[g_2,g_1]=g_3$, $[g_3,g_1]=g_5$, $[g_3,g_2]=g_6$, 
$[g_4,g_2]=g_5g_6g_7$, $[g_4,g_3]=g_6g_7$, $[g_5,g_1]=g_6$, 
$[g_5,g_2]=g_7$, $[g_5,g_4]=g_7$, $[g_6,g_1]=g_7$.

Because the center $Z(G)$ of $G$ is cyclic group of order two, 
there exists a faithful irreducible representation 
$\rho : G\rightarrow GL(V_{138})\simeq GL_8(k)$ 
of dimension $8$. 
By Theorem \ref{thHK}, $k(G)$ is rational over $k(V_{138})^G$. 
We will show that $k(V_{138})^G$ is rational over $\varphi(L_k^{(2)})$. 

The action of $G$ on $k(V_{138})=k(y_1,y_2,y_3,y_4,y_5,y_6,y_7,y_8)$ 
is given by 
\begin{align*}
g_1&:y_1\mapsto y_5, y_2\mapsto y_1, y_3\mapsto \z y_7, 
y_4\mapsto y_2, y_5\mapsto y_4, y_6\mapsto y_3, y_7\mapsto y_8, y_8\mapsto -\z y_6,\\
g_2&:y_1\mapsto y_3, y_2\mapsto -\z y_6, y_3\mapsto y_1, y_4\mapsto y_8, 
y_5\mapsto -\z y_7, y_6\mapsto \z y_2,  y_7\mapsto \z y_5, y_8\mapsto y_4,\\
g_3&:y_1\mapsto -y_1, y_2\mapsto \z y_2, y_3\mapsto -y_3, y_4\mapsto y_4, 
y_5\mapsto \z y_5, y_6\mapsto -\z y_6, y_7\mapsto -\z y_7, y_8\mapsto y_8,\\
g_4&:y_1\mapsto y_4, y_2\mapsto y_5, y_3\mapsto \z y_8, y_4\mapsto y_1, 
y_5\mapsto y_2, y_6\mapsto \z y_7, y_7\mapsto -\z y_6, y_8\mapsto -\z y_3,\\
g_5&:y_1\mapsto -\z y_1, y_2\mapsto \z y_2, y_3\mapsto \z y_3, y_4\mapsto \z y_4, 
y_5\mapsto -\z y_5, y_6\mapsto -\z y_6, y_7\mapsto \z y_7, y_8\mapsto -\z y_8,\\
g_6&:y_1\mapsto y_1, y_2\mapsto -y_2, y_3\mapsto y_3, y_4\mapsto y_4, 
y_5\mapsto -y_5, y_6\mapsto -y_6, y_7\mapsto -y_7, y_8\mapsto y_8,\\
g_7&:y_1\mapsto -y_1, y_2\mapsto -y_2, y_3\mapsto -y_3, y_4\mapsto -y_4, 
y_5\mapsto -y_5, y_6\mapsto -y_6, y_7\mapsto -y_7, y_8\mapsto -y_8.
\end{align*}

Define $z_1=\frac{y_1}{y_8}$, $z_2=\frac{y_2}{y_8}$, 
$z_3=\frac{y_3}{y_8}$, $z_4=\frac{y_4}{y_8}$, $z_5=\frac{y_5}{y_8}$, 
$z_6=\frac{y_6}{y_8}$, $z_7=\frac{y_7}{y_8}$. 
Then we obtain that $k(y_1,y_2,y_3,y_4,y_5,y_6,y_7,y_8)$
$=$ $k(z_1,z_2,z_3,z_4,z_5,z_6,z_7,z_8)$ and 
\begin{align*}
g_1&:z_1\mapsto - \tfrac{z_5}{\z z_6}, z_2\mapsto - \tfrac{z_1}{\z z_6}, 
z_3\mapsto - \tfrac{z_7}{z_6}, z_4\mapsto - \tfrac{z_2}{\z z_6}, 
z_5\mapsto - \tfrac{z_4}{\z z_6}, z_6\mapsto - \tfrac{z_3}{\z z_6}, 
 z_7\mapsto - \tfrac{1}{\z z_6}, z_8\mapsto -\z z_6 z_8,\\
g_2&:z_1\mapsto  \tfrac{z_3}{z_4}, z_2\mapsto - \tfrac{\z z_6}{z_4}, 
z_3\mapsto  \tfrac{z_1}{z_4}, z_4\mapsto  \tfrac{1}{z_4}, 
 z_5\mapsto - \tfrac{\z z_7}{z_4}, z_6\mapsto  \tfrac{\z z_2}{z_4}, z_7\mapsto  \tfrac{\z z_5}{z_4}, 
z_8\mapsto z_4 z_8,\\
g_3&:z_1\mapsto -z_1, z_2\mapsto \z z_2, z_3\mapsto -z_3, z_4\mapsto z_4, z_5\mapsto \z z_5, z_6\mapsto -\z z_6, 
 z_7\mapsto -\z z_7, z_8\mapsto z_8,\\
g_4&:z_1\mapsto - \tfrac{z_4}{\z z_3}, z_2\mapsto - \tfrac{z_5}{\z z_3}, 
z_3\mapsto - \tfrac{1}{z_3}, z_4\mapsto - \tfrac{z_1}{\z z_3}, 
z_5\mapsto - \tfrac{z_2}{\z z_3}, z_6\mapsto - \tfrac{z_7}{z_3}, z_7\mapsto  \tfrac{z_6}{z_3}, 
 z_8\mapsto -\z z_3 z_8,\\
g_5&:z_1\mapsto z_1, z_2\mapsto -z_2, z_3\mapsto -z_3, z_4\mapsto -z_4, z_5\mapsto z_5, z_6\mapsto z_6, 
 z_7\mapsto -z_7, z_8\mapsto -\z z_8,\\
g_6&:z_1\mapsto z_1, z_2\mapsto -z_2, z_3\mapsto z_3, z_4\mapsto z_4, z_5\mapsto -z_5, z_6\mapsto -z_6, 
 z_7\mapsto -z_7, z_8\mapsto z_8,\\
g_7&:z_1\mapsto z_1, z_2\mapsto z_2, z_3\mapsto z_3, z_4\mapsto z_4, z_5\mapsto z_5, z_6\mapsto z_6, z_7\mapsto z_7,
  z_8\mapsto -z_8.
\end{align*}

Apply Theorem \ref{thAHK} to $k(z_1,z_2,z_3,z_4,z_5,z_6,z_7)(z_8)$, 
$k(z_1,z_2,z_3,z_4,z_5,z_6,z_7,z_8)^G$ is rational over 
$k(z_1,z_2,z_3,z_4,z_5,z_6,z_7)^G$. 
Because $Z(G)=\langle g_7\rangle$ acts on $k(z_1,z_2,z_3,z_4,z_5,z_6,z_7)$ 
trivially, we have 
$k(z_1,z_2,z_3,z_4,z_5,z_6,z_7)^G
=k(z_1,z_2,z_3,z_4,z_5,z_6,z_7)^{\langle g_1,g_2,g_3,g_4,g_5,g_6\rangle}$. 
It suffices to show that 
$k(z_1,z_2,z_3,z_4,z_5,z_6,z_7)^{\langle g_1,g_2,g_3,g_4,g_5,g_6\rangle}$ 
is rational over  $\varphi(L_k^{(2)})$.

Define 
$u_1 = -\frac{z_2 z_7}{z_5 z_6}$, 
$u_2 = -\frac{z_6 z_7}{\z z_3}$, 
$u_3 = \frac{z_3 z_4}{z_1}$, 
$u_4 = -\frac{z_4 z_6}{z_7}$, 
$u_5 = \frac{z_5}{z_3 z_7}$, 
$u_6 = -\frac{\z z_2 z_3}{z_6}$, 
$u_7 = -\frac{z_1 z_7}{\z z_3 z_6}$.
By evaluating the determinant of the matrix $M$ 
of exponents as in Case 1 
(see the equation (\ref{eqdet})), we have $\det M=8$, 
$k(z_1,z_2,z_3,z_4,z_5,z_6,z_7)^{\langle g_3,g_5,g_6\rangle}=k(u_1,u_2,u_3,u_4,u_5,u_6,u_7)$ 
and 
\begin{align*}
g_1&:u_1\mapsto -\tfrac{1}{u_3}, u_2\mapsto -\tfrac{1}{u_2}, u_3\mapsto u_1, u_4\mapsto u_6, u_5\mapsto u_4, u_6\mapsto u_7,
  u_7\mapsto u_5,\\
g_2&:u_1\mapsto \tfrac{1}{u_1}, u_2\mapsto -\tfrac{u_2 u_5 u_6}{u_4 u_7}, u_3\mapsto \tfrac{1}{u_3}, 
 u_4\mapsto -\tfrac{u_1}{\z u_3 u_7}, u_5\mapsto -\tfrac{\z u_1 u_3}{u_6}, u_6\mapsto -\tfrac{\z}{u_1 u_3 u_5}, u_7\mapsto -\tfrac{u_3}{\z u_1 u_4},\\
g_4&:u_1\mapsto -\tfrac{1}{u_1}, u_2\mapsto u_2, u_3\mapsto -\tfrac{1}{u_3}, u_4\mapsto u_7, u_5\mapsto u_6, u_6\mapsto u_5,
  u_7\mapsto u_4.
\end{align*}

Note that $g_1^2=g_4$. 
Thus we will omit the presentation of the action $g_4$. 
Define 
$v_1=\frac{u_1+\z}{u_1-\z}$, 
$v_2=\frac{u_2 u_5 u_6 + u_4 u_7 \z}{u_2 - \z}$, 
$v_3=-\frac{u_3+z}{u_3 - \z}$, 
$v_4=u_4$, 
$v_5=u_6$, 
$v_6=\frac{\z u_1}{u_3 u_7}$, 
$v_7=-\frac{\z}{u_1 u_3 u_5}$.
Then $k(u_1,u_2,u_3,u_4,u_5,u_6,u_7)$ $=$ $k(v_1,v_2,v_3,v_4,v_5,v_6,v_7)$ 
and 
\begin{align*}
g_1:\ &v_1\mapsto v_3, v_2\mapsto -v_2, v_3\mapsto -v_1,\\ 
&v_4\mapsto v_5, 
v_5\mapsto \tfrac{(v_1+1)(v_3+1)\z}{(v_1-1)(v_3-1)v_6}, 
v_6\mapsto v_7, 
v_7\mapsto \tfrac{(v_1-1)(v_3-1)\z}{(v_1+1)(v_3+1)v_4},\\
g_2:\ &v_1\mapsto -\tfrac{1}{v_1}, v_2\mapsto -\tfrac{1}{v_2}, 
v_3\mapsto -\tfrac{1}{v_3}, v_4\mapsto v_6, v_5\mapsto v_7, 
 v_6\mapsto v_4, v_7\mapsto v_5.
\end{align*}

Define 
$w_1=\bigl(\frac{v_1+\z}{v_1-\z}\bigr)^2$,
$w_2=\bigl(\frac{v_1+\z}{v_1-\z}\bigr)\bigl(\frac{v_2+\z}{v_2-\z}\bigr)$,
$w_3=\bigl(\frac{v_1+\z}{v_1-\z}\bigr)\bigl(\frac{v_3+\z}{v_3-\z}\bigr)$, 
$w_4=v_4 + v_6$, 
$w_5=v_5 + v_7, w_6=\bigl(\frac{v_1+\z}{v_1-\z}\bigr)(v_4 - v_6)$,
$w_7=\bigl(\frac{v_1+\z}{v_1-\z}\bigr)(v_5 - v_7)$.
Then 
$k(v_1,v_2,v_3,v_4,v_5,v_6,v_7)^{\langle g_2\rangle}$
$=$ $k(w_1,w_2,w_3,w_4,w_5,w_6,w_7)$ and 
\begin{align*}
g_1:\ &w_1\mapsto \tfrac{w_3^2}{w_1},
w_2\mapsto \tfrac{w_3}{w_2},
w_3\mapsto \tfrac{w_3}{w_1},
w_4\mapsto w_5, 
w_5\mapsto \tfrac{4w_1\left(2 (w_3-1) (w_1 + w_3) w_6 - \z (w_1-w_1^2-4 w_1 w_3-w_3^2 +w_1 w_3^2) w_4\right)}{(w_1+1)(w_1+w_3^2)(w_1w_4^2-w_6^2)},\\
& w_6\mapsto \tfrac{w_3w_7}{w_1}, 
w_7\mapsto \tfrac{4w_3\left(2 w_1 (w_3-1) (w_1 + w_3) w_4-\z(w_1-w_1^2-4 w_1 w_3-w_3^2+w_1 w_3^2) w_6\right)}{(w_1+1)(w_1+w_3^2)(w_1w_4^2-w_6^2)}.
\end{align*}

Define 
\begin{align*}
p_1&=\tfrac{w_1 + w_3}{w_1 - w_3},
p_2=\tfrac{2 w_2}{w_3-1},
p_3=\tfrac{w_3+1}{w_3-1}, 
p_4=\tfrac{2 w_1  (w_1 + w_3)(w_3-1) w_4}{(w_1+1) (w_1 + w_3^2) w_6} 
-\tfrac{(w_1-w_1^2-w_3^2+w_1 w_3^2-4 w_1 w_3) \z}{(w_1+1) (w_1 + w_3^2)},\\
p_5&=-\tfrac{2 w_1 (w_1 - w_3) (w_3+1) w_5}{(w_1+1) (w_1 + w_3^2) w_7} 
+\tfrac{(w_1-w_1^2-w_3^2+w_1 w_3^2+4 w_1 w_3) \z}{(w_1+1) (w_1 + w_3^2)}, 
p_6=w_6, p_7=\tfrac{w_3 w_7}{w_1}.
\end{align*}
Then it follows from 
\begin{align*}
w_1 &= \bigl(\tfrac{p_1+1}{p_1-1}\bigr)\bigl(\tfrac{p_3+1}{p_3-1}\bigr), 
w_2 = \tfrac{p_2}{p_3-1}, 
w_3 = \tfrac{p_3+1}{p_3-1}, 
w_4 = \tfrac{(p_1^2 p_3^2-1) p_4 +\z ( 2 p_1^2 - p_1^2 p_3^2-1)}
{2 p_1 (p_1+1) (p_3+1)}p_6,\\ 
w_5 &= -\tfrac{(p_1^2 p_3^2-1) p_5 +\z (2 p_3^2 - p_1^2 p_3^2-1)}
{2 p_3 (p_1-1) (p_3+1)}p_7,
w_6 = p_6, 
w_7 = \bigl(\tfrac{p_1+1}{p_1-1}\bigr)p_7
\end{align*}
that $k(w_1,w_2,w_3,w_4,w_5,w_6,w_7)=k(p_1,p_2,p_3,p_4,p_5,p_6,p_7)$ 
and 
\begin{align*}
g_1:\ &p_1 \mapsto p_3, 
p_2 \mapsto -\tfrac{(p_1+1)(p_3+1)}{p_2}, 
p_3 \mapsto -p_1, 
p_4 \mapsto p_5, 
p_5 \mapsto -\tfrac{1}{p_4}, 
p_6 \mapsto p_7,\\ 
\ &p_7 \mapsto 
\tfrac{16p_1^2(p_1^2-1)(p_3^2-1)p_4
\left((p_1^2p_3^2-1)(p_4^2-1)-2(2p_1^2-p_1^2p_3^2-1)p_4\z\right)}
{(p_1^2p_3^2-1)\left(2p_4^2(-8p_1^2+8p_1^4+6p_1^2p_3^2-8p_1^4p_3^2+
p_1^4p_3^4+1)+(p_1^2p_3^2-1)^2(p_4^4+1)\right)p_6}.
\end{align*}

We also define 
\begin{align*}
X_1&=p_1,
X_2=p_3, 
X_3=\tfrac{p_4^2+2p_4\z-1}{p_1(p_4^2+1)},
X_4=\tfrac{p_5^2+2p_5\z-1}{p_3(p_5^2+1)},\\
X_5&=\tfrac{(p_1^2p_3^2-1)p_4^3p_6
\left(p_4(4p_1^2-p_1^2p_3^2+p_4^2-p_1^2p_3^2p_4^2-3)
+(p_1^2p_3^2+3p_4^2-4p_1^2p_4^2+p_1^2p_3^2p_4^2-1)\z\right)}
{4\eta p_1^2(p_1+1)(p_3+1)(p_4^2+1)},\\
X_6&=-\tfrac{(p_1^2p_3^2-1)p_5^3p_7
\left(p_5(4p_3^2-p_1^2p_3^2+p_5^2-p_1^2p_3^2p_5^2-3)
+(p_1^2p_3^2+3p_5^2-4p_3^2p_5^2+p_1^2p_3^2p_5^2-1)\z\right)}
{4\eta p_3^2(p_1-1)(p_3+1)(p_5^2+1)},\\
X_7&=\tfrac{p_1p_3\left((p_1+p_2+1)(p_1-p_2+1)-2\z(p_1+1)p_2\right)}{p_1^2+2p_1+p_2^2+1}.
\end{align*}
Then it follows from 
\begin{align*}
p_1&=X_1,
p_2=\tfrac{(X_1+1)(X_7-X_1X_2)\z}{X_7+X_1X_2},
p_3=X_2, 
p_4=\tfrac{X_1X_3+1}{X_1X_3-1}\z,
p_5=\tfrac{X_2X_4+1}{X_2X_4-1}\z,\\
p_6&=-\tfrac{2\eta (X_1+1)(X_2+1)(X_1X_3-1)^4X_5}
{(X_1^2X_2^2-1)(X_1X_3+1)^3(X_2^2-X_3^2+X_1^2X_3^2-1)}, 
p_7=\tfrac{2\eta (X_1-1)(X_2+1)(X_2X_4-1)^4X_6}
{(X_1^2X_2^2-1)(X_2X_4+1)^3(X_1^2-X_4^2+X_2^2X_4^2-1)}
\end{align*}
that $k(p_1,p_2,p_3,p_4,p_5,p_6,p_7)=k(X_1,X_2,X_3,X_4,X_5,X_6,X_7)$. 
We see that 
the action of $\langle g_1\rangle$ on $k(X_1,X_2,X_3,X_4,X_5,X_6)$ 
and that of $\langle\rho\rangle$ in Definition \ref{defL23} (i) 
are exactly the same. 
We also have 
$k(X_1,X_2,X_3,X_4,X_5,X_6,X_7)^{\langle g_1\rangle}=
k(X_1,X_2,X_3,X_4,X_5,X_6)^{\langle g_1\rangle}(X_7)$ 
because $X_7$ is an invariant under the action of $\langle g_1\rangle$.
Hence $k(V_{138})^G$ is rational over $\varphi(L_k^{(2)})$.\\

{Case $11$: $G=G(2^7,1544)$ which belongs to $\Phi_{30}$.}

$G=\langle g_1,g_2,g_3,g_4,g_5,g_6,g_7\rangle$ with relations 
$g_1^2=1$, $g_2^2=1$, 
$g_3^2=1$, $g_4^2=1$, 
$g_5^2=1$, $g_6^2=1$, $g_7^2=1$, 
$Z(G)=\langle g_5,g_6,g_7\rangle$,
$[g_2,g_1]=g_5$, $[g_3,g_1]=g_6$, 
$[g_3,g_2]=g_7$, $[g_4,g_2]=g_5g_6$, 
$[g_4,g_3]=g_5$.


There exists a faithful representation 
$\rho : G\rightarrow GL(V_{1544})\simeq GL_{10}(k)$ 
of dimension $10$ 
which is decomposable into three irreducible components 
$V_{1544}\simeq U_4\oplus U_2\oplus U_4^\prime$ 
of dimension $4$, $2$ and $4$ respectively. 
By Theorem \ref{thHK}, $k(G)$ is rational over $k(V_{1544})^G$. 
We will show that $k(V_{1544})^G$ is rational over $\varphi(L_k^{(3)})$. 

The action of $G$ on $k(V_{1544})=k(y_1,y_2,y_3,y_4,y_5,y_6,y_7,y_8,y_9,y_{10})$ 
is given by 
{\small 
\begin{align*}
g_1&:y_1\mapsto y_3, y_2\mapsto y_4, y_3\mapsto y_1, y_4\mapsto y_2, y_5\mapsto y_5, y_6\mapsto y_6, y_7\mapsto y_9,
  y_8\mapsto y_{10}, y_9\mapsto y_7, y_{10}\mapsto y_8,\\
g_2&:y_1\mapsto y_2, y_2\mapsto y_1, y_3\mapsto y_4, y_4\mapsto y_3, y_5\mapsto y_6, y_6\mapsto y_5, 
 y_7\mapsto -y_8, y_8\mapsto -y_7, y_9\mapsto y_{10}, y_{10}\mapsto y_9,\\
g_3&:y_1\mapsto -y_1, y_2\mapsto -y_2, y_3\mapsto y_3, y_4\mapsto y_4, y_5\mapsto -y_5, y_6\mapsto y_6, 
 y_7\mapsto y_8, y_8\mapsto y_7, y_9\mapsto y_{10}, y_{10}\mapsto y_9,\\
g_4&:y_1\mapsto -y_1, y_2\mapsto y_2, y_3\mapsto -y_3, y_4\mapsto y_4, y_5\mapsto y_5, y_6\mapsto y_6, 
 y_7\mapsto -y_7, y_8\mapsto y_8, y_9\mapsto -y_9, y_{10}\mapsto y_{10},\\
g_5&:y_1\mapsto y_1, y_2\mapsto y_2, y_3\mapsto y_3, y_4\mapsto y_4, y_5\mapsto y_5, y_6\mapsto y_6, 
 y_7\mapsto -y_7, y_8\mapsto -y_8, y_9\mapsto -y_9, y_{10}\mapsto -y_{10},\\
g_6&:y_1\mapsto -y_1, y_2\mapsto -y_2, y_3\mapsto -y_3, y_4\mapsto -y_4, y_5\mapsto y_5, y_6\mapsto y_6, 
 y_7\mapsto y_7, y_8\mapsto y_8, y_9\mapsto y_9, y_{10}\mapsto y_{10},\\
g_7&:y_1\mapsto y_1, y_2\mapsto y_2, y_3\mapsto y_3, y_4\mapsto y_4, y_5\mapsto -y_5, y_6\mapsto -y_6, 
 y_7\mapsto y_7, y_8\mapsto y_8, y_9\mapsto y_9, y_{10}\mapsto y_{10}.
\end{align*}
}

Define $z_1=\frac{y_1}{y_4}$, $z_2=\frac{y_2}{y_4}$, 
$z_3=\frac{y_3}{y_4}$, $z_4=\frac{y_5}{y_6}$, $z_5=\frac{y_7}{y_{10}}$, 
$z_6=\frac{y_8}{y_{10}}$, $z_7=\frac{y_9}{y_{10}}$, $z_8=y_4$, $z_9=y_6$, $z_{10}=y_{10}$. 
Then $k(y_1,y_2,y_3,y_4,y_5,y_6,y_7,y_8,y_9,y_{10})
=k(z_1,z_2,z_3,z_4,z_5,z_6,z_7,z_8,z_9,z_{10})$ and 
{\small 
\begin{align*}
g_1&:z_1\mapsto \tfrac{z_3}{z_2}, z_2\mapsto \tfrac{1}{z_2}, z_3\mapsto \tfrac{z_1}{z_2}, z_4\mapsto z_4, z_5\mapsto \tfrac{z_7}{z_6}, 
 z_6\mapsto \tfrac{1}{z_6}, z_7\mapsto \tfrac{z_5}{z_6}, z_8\mapsto z_2 z_8, z_9\mapsto z_9, z_{10}\mapsto z_{10} z_6,\\
g_2&:z_1\mapsto \tfrac{z_2}{z_3}, z_2\mapsto \tfrac{z_1}{z_3}, z_3\mapsto \tfrac{1}{z_3}, z_4\mapsto \tfrac{1}{z_4}, z_5\mapsto -\tfrac{z_6}{z_7}, 
 z_6\mapsto -\tfrac{z_5}{z_7}, z_7\mapsto \tfrac{1}{z_7}, z_8\mapsto z_3 z_8, z_9\mapsto z_4 z_9, z_{10}\mapsto z_{10} z_7,\\
g_3&:z_1\mapsto -z_1, z_2\mapsto -z_2, z_3\mapsto z_3, z_4\mapsto -z_4, z_5\mapsto \tfrac{z_6}{z_7}, z_6\mapsto \tfrac{z_5}{z_7}, z_7\mapsto \tfrac{1}{z_7}, z_8\mapsto z_8, z_9\mapsto z_9, z_{10}\mapsto z_{10} z_7,\\
g_4&:z_1\mapsto -z_1, z_2\mapsto z_2, z_3\mapsto -z_3, z_4\mapsto z_4, z_5\mapsto -z_5, z_6\mapsto z_6, 
 z_7\mapsto -z_7, z_8\mapsto z_8, z_9\mapsto z_9, z_{10}\mapsto z_{10},\\
g_5&:z_1\mapsto z_1, z_2\mapsto z_2, z_3\mapsto z_3, z_4\mapsto z_4, z_5\mapsto z_5, z_6\mapsto z_6, z_7\mapsto z_7,
  z_8\mapsto z_8, z_9\mapsto z_9, z_{10}\mapsto -z_{10},\\
g_6&:z_1\mapsto z_1, z_2\mapsto z_2, z_3\mapsto z_3, z_4\mapsto z_4, z_5\mapsto z_5, z_6\mapsto z_6, z_7\mapsto z_7,
  z_8\mapsto -z_8, z_9\mapsto z_9, z_{10}\mapsto z_{10},\\
g_7&:z_1\mapsto z_1, z_2\mapsto z_2, z_3\mapsto z_3, z_4\mapsto z_4, z_5\mapsto z_5, z_6\mapsto z_6, z_7\mapsto z_7,
  z_8\mapsto z_8, z_9\mapsto -z_9, z_{10}\mapsto z_{10}.
\end{align*}
}

By applying 
Theorem \ref{thAHK} three times to 
$k(z_1,z_2,z_3,z_4,z_5,z_6,z_7)(z_8,z_9,z_{10})$, 
we obtain that 
$k(z_1,z_2,z_3,z_4,z_5,z_6,z_7,z_8,z_9,z_{10})^G$ is rational over 
$k(z_1,z_2,z_3,z_4,z_5,z_6,z_7)^G$. 
We also find that 
$k(z_1,z_2,z_3,z_4,z_5,z_6,z_7)^G
=k(z_1,z_2,z_3,z_4,z_5,z_6,z_7)^{\langle g_1,g_2,g_3,g_4\rangle}$ 
because $Z(G)=\langle g_5,g_6,g_7\rangle$ acts on $k(z_1,z_2,z_3,z_4,z_5,z_6,z_7)$ 
trivially. 
Hence it suffices to show that 
$k(z_1,z_2,z_3,z_4,z_5,z_6,z_7)^{\langle g_1,g_2,g_3,g_4\rangle}$ 
is rational over  $\varphi(L_k^{(3)})$.

Define $u_1=\frac{z_1}{z_2 z_7}$, $u_2=\frac{z_1}{z_3}$, 
$u_3=\frac{z_1 z_3}{z_2}$, $u_4=z_4$, $u_5=\frac{z_1 z_5}{z_2}$, 
$u_6=\frac{z_1 z_6}{z_2 z_7}$, $u_7=\frac{z_1 z_7}{z_2}$.
By evaluating the determinant of the matrix $M$ 
of exponents as in Case 1 
(see the equation (\ref{eqdet})), we have $\det M=2$, 
$k(z_1,z_2,z_3,z_4,z_5,z_6,z_7)^{\langle g_4\rangle}=k(u_1,u_2,u_3,u_4,u_5,u_6,u_7)$ 
and 
\begin{align*}
g_1&:u_1\mapsto \tfrac{u_3 u_6}{u_1 u_5}, u_2\mapsto \tfrac{1}{u_2}, 
u_3\mapsto u_3, u_4\mapsto u_4, u_5\mapsto \tfrac{u_3}{u_6}, 
 u_6\mapsto \tfrac{u_3}{u_5}, u_7\mapsto \tfrac{u_3 u_5}{u_6 u_7},\\
g_2&:u_1\mapsto \tfrac{1}{u_1}, u_2\mapsto \tfrac{u_2 u_3}{u_1 u_7}, u_3\mapsto \tfrac{1}{u_3}, 
u_4\mapsto \tfrac{1}{u_4}, u_5\mapsto -\tfrac{u_6}{u_1 u_7}, 
u_6\mapsto -\tfrac{u_5}{u_1 u_7}, u_7\mapsto \tfrac{1}{u_7},\\
g_3&:u_1\mapsto u_7, u_2\mapsto -u_2, u_3\mapsto u_3, u_4\mapsto -u_4, u_5\mapsto u_6, u_6\mapsto u_5,  u_7\mapsto u_1.
\end{align*}

Define $v_1=u_1+u_7$, $v_2=u_2u_4$,$v_3=u_3$, $v_4=u_4^2$, 
$v_5=u_5+u_6$, $v_6=u_4(u_5-u_6)$, $v_7=u_4(u_1- u_7)$. 
Then $k(u_1,u_2,u_3,u_4,u_5,u_6,u_7)^{\langle g_3\rangle}=k(v_1,v_2,v_3,v_4,v_5,v_6,v_7)$ 
and 
\begin{align*}
g_1:&\ v_1\mapsto \tfrac{4 v_3 v_4 (v_1 v_4 v_5^2 + v_1 v_6^2 + 2 v_5 v_6 v_7)}{(v_4 v_5^2 - 
     v_6^2) (v_1^2 v_4 - v_7^2)}, v_2\mapsto \tfrac{v_4}{v_2}, v_3\mapsto v_3, v_4\mapsto v_4,\\ 
 &\ v_5\mapsto \tfrac{4 v_3 v_4 v_5}{v_4 v_5^2 - v_6^2}, 
 v_6\mapsto \tfrac{4 v_3 v_4 v_6}{v_4 v_5^2 - v_6^2}, 
 v_7\mapsto -\tfrac{4 v_3 v_4 (2 v_1 v_4 v_5 v_6 + v_4 v_5^2 v_7 + v_6^2 v_7)}{(v_4 v_5^2 - 
      v_6^2) (v_1^2 v_4 - v_7^2)},\\
g_2:&\ v_1\mapsto \tfrac{4 v_1 v_4}{v_1^2 v_4 - v_7^2}, 
v_2\mapsto \tfrac{4 v_2 v_3}{v_1^2 v_4 - v_7^2}, 
 v_3\mapsto \tfrac{1}{v_3}, v_4\mapsto \tfrac{1}{v_4}, 
v_5\mapsto -\tfrac{4 v_4 v_5}{v_1^2 v_4 - v_7^2}, 
 v_6\mapsto \tfrac{4 v_6}{v_1^2 v_4 - v_7^2}, v_7\mapsto -\tfrac{4 v_7}{v_1^2 v_4 - v_7^2}.
\end{align*}

Define 
\begin{align*}
X_1&=\tfrac{1}{v_4},
X_2=\tfrac{1}{v_2},
X_3=\tfrac{v_1^2v_4-v_7^2}{4v_2v_3},
X_4=\tfrac{(v_4v_5^2-v_6^2)(v_1^2v_4-v_7^2)}{2v_4v_5(v_1v_4v_5+v_6v_7)},\\
X_5&=-\tfrac{(v_1v_4v_5^2+v_1v_6^2+2v_5v_6v_7)}{v_1(v_1v_4v_5+v_6v_7)},
X_6=\tfrac{v_4(v_1v_6+v_5v_7)}{v_1v_4v_5+v_6v_7},
X_7=\tfrac{v_6}{v_5}.
\end{align*}
Then it follows from 
\begin{align*}
v_1&=-\tfrac{2(X_4-X_1X_4X_6X_7)}{X_1X_6^2+X_1X_7^2-X_1^2X_6^2X_7^2-1},
v_2=\tfrac{1}{X_2}, 
v_3=\tfrac{X_2X_4^2}{X_1X_3(X_1X_6^2-1)(X_1X_7^2-1)},
v_4=\tfrac{1}{X_1},\\
v_5&=-\tfrac{2(X_1X_4X_5X_6X_7-X_4X_5)}
{X_1X_6^2-X_1X_6X_7+X_1^2X_6^3X_7+X_1X_7^2-X_1^2X_6^2X_7^2+X_1^2X_6X_7^3-X_1^3X_6^3X_7^3-1},\\
v_6&=-\tfrac{2X_4X_5X_7(X_1X_6X_7-1)}{(X_1X_6^2-1)(X_1X_6X_7-X_1X_7^2-X_1^2X_6X_7^3+1)},
v_7=\tfrac{2(X_4X_6-X_4X_7)}{(X_1X_6^2-1)(X_1X_7^2-1)}
\end{align*}
that 
$k(v_1,v_2,v_3,v_4,v_5,v_6,v_7)=k(X_1,X_2,X_3,X_4,X_5,X_6,X_7)$. 
We see that the action of $\langle g_1,g_2\rangle$ on $k(X_1,X_2,X_3,X_4,X_5,X_6,X_7)$ 
and that of $\langle\lambda_1,\lambda_2\rangle$ in Definition \ref{defL23} (ii) 
are exactly the same. 
Hence $k(V_{1544})^{G}$ is rational over $\varphi(L_k^{(3)})$. \qed

\begin{proof}[Proof of Theorem \ref{thm}]
~\\
Take a base field $k$ as $k=\bC$. 
The assertion follows from Theorem \ref{thBB} and Theorem \ref{thmain2}. 
\end{proof}

\section{Classification of groups of order $128$ into $115$ isoclinism families}\label{seTable}
\begin{center}
{\scriptsize 
\begin{tabular}{llr} 
$\Phi_j$&$G(2^7,i),i\in$&\#\\\hline
1&\{1,42,128,159,179,456,837,988,997,1601,2136,2150,2301,2319,2328\}&15\\\hline
2&\{5,43,44,106,129,131,153,160,164,180,181,182,457,458,459,476,480,483,498,501,\\
&838,839,840,843,881,894,899,914,989,990,998,999,1000,1001,1002,1003,1004,1602,1603,1604,\\
&1606,1608,1634,1649,1658,1690,1696,2137,2138,2151,2152,2153,2154,2155,2156,2302,2303,2320,2321,2322\}&60\\\hline
3&\{8,9,10,11,61,69,99,102,108,111,133,154,206,207,208,294,295,296,307,308,\\
&309,464,466,467,469,490,492,493,496,504,506,507,509,848,884,902,1622,1623,1624,1631,\\
&1639,1640,1641,1646,1668,1669,1670,1671,1685,2306,2307,2308,2309\}&53\\\hline
4&\{165,166,167,168,169,183,184,185,186,187,460,461,477,478,481,482,484,485,499,500,\\
&502,503,511,529,549,563,571,574,582,601,648,655,671,686,716,844,845,882,883,895,\\
&896,900,901,915,1009,1010,1011,1012,1013,1014,1015,1016,1017,1018,1019,1020,1021,1022,1023,1024,\\
&1025,1026,1027,1028,1029,1030,1031,1032,1033,1034,1035,1036,1037,1038,1039,1609,1610,1635,1636,1650,\\
&1651,1652,1654,1659,1660,1661,1663,1691,1692,1697,1698,2163,2164,2165,2166,2167,2168,2169,2170,2171,\\
&2172,2173,2174,2175,2176\}&105\\\hline
5&\{1005,1006,1007,1008,1605,1607,1704,1714,1720,2139,2157,2158,2159,2160,2161,2162,2304,2305,2323,2324,\\
&2325\}&21\\\hline
6&\{209,210,211,297,310,311,312,465,468,470,491,494,495,497,505,508,510,849,885,903,1625,1626,1627,\\
&1628,1632,1642,1643,1647,1672,1673,1674,1675,1676,1677,1686,1687,2310,2311,2312\}&
39\\\hline
7&\{12,13,14,15,46,54,109,110,132,188,189,190,191,192,193,471,472,473,474,475,\\
&486,487,488,489,846,847,1613,1614,1617,1618,1619\}&31\\\hline
8&\{63,64,67,103,104,105,112,113,114,868,869,870,874,888,889,890,892,904,905,906,\\
&910,2140,2141,2142,2143\}&25\\\hline
9&\{170,171,172,173,174,175,176,177,178,512,530,550,564,572,575,583,602,649,656,672,\\
&687,717,1116,1117,1118,1119,1120,1121,1122,1123,1124,1125,1126,1127,1128,1129,1130,1131,1132,1133,\\
&1134\}&41\\\hline
10&\{1070,1071,1072,1073,1074,1075,1076,1077,1078,1079,1080,1081,1082,1083,1084,1085,1086,1087,1088,1089,\\
&1090,1091,1092,1093,1094,1095,1096,1097,1098,1099,1100,1101,1102,1103,1104,1105,1106,1612,1638,1656,\\
&1657,1665,1666,1667,1694,1695,1700,1701,1702,1703,1706,1709,1711,1716,1718,1722,1723,1726,2194,2195,\\
&2196,2197,2198,2199,2200,2201,2202,2203,2204,2205,2206,2207,2208,2209,2210,2211,2212,2213,2214,2215\}&80\\\hline
11&\{1040,1041,1042,1043,1044,1045,1046,1047,1048,1049,1050,1051,1052,1053,1054,1055,1056,1057,1058,1059,\\
&1060,1061,1062,1063,1064,1065,1066,1067,1068,1069,1611,1637,1653,1655,1662,1664,1693,1699,1705,1707,\\
&1708,1713,1715,1717,1721,1724,1725,2177,2178,2179,2180,2181,2182,2183,2184,2185,2186,2187,2188,2189,\\
&2190,2191,2192,2193\}&64\\\hline
12&\{6,7,45,107,130,462,463,479,841,842\}&10\\\hline
13&\{1107,1108,1109,1110,1111,1112,1113,1114,1115,1710,1712,1719,1727,2257,2258,2259,2260,2261,2262,2263\}&20\\\hline
14&\{224,225,226,298,299,300,321,322,323,539,540,542,566,567,569,576,577,578,580,624,\\
&625,628,665,666,669,673,674,677,688,689,690,694,700,701,702,706,722,723,726,1779,\\
&1780,1781,1782,1860,1861,1862,1863,1875,1876,1877,1878,1879,1889,1890,1891,1892,1899,1900,1901,1902\}&60\\\hline
15&\{212,213,214,215,216,217,218,219,220,221,222,223,302,303,304,305,306,313,314,315,\\
&316,317,318,319,320,518,519,520,521,534,535,536,537,538,555,556,557,558,584,585,\\
&586,587,594,595,596,597,598,603,604,605,606,607,608,609,610,611,612,650,651,652,\\
&653,654,657,658,659,660,661,679,680,681,682,712,713,714,715,718,719,720,721,1728,\\
&1729,1730,1731,1732,1733,1734,1761,1762,1763,1764,1765,1766,1767,1802,1803,1804,1805,1806,1807,1808,\\
&1817,1818,1819,1820,1821,1822,1823,1832,1833,1834,1835,1836,1837,1838,1839,1840\}&116\\\hline
16&\{227,228,229,301,324,325,326,541,543,568,570,579,581,626,627,629,667,668,670,675,\\
&676,678,691,692,693,695,703,704,705,707,724,725,727,1783,1784,1785,1786,1864,1865,1866,\\
&1867,1880,1881,1882,1893,1894,1903,1904\}&48\\\hline
17&\{2,3,4,26,27,28,29,30,31,32,33,34,35,47,55,62,70,230,231,232,233,234,235,254,255,256,257,258,270,\\
&271,272,273,274,275,276\}&35\\\hline
18&\{1629,1630,1633,1644,1645,1648,1678,1679,1680,1681,1682,1683,1684,1688,1689,2313,2314,2315,2316\}&19\\\hline
19&\{871,872,873,875,891,893,907,908,909,911,2144,2145,2146
\}&13\\\hline
20&\{100,101,115,116,117,886,887\}&7\\\hline
21&\{65,66,68,118,119,120,121,876,877,878,879,880\}&12\\\hline
22&\{48,49,52,58,60,122,123,124,850,851,852\}&11\\\hline
23&\{50,51,53,56,57,59,125,126,127,856,857,858,862,863,864
\}&15\\\hline
24&\{546,547,548,638,639,640,662,663,664,683,684,685,708,709,710,711,1796,1797,1798,1799\}&20\\\hline
25&\{522,523,599,600,621,622,623,1755,1756,1757\}&10\\\hline
\end{tabular}
}

\newpage
{\scriptsize
\begin{tabular}{llr}
$\Phi_j$&$G(2^7,i),i\in$&\#\\\hline
26&\{524,525,526,588,589,590,591,592,593,613,614,615,616,617,618,619,620,1746,1747,1748,\\
&1749,1750\}&22\\\hline
27&\{147,148,149,155,156,157,991,992,993,994\}&10\\\hline
28&\{1400,1401,1402,1403,1404,1405,1406,1407,1408,1409,1410,1411,1412,1413,1414,1415,1416,1417,1418,1419,\\
&1420,1421,1422,1423,1424,1425,1426,1427,1428,1429,1430,1431,1432,1433,1434,1435,1436,1437,1438,1439,\\
&1440,1441,1442,1443,1444,1445,1446,1447,1448,1449,1450,1451,1452,1453,1454,1455,1456,1457,1458,1459,\\
&1460,1461,1462,1463,1464,1465,1466,1467,1468,1469,1470,1471,1472,1473,1474,1475,1476,1477,1478,1479,\\
&1480,1481,1482,1483,1484,1485,1486,1487,1488,1489,1490,1491,1492,1493,1494,1495,1496,1497,1498,1499,\\
&1500,1501,1502,1503,1504,1505,1506,1507,1508,1509,1510,1511,1512,1513,1514,1515,1516,1517,1518,1519,\\
&1520,1521,1522,1523,1524,1525,1526,1527,1528,1529,1530,1531,1532,1533,1534,1535,1536,1537,1538,1539,\\
&1540,1541,1542,1543\}&144\\\hline
29&\{1135,1136,1137,1138,1139,1140,1141,1142,1143,1144,1145,1146,1147,1148,1149,1150,1151,1152,1153,1154,\\
&1155,1156,1157,1158,1159,1160,1161,1162,1163,1164,1165,1166,1167,1168,1169,1170,1171,1172,1173,1174,\\
&1175,1176,1177,1178,1179,1180,1181,1182,1183,1184,1185,1186,1187,1188,1189,1190,1191,1192,1193,1194,\\
&1195,1196,1197,1198,1199,1200,1201,1202,1203,1204,1205,1206,1207,1208,1209,1210,1211,1212,1213,1214,\\
&1215,1216,1217,1218,1219,1220,1221,1222,1223,1224,1225,1226,1227,1228,1229,1230,1231,1232,1233,1234,\\
&1235,1236,1237,1238,1239,1240,1241,1242,1243,1244,1245,1246,1247,1248,1249,1250,1251,1252,1253,1254,\\
&1255,1256,1257,1258,1259,1260,1261,1262,1263,1264,1265,1266,1267,1268,1269,1270,1271,1272,1273,1274,\\
&1275,1276,1277,1278,1279,1280,1281,1282,1283,1284,1285,1286,1287,1288,1289,1290,1291,1292,1293,1294,\\
&1295,1296,1297,1298,1299,1300,1301,1302,1303,1304,1305,1306,1307,1308,1309,1310,1311,1312,1313,1314,\\
&1315,1316,1317,1318,1319,1320,1321,1322,1323,1324,1325,1326,1327,1328,1329,1330,1331,1332,1333,1334,\\
&1335,1336,1337,1338,1339,1340,1341,1342,1343,1344\}&210\\\hline
30&\{1544,1545,1546,1547,1548,1549,1550,1551,1552,1553,1554,1555,1556,1557,1558,1559,1560,1561,1562,1563,\\
&1564,1565,1566,1567,1568,1569,1570,1571,1572,1573,1574,1575,1576,1577\}&34\\\hline
31&\{1345,1346,1347,1348,1349,1350,1351,1352,1353,1354,1355,1356,1357,1358,1359,1360,1361,1362,1363,1364,\\
&1365,1366,1367,1368,1369,1370,1371,1372,1373,1374,1375,1376,1377,1378,1379,1380,1381,1382,1383,1384,\\
&1385,1386,1387,1388,1389,1390,1391,1392,1393,1394,1395,1396,1397,1398,1399\}&55\\\hline
32&\{1578,1579,1580,1581,1582,1583,1584,1585,1586,1587,1588,1589,1590,1591,1592,1593,1594,1595,1596,1597,\\
&1598,1599,1600\}&
23\\\hline
33&\{2264,2265,2266,2267,2268,2269,2270,2271,2272,2273,2274,2275,2276,2277,2278,2279,2280,2281,2282,2283,\\
&2284,2285,2286,2287,2288,2289,2290,2291,2292,2293,2294,2295,2296,2297,2298,2299,2300\}&37\\\hline
34&\{2216,2217,2218,2219,2220,2221,2222,2223,2224,2225,2226,2227,2228,2229,2230,2231,2232,2233,2234,2235,\\
&2236,2237,2238,2239,2240,2241,2242,2243,2244,2245,2246,2247,2248,2249,2250,2251,2252,2253,2254,2255,\\
&2256\}&41\\\hline
35&\{2326,2327\}&2\\\hline
36&\{731,732,733,743,744,745,746,747,748,755,756,757,758,765,766,767,768,773,774,775,\\
&786,787,788,789,790,791,797,798,799,800,803,804,805,806,807,808,809,815,816,817,\\
&818,819,820,824,825,826,827,828,829,831,832,833\}&52\\\hline
37&\{242,243,244,245,246,247,265,266,267,268,269,287,288,289,290,291,292,293\}&18\\\hline
38&\{236,237,238,239,240,241,248,249,250,251,252,253,259,260,261,262,263,264,277,278,\\
&279,280,281,282,283,284,285,286\}&
28\\\hline
39&\{36,37,38,39,40,41\}&6\\\hline
40&\{1996,1997,1998,1999,2000,2001,2002,2003,2004,2005,2006,2007,2008,2009,2010,2038,2039,2040,2041,2042,\\
&2043,2044,2045,2046,2047,2048,2049,2050,2051,2052,2053,2054,2055,2056,2057,2058,2059,2060,2061,2062,\\
&2063,2064,2065,2075,2076,2077,2078,2079,2080,2081,2082,2083,2084,2085,2086,2087,2088,2089,2098,2099,\\
&2100,2101,2102,2103,2104,2105,2106,2107,2108,2109,2116,2117,2118,2119,2120,2121,2122,2129,2130,2131,\\
&2132,2133,2134,2135\}&84\\\hline
41&\{2011,2012,2013,2014,2015,2016,2017,2018,2019,2026,2027,2028,2029,2030,2031,2032,2033,2034,2035,2036,\\
&2037,2066,2067,2068,2069,2070,2071,2072,2073,2074,2090,2091,2092,2093,2094,2095,2096,2097,2110,2111,\\
&2112,2113,2114,2115,2123,2124,2125,2126,2127,2128\}&50\\\hline
42&\{1930,1931,1932,1933,1934,1935,1936,1952,1953,1954,1955,1956,1957,1973,1974,1975,1976,1989,1990,1991,\\
&1992,1993,1994,1995\}&24\\\hline
43&\{1924,1925,1926,1927,1928,1929,1945,1946,1947,1948,1949,1950,1951,1966,1967,1968,1969,1970,1971,1972,\\
&1983,1984,1985,1986,1987,1988\}&26\\\hline
44&\{1918,1919,1920,1921,1922,1923,1937,1938,1939,1940,1941,1942,1943,1944,1958,1959,1960,1961,1962,1963,\\
&1964,1965,1977,1978,1979,1980,1981,1982\}&28\\\hline
45&\{16,17,18,19,20,21,22,23,24,25\}&10\\\hline
\end{tabular}
}

\newpage
{\scriptsize 
\begin{tabular}{llr} 
$\Phi_j$&$G(2^7,i),i\in$&\#\\\hline
46&\{1740,1741,1742,1743,1744,1745,1773,1774,1775,1776,1777,1778,1813,1814,1815,1816,1828,1829,1830,1831,\\
&1850,1851,1852,1853,1854,1855,1856,1857,1858,1859\}&30\\\hline
47&\{1735,1736,1737,1738,1739,1768,1769,1770,1771,1772,1809,1810,1811,1812,1824,1825,1826,1827,1841,1842,\\
&1843,1844,1845,1846,1847,1848,1849\}&27\\\hline
48&\{200,201,202,203,204,205,630,631,632,633,696,697,698,699,728,729,730\}&17\\\hline
49&\{1789,1790,1791,1792,1793,1794,1795,1868,1869,1886,1887,1888,1895,1896,1905,1906,1907,1908,1909,1910,\\
&1911\}&21\\\hline
50&\{544,545,559,565,573,897,898\}&7\\\hline
51&\{1787,1788,1870,1871,1872,1873,1874,1883,1884,1885,1897,1898,1912,1913,1914,1915,1916,1917\}&18\\\hline
52&\{194,195,196,197,198,199,513,514,515,516,517,531,532,533,551,552,553,554\}&18\\\hline
53&\{2317,2318\}&2\\\hline
54&\{1615,1616,1620,1621\}&4\\\hline
55&\{387,388,389,390,391,392,393,394,395,396\}&10\\\hline
56&\{351,352,353,354,355,356,357,358,359,360,361,362,363,364,365,366,367,368,369,370,\\
&371,372,373,374,375,376,377,378,379,380,381,382,383,384,385,386\}&36\\\hline
57&\{327,328,329,330,331,332,333,334,335,336,337,338,339,340,341,342,343,344,345,346,347,348,349,350\}&24\\\hline
58&\{417,418,419,420,421,422,423,424,425,426,427,428,429,430,431,432,433,434,435,436\}&20\\\hline
59&\{397,398,399,400,401,402,403,404,405,406,407,408,409,410,411,412,413,414,415,416\}&20\\\hline
60&\{446,447,448,449,450,451,452,453,454,455\}&10\\\hline
61&\{437,438,439,440,441,442,443,444,445\}&9\\\hline
62&\{750,751,752,777,778,779,783,784,795,796,810,811,821,822
\}&14\\\hline
63&\{734,735,736,737,738,739,759,760,769,770,771,772,792,793
\}&14\\\hline
64&\{763,782,785,812,834,835\}&6\\\hline
65&\{753,754,761,762,776,794\}&6\\\hline
66&\{742,749,823,830\}&4\\\hline
67&\{2020,2021\}&2\\\hline
68&\{2022,2023,2024,2025\}&4\\\hline
69&\{1751,1752,1753,1754\}&4\\\hline
70&\{634,635,636,637\}&4\\\hline
71&\{645,646,647\}&3\\\hline
72&\{1758,1759,1760\}&3\\\hline
73&\{641,642,643,644\}&4\\\hline
74&\{1800,1801\}&2\\\hline
75&\{527,528\}&2\\\hline
76&\{560,561,562\}&3\\\hline
77&\{75,76,77,78,83,84,85,86\}&8\\\hline
78&\{79,80,81,82,92,93,94,95,96,97\}&10\\\hline
79&\{916,917,918,919,920,921,938,939,940,941,942,943,956,957,958,959,960,961,964,965,\\
&966,967,968,969\}&24\\\hline
80&\{950,951,952,975,976,977,982,983,987\}&9\\\hline
81&\{947,948,949,972,973,974,978,979,980,981,984,985,986\}&
13\\\hline
82&\{912,913\}&2\\\hline
83&\{853,854,855\}&3\\\hline
84&\{859,860,861,865,866,867\}&6\\\hline
85&\{2147,2148,2149\}&3\\\hline
86&\{780,781\}&2\\\hline
87&\{836\}&1\\\hline
88&\{801,802\}&2\\\hline
89&\{740,741\}&2\\\hline
90&\{813,814\}&2\\\hline
91&\{764\}&1\\\hline
92&\{934,935\}&2\\\hline
93&\{936,937\}&2\\\hline
94&\{931,932,933\}&3\\\hline
95&\{928,929,930\}&3\\\hline
96&\{71,72,73,74\}&4\\\hline
97&\{98\}&1\\\hline
\end{tabular}
}

\newpage
{\scriptsize 
\begin{tabular}{llr} 
$\Phi_j$ & $G(2^7,i), i\in $ & \#\\\hline
98&\{89,90,91\}&3\\\hline
99&\{87,88\}&2\\\hline
100&\{924,925,926,927\}&4\\\hline
101&\{922,923\}&2\\\hline
102&\{970,971\}&2\\\hline
103&\{944,945,946\}&3\\\hline
104&\{962,963\}&2\\\hline
105&\{953,954,955\}&3\\\hline
106&\{144,145\}&2\\\hline
107&\{140,141,142,143\}\hspace*{12cm}&4\\\hline
108&\{150,151,152\}&3\\\hline
109&\{995,996\}&2\\\hline
110&\{158\}&1\\\hline
111&\{134,135\}&2\\\hline
112&\{136,137\}&2\\\hline
113&\{161,162,163\}&3\\\hline
114&\{138,139\}&2\\\hline
115&\{146\}&1\\\hline
&&2328
\end{tabular}\vspace*{5mm}
}
Table $2$ 
Classification of groups of order $128$ into $115$ isoclinism families
\end{center}

\smallskip
\begin{remark}
For nine groups $G=G(2^6, i)$ which belong to $\Phi_{16}$ 
with $B_0(G)\neq 0$, we have 
\begin{align*}
G(2^6,149)\times C_2\simeq G(2^7,1783),\  
G(2^6,150)\times C_2\simeq G(2^7,1784),\ 
G(2^6,151)\times C_2\simeq G(2^7,1785),\\
G(2^6,170)\times C_2\simeq G(2^7,1864),\ 
G(2^6,171)\times C_2\simeq G(2^7,1865),\ 
G(2^6,172)\times C_2\simeq G(2^7,1866),\\
G(2^6,177)\times C_2\simeq G(2^7,1880),\ 
G(2^6,178)\times C_2\simeq G(2^7,1881),\ 
G(2^6,182)\times C_2\simeq G(2^7,1893). 
\end{align*}
\end{remark}



\vspace*{5mm}
\noindent
Akinari Hoshi\\
Department of Mathematics\\
Niigata University\\
8050 Ikarashi 2-no-cho\\
Nishi-ku, Niigata, 950-2181\\
Japan\\
E-mail: \texttt{hoshi}@\texttt{math.sc.niigata-u.ac.jp}\\
Web: \texttt{http://mathweb.sc.niigata-u.ac.jp/\~{}hoshi/}
\end{document}